\documentclass[12pt]{amsart}
\usepackage{SSdefn}
\usepackage{tikz}
\usetikzlibrary{positioning,shapes,arrows,mindmap,trees}

\newcommand{\uRep}{\ul{\smash{\mathrm{Rep}}}}

\DeclareMathOperator{\colim}{colim}
\newcommand{\pf}{\mathrm{pf}}
\newcommand{\FI}{\mathbf{FI}}
\newcommand{\FS}{\mathbf{FS}}
\newcommand{\FA}{\mathbf{FA}}
\newcommand{\VI}{\mathbf{VI}}
\newcommand{\VA}{\mathbf{VA}}
\newcommand{\VB}{\mathbf{VB}}
\newcommand{\VS}{\mathbf{VS}}
\newcommand{\FB}{\mathbf{FB}}
\newcommand{\OSp}{\mathbf{OSp}}

\tikzstyle{arrow} = [-,>=stealth]
\tikzset{node/.style={%
      draw,
      circle,
      fill,
      inner sep=0,
      outer sep=0,
      minimum size=2pt,
      node distance=30pt,
}}

\title[Brauer categories I: triangular categories]{The representation theory of Brauer categories I: triangular categories}
\date{April 12, 2022}

\author{Steven V Sam}
\address{Department of Mathematics, University of California, San Diego, CA}
\email{\href{mailto:ssam@ucsd.edu}{ssam@ucsd.edu}}
\urladdr{\url{http://math.ucsd.edu/~ssam/}}
\thanks{SS was supported by NSF grant DMS-1812462.}

\author{Andrew Snowden}
\address{Department of Mathematics, University of Michigan, Ann Arbor, MI}
\email{\href{mailto:asnowden@umich.edu}{asnowden@umich.edu}}
\urladdr{\url{http://www-personal.umich.edu/~asnowden/}}
\thanks{AS was supported by NSF grants DMS-1453893.}

\begin{document}

\begin{abstract}
This is the first in a series of papers in which we study representations of the Brauer category and its allies. We define a general notion of triangular category that abstracts key properties of the triangular decomposition of a semisimple complex Lie algebra, and develop a highest weight theory for them. We show that the Brauer category, the partition category, and a number of related diagram categories admit this structure.
\end{abstract}

\maketitle 
\tableofcontents

\section{Overview}

Brauer algebras were introduced in \cite{brauer} to extend the Schur--Weyl duality between symmetric groups and general linear groups to the orthogonal and symplectic groups. They are an archetypal example of {\it diagram algebras}: the $n$th Brauer algebra has a basis consisting of perfect matchings on a set of $2n$ vertices, which is separated into two subsets of size $n$. These diagrams can be viewed as ``functions'' from one set of size $n$ to another, and this perspective continues to make sense when the two sets have different sizes. This observation leads to the idea of the Brauer {\it category}, which has been used to give further insight into the invariant theory of classical groups \cite{LZ-SFT, LZ-brauer} and plays a key role in Deligne's interpolation categories \cite{deligne}. This paper is the first in a series in which we initiate a systematic study of the representation theory of this category and its many relatives. In the rest of this section, we give an overview of the series.

\begin{figure}[!h]
\scalebox{0.7}{
\begin{tikzpicture}
  \path[mindmap,concept color=olive!25!white]
    node[concept] {Brauer \\ category}
    [clockwise from=30]
    child[concept color=gray!15!white] { node[concept] {Brauer algebras} }
    child[concept color=gray!15!white] { node[concept] {Category $\cO$} }
    child[concept color=gray!15!white] { node[concept] {Super groups} }
    child[concept color=gray!15!white] { node[concept] {Deligne category} }
    child[concept color=gray!15!white] { node[concept] {TCA's} };    
\end{tikzpicture}
}
\caption{Concepts connected to the Brauer category.} \label{fig1}
\end{figure}

\subsection{Main results} \label{ss:main}

Let $\fG$ be the Brauer category over the complex numbers with parameter $\delta$ (see \S \ref{sec:brauer} for the definition). A {\bf representation} of $\fG$, or a {\bf $\fG$-module}, is a functor $\fG \to \Vec$, where $\Vec$ is the category of complex vector spaces. Concretely, a $\fG$-module is a sequence $(M_n)_{n \ge 0}$, where $M_n$ is a representation of the symmetric group $\fS_n$, together with ``upwards'' transition maps $M_n \to M_{n+2}$ and ``downwards'' transition maps $M_{n+2} \to M_n$ satisfying certain conditions; in fact, $M_n$ is a module over the Brauer algebra $B_n(\delta)$, but this extra structure can be recovered from the transition maps. The goal of this series of papers is to determine the structure of $\Mod_{\fG}$, the category of $\fG$-modules.

Before explaining our results, it is helpful to highlight two perspectives on $\fG$-modules:
\begin{itemize}
\item One can view $\fG$-modules from the point of view of representation theory. For instance, we show that $\Mod_{\fG}$ is (more or less) a highest weight category.
\item One can also view $\fG$-modules from the point of view of commutative algebra (in a general sense). Indeed, there is a tensor product $\otimes_{\fG}$ on $\fG$-modules that behaves like the tensor product for a commutative ring (e.g., it preserves finite generation, and is right exact but not exact). This allows one to consider notions such as ideals, prime ideals, annihilators, support, and so on.
\end{itemize}
The combination of these two structures is rather unusual, as categories exhibiting a highest weight structure are usually locally of finite length, and therefore quite different from the typical module categories seen in commutative algebra. It is this dual nature that gives the theory of $\fG$-modules much of its unique character.

The first three papers in this series focus on the representation-theoretic aspects of the Brauer category:
\begin{enumerate}[\indent 1.]
\item This paper develops the theory of {\it triangular categories}, which extracts the key properties of the triangular decomposition of a semisimple complex Lie algebra with respect to a parabolic subalgebra. We show that the category of representations of a triangular category behaves a lot like a highest weight category, and show that the Brauer category and its variants admit this structure.

\item As stated, a $\fG$-module is a sequence $(M_n)_{n \ge 0}$ of symmetric group representations---that is, a linear species---equipped with some transition maps. In the second paper \cite{brauercat2}, we show how the transition maps can be neatly packaged using the tensor product on linear species.

Let $\bV$ be the ``standard'' linear species: $\bV_1$ is one-dimensional and $\bV_n=0$ for $n \ne 1$. The upwards transition maps in $M$ can be encoded as a map of species $b \colon \Sym^2(\bV) \otimes M \to M$, while the downwards maps can be encoded as a map $c \colon M \to \Sym^2(\bV) \otimes M$. The transition maps satisfy a number of relations, and the main problem is to understand these relations in terms of $b$ and $c$.

The category of linear species does not have duals; however, if we pretend for the moment that it does, we can convert $c$ into a map $c^* \colon \Sym^2(\bV^*) \otimes M \to M$. We show that the relations satisfied by $b$ and $c$ are exactly the relations that would be needed for $b$ and $c^*$ to define a representation of the symplectic Lie algebra, if we regard $\Sym^2(\bV)$ and $\Sym^2(\bV^*)$ as upper and lower triangular nilpotent subalgebras. We therefore refer to $(b,c)$ as a representation of the ``curried symplectic algebra,'' since, in a way, it comes from currying the $\Sym^2(\bV^*)$ factor to the other side.

We establish a general theory of curried algebras, and show that many classical diagram categories have surprisingly concrete interpretations from this point of view.

\item Classical Schur--Weyl duality yields an equivalence between the category of linear species and the category of polynomial representations of the infinite general linear group, and this equivalence is compatible with tensor products. In the third paper \cite{brauercat3}, we transport the description of $\fG$-modules from the second paper across Schur--Weyl duality. We find that $\Mod_{\fG}$ is equivalent to a version of parabolic category $\cO$ for the infinite rank symplectic Lie algebra. We then translate known results about this category (such as the block structure) to obtain information about $\fG$-modules.
\end{enumerate}
The final two papers investigate the commutative algebra aspects of the Brauer category (but are heavily representation-theoretic as well):
\begin{enumerate}[\indent 1.] \setcounter{enumi}{3}

\item The fourth paper \cite{brauercat4} is a study of torsion modules (i.e., modules that locally have non-zero annihilator). We show that the torsion category admits a filtration by Serre subcategories (characterized by annihilators) such that the successive quotients are equivalent to $\Rep(\OSp(r \vert 2s))$ as $r$ and $s$ vary over non-negative integers with $r-2s=\delta$.

\item We define the category of ``generic'' $\fG$-modules, denoted $\Mod_{\fG}^{\rm gen}$, to be the Serre quotient of the category of all $\fG$-modules by the subcategory of torsion modules. This is a sort of ``fraction field'' of the Brauer category. The fifth paper \cite{brauercat5} analyzes this category. The main result is that $\Mod_{\fG}^{\rm gen}$ is the abelian envelope of the Deligne interpolation category for the orthogonal group.

\end{enumerate}
The above is a high-level overview of our main results. We have numerous specific results (e.g., classification of injectives, description of the Grothendieck group, etc.) that we do not attempt to describe here.

\subsection{Other diagram categories}

There are several variations of the Brauer category that could just as easily have been emphasized in our work. The methods that we develop apply equally well to them, but for the most part, we will mention the relevant similarities and differences without giving complete proofs. We have chosen to focus on the Brauer category as our main example due to its familiarity and the simplicity of its definition (for example, some of the other categories have complicated sign conventions).

However, we wish to emphasize that the uniformity and broad applicability of the results is one of the attractive features of our theory. For example, as stated above, the Brauer category is a sort of ``curried'' symplectic Lie algebra. This fits into a natural class of examples that includes many familiar Lie algebras, as well as many other well-studied algebras, such as Weyl and Clifford algebras. In fact, in trying to complete this picture, one is naturally led to examples of categories or algebras that have received little or no attention, but nonetheless have a rich representation theory.

Furthermore, while some of our results about the Brauer category and Brauer algebras have been proven by different methods or using different language, one of our goals for developing a uniform framework is to make the analogues of these results for related examples transparent and immediate, and to minimize the need to discover them on a case-by-case basis.

\subsection{Motivation}

Our investigation began in \cite{infrank} where we found models for the representation theory of the orthogonal $\bO_n(\bC)$ and symplectic groups $\Sp_{2n}(\bC)$ in the limit $n \to \infty$. One of the models is the category of representations of the {\it upwards} Brauer category; this is the subcategory of the Brauer category consisting only of upwards transition maps. The present work is a natural outgrowth of these investigations.

A second source of motivation comes from the general philosophy of representation stability: given a sequence of groups or algebras that naturally assemble into a category, it is often fruitful to study representations of the category as a whole. This idea has met with success for symmetric groups (leading to the theory of $\FI$-modules \cite{fimodule,fi-noeth} and twisted commutative algebras \cite{expos}), and general linear groups (leading to $\VI$- and $\mathbf{VIC}$-modules \cite{putman-sam}), to name two examples. We aim to show in these papers that this perspective is also compelling when applied to the Brauer algebras.

\subsection{Relation to previous work}

We now discuss how the work in this series relates to previous work, at a general level. We will discuss more specific connections in each paper.

\subsubsection{Diagram algebras}

The literature on diagram algebras is vast; we mention only a few of the many relevant papers here:
\begin{itemize}
\item The papers \cite{CDM,CDDM,martin3,cox} determine the block structure and Cartan matrix of the Brauer algebra, and connects it to parabolic category $\cO$ in type D; \cite{ES} shows that the module category for the Brauer algebra is in fact equivalent to a certain piece of parabolic category $\cO$ in type D.
\item The series \cite{stroppel1,stroppel2,stroppel3,stroppel4} studies Khovanov's diagram algebra, which are in some ways similar to Brauer algebras. Connections to category $\cO$ and the super general linear group are established.
\item \cite{coulembier4} shows that there is a Ringel duality between the module category for the Brauer algebra and a subcategory of the representation category of the orthosymplectic Lie group.
\item The Brauer algebra is cellular \cite{graham}, and usually (but not always) quasi-hereditary \cite[Theorem~1.3]{konigxi} (there are problems when $\delta=0$).
\item The Brauer algebra is semi-simple at non-integer parameter \cite{wenzl} (see also \cite{doran,rui}).
\end{itemize}
The above results have a similar flavor to many of our results. However, there are significant differences in some cases: for example, the above papers relate the Brauer algebra to category $\cO$ in type D, while we relate the Brauer category to category $\cO$ in type C; also, we show that the Brauer category always has a highest weight structure, while the Brauer algebra can fail to admit such structure when $\delta=0$. On the other hand, we will show in \cite{brauercat3} that the Brauer category is semi-simple at non-integer parameter, and use this to reprove Wenzl's theorem. The upshot is that the connection between the Brauer category and Brauer algebra can be quite subtle: in some cases results flow from one to the other, while in others one finds superficially similar results that are actually quite different.

In addition to the above papers, related notions of a triangular decomposition of a category have appeared before in \cite{bellamy-thiel}, \cite{CZ}, \cite{KM}, and \cite{konig}.

\subsubsection{Deligne categories}

In \cite{deligne}, Deligne introduced the category $\uRep(\bO_{\delta})$, for arbitrary parameter $\delta$, that (in a sense) interpolates the categories $\Rep(\bO_n)$ for $n \in \bN$; he also defined variants in other cases, including the general linear group and the symmetric group. In the years since, these categories have received a great deal of attention. From their inception, these categories have been intimately related to the Brauer category: indeed, Deligne defined $\uRep(\bO_{\delta})$ by simply taking the additive and Karoubian envelope of the Brauer category.

Deligne's categories are typically not abelian, and a fundamental problem is to construct and understand their abelian envelopes. This has been carried out for $\uRep(\GL_{\delta})$ in \cite{entova} using a construction involving the super general linear group, and recently for $\uRep(\mathbf{Pe}_{\delta})$ in \cite{entova2} by similar means.

As stated above, we show that the abelian envelope of $\uRep(\bO_{\delta})$ can be realized as the generic category $\Mod_{\fG}^{\rm gen}$ of the category of $\fG$-modules. We believe this construction of the abelian envelope has some advantages over previous ones in that it is extremely simple (being purely in terms of the Brauer category, and not involving auxiliary concepts like super groups) and totally uniform (the same construction applies in each case, including the case of the symmetric group). We remark that while constructing the category $\Mod_{\fG}^{\rm gen}$ is very simple, proving that it is the abelian envelope takes work, and our proof does involve super groups. It also reveals a deeper connection between Deligne's category and the Brauer category.

\subsubsection{Twisted commutative algebras}

A {\bf twisted commutative algebra (tca)} is a commutative algebra object in the category of linear species (see \cite{expos} for a general overview). In characteristic~0, one can apply Schur--Weyl duality to view a tca as an ordinary commutative algebra equipped with an action of the infinite general linear group, under which it forms a polynomial representation. Many of the upwards categories in Brauer-like categories are equivalent to tca's, or closely related objects; for example, modules over the upwards Brauer category are equivalent to modules over the tca $\Sym(\Sym^2(\bC^{\infty}))$ (see  \cite[Remark~1.3]{sym2noeth}). We can therefore apply the many results proven about tca's in recent years to the study of the Brauer category. This will be a recurrent theme in this series: for example, in this paper we use \cite{sym2noeth} to show that the Brauer category is noetherian, while in \cite{brauercat3}, we will see that ideals of the Brauer category are closely related to the equivariant prime ideals of $\Sym(\Sym^2(\bC^{\infty}))$ as studied in \cite{tcaspec}.

\subsection{Open problems}

We will discuss a number of specific open problems throughout this series. Here, we highlight three of a broad scope:
\begin{itemize}
\item In this series of papers, the Brauer algebras play a surprisingly minimal role. It would be interesting to connect our results on the Brauer category back to the Brauer algebras. For example, we show that the Brauer category is related to parabolic category $\cO$ of type C, while it is known that the Brauer algebras are related to parabolic category $\cO$ of type D; what, if any, is the relationship between these connections?
\item There are a number of diagram categories (such as the Temperley--Lieb category and its variants, or BMW categories) that we ignore, or discuss only briefly. A natural problem is to extend our results to these categories.
\item A $d$-dimensional topological quantum field theory is a monoidal functor from the $d$-dimensional cobordism category to $\Vec$. When $d=1$, the cobordism category is closely related to the (walled) Brauer category (see \cite[\S 3.1]{carqueville}) and when $d=2$, it is closely related to the partition category (see \cite[\S 2.2]{jellyfish}). Thus this series can be viewed as a study of ``non-monoidal'' low-dimensional TQFT's. A natural problem is to investigate non-monoidal TQFT's in higher dimension.
\end{itemize}

\section{Introduction}

We now discuss the contents of this paper in more detail.

\subsection{Triangular categories}

The primary theoretical focus of this paper is the notion of \emph{triangular category}. A triangular category is a category equipped with notions of ``upwards'' and ``downwards'' morphisms satisfying some axioms (see Definition~\ref{defn:tri-struct}). The definition is modeled on the triangular decomposition $\fg = \fn_- \oplus \fh \oplus \fn_+$ of a semisimple complex Lie algebra $\fg$, with the strictly upwards (resp.\ downwards) morphisms playing the role of $\fn_+$ (resp.\ $\fn_-$); the morphisms that are simultaneously upwards and downwards fulfill the role of the Cartan subalgebra $\fh$.

Suppose $\fG$ is a triangular category. We show that the category $\Mod_{\fG}$ is essentially a highest weight category, in the following sense. We associate to $\fG$ a set $\Lambda$ of ``weights,'' which is partially ordered. Given a weight $\lambda$, we define a standard module $\Delta_{\lambda}$, we construct a simple quotient $L_{\lambda}$ of $\Delta_{\lambda}$, and we show that $L_{\lambda}$ has a projective cover $P_{\lambda}$ that admits a filtration by standard objects. There are also co-standard modules $\nabla_{\lambda}$ and injective envelopes $I_{\lambda}$ that behave similarly. The proofs apply more or less standard arguments that we adapt to the setting of $\fG$-modules.

We said that $\Mod_{\fG}$ is ``essentially'' a highest weight category since highest weight categories are typically required to be locally artinian, but $\Mod_{\fG}$ (in most cases of interest) is not: for instance, the standard objects are typically not of finite length. Thus one must take care when attempting to apply familiar results from highest weight categories.

\subsection{The Brauer category}

The main interesting content of the article is the wealth of natural examples that exist. As indicated in the title, the Brauer category $\fG$ remains our primary motivation and we treat it in detail (we treat a number of other examples in less detail). We summarize our analysis here. To begin, we show:

\begin{theorem}
The Brauer category $\fG$ is naturally a triangular category.
\end{theorem}

The ``upwards'' (resp.\ ``downwards'') morphisms in the triangular structure come from those Brauer diagrams that have no horizontal edge in the bottom (resp.\ top) row. The general theory of triangular categories can now be applied to $\fG$-modules. We record the most important consequences in the following theorem.

\begin{theorem}
Let $\Lambda$ denote the set of all integer partitions.
\begin{enumerate}
\item The simple $\fG$-modules are naturally indexed by $\Lambda$. Precisely, if $\lambda$ is a partition of $n$ then there is a unique (up to isomorphism) simple $\fG$-module $L_{\lambda}$ such that $L_{\lambda}([n])$ is isomorphic (as an $\fS_n$-representation) to the Specht module corresponding to $\lambda$, and $L_{\lambda}([m])=0$ for $m<n$.
\item The simple module $L_{\lambda}$ admits a projective cover $P_{\lambda}$. The $P_{\lambda}$ account for all of the finitely generated indecomposable projective $\fG$-modules.
\item For each $\lambda \in \Lambda$ there is a standard $\fG$-module $\Delta_{\lambda}$ such that:
\begin{enumerate}
\item The simple constituents of $\Delta_{\lambda}$ have the form $L_{\mu}$ with $\vert \mu \vert \ge \vert \lambda \vert$; moreover, $L_{\lambda}$ occurs with multiplicity one (as a quotient).
\item $P_{\lambda}$ admits a finite filtration $0=F_0 \subset \cdots \subset F_r=P_{\lambda}$ such that $F_i/F_{i-1} \cong \Delta_{\mu(i)}$ for some $\mu(i) \in \Lambda$; moreover, $\mu(r)=\lambda$ and $\vert \mu(i) \vert<\vert \lambda \vert$ for $i<r$.
\end{enumerate}
\end{enumerate}
\end{theorem}

This theorem gives a coarse picture of the representation theory of $\fG$ that serves as a foundation for the subsequent papers in this series. We next establish the following result:

\begin{theorem}
The category $\Mod_{\fG}$ is locally noetherian. In other words, any submodule of a finitely generated $\fG$-module is again finitely generated.
\end{theorem}

This theorem is of fundamental importance, and will be used constantly in the following papers. Indeed, finitely generated $\fG$-modules typically do not have finite length, so the noetherian result is important for ensuring that various constructions preserve finite generation. The theorem follows rather easily from our previous results \cite{sym2noeth} on noetherianity of some twisted commutative algebras. In principle, the noetherianity result for $\fG$-modules should be much easier than that for the tca, but this is the only proof we know.

The Brauer category has two other notable pieces of structure. First, it is naturally equivalent to its opposite category. From this, we obtain a duality functor $(-)^{\vee}$ on $\Mod_{\fG}$. This functor should be thought of as an analog of Pontryagin duality: for instance, it interchanges (principal) projectives and injectives, but does not preserve finite generation. In \cite{brauercat5} we will encounter a more subtle duality that does preserve finite generation.

Second, $\fG$ has a natural monoidal structure, given by taking the disjoint union of Brauer diagrams. This induces a tensor product $\otimes_{\fG}$ on $\Mod_{\fG}$ known as Day convolution. As discussed in \S \ref{ss:main}, this enables us to import various notions from commutative algebra, such as ``ideal,'' to the theory of $\fG$-modules. This will play a prominent role in the later papers in this series. In the present paper, we prove only one non-trivial proposition about the tensor product: the tensor product of standard modules admits a standard filtration, and the higher Tor's vanish.

Finally, we remark that when the parameter $\delta$ is an integer there is an important family of $\fG$-modules $T_{p|q}$ that we call the \emph{tautological modules}. The existence of these modules is seemingly unrelated to the triangular structure, but closely connected to the tensor product: indeed, these are exactly the $\fG$-modules that are symmetric monoidal functors (this is closely related to the classification of 1-dimensional TQFT's, see \cite[Theorem~3.1]{carqueville}). In the context of \S \ref{ss:main}, the tautological modules are intimately related to the commutative algebra side of the picture, and provide the bridge to $\Rep(\OSp(p|q))$.

\subsection{Relation to previous work}

The idea of abstracting properties of the triangular decomposition of a semisimple Lie algebra is not new; for example, see \cite[\S 2]{ggor}. 

As discussed, the module category of a triangular category is closely related to the notion of a highest weight category, first introduced in \cite{CPS} (see also the closely related notion of BGG algebra in \cite{irving}).

After completing this paper, we learned of the recent paper \cite{stroppel5} that develops highest weight theory in more general settings. Many of the general ideas here are present there in some form (for example, \cite[Definition~5.24]{stroppel5} is essentially the same as a triangular category with trivial endomorphism rings). However, the language and emphasis of the two papers are quite different.

\subsection{Outline}

In \S\ref{sec:rep-cat}, we develop some basic properties of representations of categories, including change of category and tensor products. In \S\ref{sec:tri-cat}, we introduce triangular categories and develop their theory. In \S\ref{sec:brauer}, we explain the example of the Brauer category in detail. In \S \ref{sec:partition}, we treat the partition category; it is similar to the Brauer category in some ways, but different enough to receive special attention. In \S \ref{sec:other}, we list a number of other prominent examples of triangular categories with little detail. Finally, in \S\ref{sec:catO} we explain how to realize the triangular decomposition of a semisimple Lie algebra as a triangular category, and also discuss an extension to positive characteristic representations.

\section{Representations of categories} \label{sec:rep-cat}

\subsection{General definitions}

Fix a field $\bk$. Let $\fC$ be an essentially small $\bk$-linear category. A {\bf $\fC$-module} is a $\bk$-linear functor $\fC \to \Vec$, where $\Vec$ denotes the category of $\bk$-vector spaces. For an object $x$ and a morphism $\alpha \colon x \to y$, we write $\alpha_* \colon M(x) \to M(y)$ for the linear map $M(\alpha)$. A {\bf morphism} of $\fC$-modules is a natural transformation of functors. For $\fC$-modules $M$ and $N$, we write $\Hom_{\fC}(M, N)$ for the space of morphisms $M \to N$ of $\fC$-modules. We also write $\Hom_{\fC}(x,y)$ for the space of morphisms $x \to y$ for $x,y \in \fC$; this should not cause confusion, since we always use lowercase letters for objects of $\fC$ and uppercase letters for $\fC$-modules. We let $\Mod_{\fC}$ denote the category of $\fC$-modules. It is a Grothendieck abelian category.

\subsection{Finiteness conditions}

We now introduce a number of finiteness conditions on $\fC$-modules. Let $M$ be a $\fC$-module.
\begin{itemize}
\item Given a collection $S$ of elements in various $M(x)$'s, the submodule of $M$ generated by $S$ is the smallest submodule containing each element of $S$. We say that $M$ is {\bf finitely generated} if it is generated by a finite collection of elements. We write $\Mod_{\fC}^{\rm fg}$ for the category of finitely generated modules. It need not be an abelian subcategory.
\item We say that a $\fC$-module $M$ is {\bf pointwise finite} if $M(x)$ is a finite dimensional vector space for all $x \in \fC$. It is not difficult to see that if $M$ is finitely generated and all $\Hom$ spaces in $\fC$ are finite dimensional then $M$ is pointwise finite. (This will become more clear in \S  \ref{ss:principal}.) We write $\Mod_{\fC}^{\pf}$ for the category of pointwise finite modules. It is an abelian subcategory of $\Mod_{\fC}$.
\item We say that $M$ is {\bf noetherian} if every submodule of $M$ is finitely generated. We say that $\fC$ is {\bf noetherian} if every finitely generated $\fC$-module is noetherian. In this case, $\Mod_{\fC}^{\rm fg}$ is an abelian subcategory of $\Mod_{\fC}$.
\end{itemize}

\subsection{Duality} \label{ss:duality}

Let $\fC^{\op}$ be the opposite category of $\fC$. Let $M$ be a $\fC$-module. We define a $\fC^{\op}$-module $M^{\vee}$ by $M^{\vee}(x)=M(x)^*$, where $(-)^*$ denotes the dual vector space. This construction defines an exact functor $\Mod_{\fC}^{\op} \to \Mod_{\fC^{\op}}$. There is a canonical morphism $M \to (M^{\vee})^{\vee}$ that is an isomorphism if $M$ is pointwise finite. It follows that duality induces an equivalence of categories $(\Mod_{\fC}^{\pf})^{\op} \cong \Mod_{\fC^{\op}}^{\pf}$.

\begin{proposition} \label{prop:dual-adjoint}
Let $M$ be a $\fC$-module and let $N$ be a $\fC^{\op}$-module. Then we have a canonical isomorphism
\begin{displaymath}
\Hom_{\fC}(M, N^{\vee}) = \Hom_{\fC^{\op}}(N, M^{\vee})
\end{displaymath}
This holds even if $M$ and $N$ are not pointwise finite.
\end{proposition}

\begin{proof}
Given a morphism of $\fC$-modules $f \colon M \to N^\vee$, we can take its pointwise dual to get a morphism of $\fC^\op$-modules $f^\vee \colon (N^\vee)^\vee \to M^\vee$ which can be precomposed with the natural map $N \to (N^\vee)^\vee$. Similarly, we can define a map in the other direction. A straightforward check shows that they are inverse to one another.
\end{proof}

\subsection{Projectives and injectives} \label{ss:principal}

For an object $x$ of $\fC$, we define a $\fC$-module $\bP_x=\bP_{\fC,x}$, called the {\bf principal projective} at $x$, by $\bP_x(y)=\Hom_{\fC}(x, y)$. We also define a $\fC$-module $\bI_x=\bI_{\fC,x}$, called the {\bf principal injective} at $x$, by $\bI_x(y)=\Hom_{\fC}(y, x)^*$. Note that $\bI_{\fC,x}=\bP_{\fC^{\op},x}^{\vee}$.

\begin{proposition}\label{prop:yoneda}
Let $M$ be a $\fC$-module. Then $\Hom_{\fC}(\bP_x, M)=M(x)$ and $\Hom_{\fC}(M, \bI_x)=M(x)^*$. In particular, $\bP_x$ is a projective $\fC$-module and $\bI_x$ is an injective $\fC$-module.
\end{proposition}

\begin{proof}
The identification of $\Hom_{\fC}(\bP_x, M)$ with $M(x)$ is Yoneda's lemma: note that $\bP_x$ is the functor represented by $x$. We have
\begin{displaymath}
\Hom_{\fC}(M, \bI_x) = \Hom_{\fC}(M, \bP_{\fC^{\op},x}^{\vee}) = \Hom_{\fC^{\op}}(\bP_{\fC^{\op},x}, M^{\vee}) = M^{\vee}(x)=M(x)^*,
\end{displaymath}
where in the first step we used the identification $\bI_x=\bP_{\fC^{\op},x}^{\vee}$ stated before the proposition, in the second step we used Proposition~\ref{prop:dual-adjoint}, and in the third step we used the mapping property for $\bP_{\fC^{\op},x}$ just established. Since $M \mapsto M(x)$ and $M \mapsto M(x)^*$ are exact functors of $M$, we see that $\bP_x$ is projective and $\bI_x$ is injective.
\end{proof}

As a corollary, we see that a $\fC$-module is finitely generated if and only if it is isomorphic to a quotient of a finite direct sum of principal projectives. We also see that $\Mod_{\fC}$ has enough injectives and projectives. The proposition shows that the duals of principal projectives are injectives. This holds more generally:

\begin{proposition} \label{prop:proj-dual}
Let $P$ be a projective $\fC$-module. Then $P^{\vee}$ is an injective $\fC^{\op}$-module.
\end{proposition}

\begin{proof}
Let $M \to N$ be an injection of $\fC^{\op}$-modules. Consider the diagram
\begin{displaymath}
\xymatrix{
\Hom_{\fC^{\op}}(N, P^{\vee}) \ar[r] \ar@{=}[d] & \Hom_{\fC^{\op}}(M, P^{\vee}) \ar@{=}[d] \\
\Hom_{\fC}(P, N^{\vee}) \ar[r] & \Hom_{\fC}(P, M^{\vee}) }
\end{displaymath}
where the vertical identifications come from Proposition~\ref{prop:dual-adjoint}. One readily verifies that the diagram commutes. Since $N^{\vee} \to M^{\vee}$ is surjective and $P$ is projective, the bottom arrow is a surjection. Thus the top arrow is a surjection as well, and so $P^{\vee}$ is injective.
\end{proof}

\begin{remark} \label{rmk:dual-proj}
The analog of this proposition for injective modules is not true in general. Here is a simple example. Let $\fC$ be the linearization of the category associated to the poset $(\bN, \le)$. Let $I$ be the principal projective $\bP_0$; this takes all objects of $\fC$ to $\bk$ and all morphisms to the identity. In fact, $I$ is an injective $\fC$-module. Indeed, if $M$ is a $\fC$-module then giving a map $M \to I$ is equivalent to giving maps $M_n \to \bk$ for all $n$ compatible with transition maps; but this is exactly a map $\colim M \to \bk$. We thus see that $\Hom_{\fC}(M, I) = (\colim M)^*$, which is an exact functor of $M$. A similar analysis shows that for a $\fC^{\op}$-module $N$ we have $\Hom_{\fC^{\op}}(I^{\vee}, N)=\lim N$, which is not an exact functor of $N$. Thus $I$ is injective but $I^{\vee}$ is not projective.

This example demonstrates one additional phenomenon. Duality provides an equivalence of $\Mod_{\fC}^{\pf}$ with the opposite category of $\Mod_{\fC^{\op}}^{\pf}$. Since $I$ is injective in $\Mod_{\fC}^{\pf}$, it follows that $I^{\vee}$ \emph{is} projective in $\Mod_{\fC^{\op}}^{\pf}$. (This can be seen directly, as if $N$ is a pointwise finite $\fC^{\op}$-module then it obviously satisfies the Mittag--Leffler condition and so $\rR^1 \lim N=0$.) This shows that projective objects of $\Mod_{\fD}^{\pf}$ need not remain projective in $\Mod_{\fD}$, even for very nice categories $\fD$. If $\fD$ has finite $\Hom$ spaces then injective modules objects of $\Mod_{\fD}^{\pf}$ do remain injective in $\Mod_{\fD}$, by a version of Baer's criterion.
\end{remark}

\begin{proposition}
Let $M$ be a $\fC$-module and let $N$ be a $\fC^{\op}$-module. Then we have a natural isomorphism $\Ext^i_{\fC}(M, N^{\vee}) = \Ext^i_{\fC^{\op}}(N, M^{\vee})$.
\end{proposition}

\begin{proof}
Let $P_{\bullet} \to M$ be a projective resolution of $M$ as a $\fC$-module. Since $(-)^{\vee}$ is exact, we see from the previous proposition that $M^{\vee} \to P_{\bullet}^{\vee}$ is an injective resolution of $M^{\vee}$ as a $\fC^{\op}$-module. Thus $\Ext^{\bullet}_{\fC}(M, N^{\vee})$ is computed by the complex $\Hom_{\fC}(P_{\bullet}, N^{\vee})$, while $\Ext^{\bullet}_{\fC^{\op}}(N, M^{\vee})$ is computed by the complex $\Hom_{\fC^{\op}}(N, P_{\bullet}^{\vee})$. These complexes are isomorphic by Proposition~\ref{prop:dual-adjoint}.
\end{proof}

\subsection{Tensoring over $\fC$} \label{ss:tensorC}

Let $M$ be a $\fC$-module and let $N$ be a $\fC^{\op}$-module. We define a $\bk$-vector space $N \odot_{\fC} M$ by a mapping property, as follows. To give a linear map $f$ from $N \odot_{\fC} M$ to a vector space $V$ is the same as giving linear maps $f_x \colon N(x) \otimes M(x) \to V$ for all $x \in \fC$ such that for $m \in M(x)$, $n \in N(y)$, and a morphism $\alpha \colon x \to y$, we have
\[
  f_y(n \otimes \alpha_* m)=f_x(\alpha_* n \otimes m).
\]
Note that $\alpha$ defines a morphism $y \to x$ in $\fC^{\op}$, and thus induces a linear map $\alpha_* \colon N(y) \to N(x)$. One can construct $N \odot_{\fC} M$ as a quotient of $\bigoplus_{x \in \fC} N(x) \otimes M(x)$ by appropriate relations. (Actually, one should use a skeletal subcategory of $\fC$ so that the direct sum is small.) One readily verifies that $\odot_{\fC}$ is cocontinuous in each variable.

\begin{proposition} \label{prop:tensor-hom-dual}
Let $M$ and $N$ be as above. Then there is a natural map
\begin{displaymath}
(N \odot_{\fC} M)^* \to \Hom_{\fC}(M, N^{\vee})
\end{displaymath}
that is an isomorphism if $N$ is pointwise finite.
\end{proposition}

\begin{proof}
By the mapping property, giving an element of $(N \odot_{\fC} M)^*$ is equivalent to giving maps $f_x \colon N(x) \otimes M(x) \to \bk$ for all $x$, satisfying certain relations. The map $f_x$ determines a map $g_x \colon M(x) \to N(x)^*$. One easily verifies that the relations on the $f$'s translate to the $g$'s defining a morphism of $\fC$-modules. If $N(x)$ is finite dimensional for all $x$, then $(M(x) \otimes N(x))^*$ is identified with $\Hom(M(x), N(x)^*)$, and the construction is reversible.
\end{proof}

\begin{proposition} \label{prop:tensor-eval}
Let $M$ be a $\fC$-module and let $x \in \fC$. Then we have a natural isomorphism $\bP_{\fC^{\op},x} \odot_{\fC} M = M(x)$. Similarly, if $N$ is a $\fC^\op$-module, then we have a natural isomorphism $N \odot_\fC \bP_{\fC,x} = N(x)$.
\end{proposition}

\begin{proof}
  By the tensor product relations, giving a map $f \colon \bP_{\fC^\op,x} \odot_\fC M \to V$ is the same as giving maps $f_y \colon \bP_{\fC^\op,x}(y) \otimes M(y) \to V$ such that for all $\alpha \colon y \to z$, $m \in M(x)$, and $n \colon y \to x$, we have
  \[
    f_y(n \circ \alpha, m) = f_z( n \otimes \alpha_*(m)).
  \]
  In particular, we can define $f' \colon M(x) \to V$ by $f'(m) = f_x(1_x \otimes m)$ and this captures all of the above data. So $M(x)$ satisfies the same universal property as $\bP_{\fC^\op,x} \odot_\fC M$, which gives the desired identification. The proof of the other isomorphism is similar.
\end{proof}

\subsection{Pushforward and pullback}

Let $f \colon \fC \to \fD$ be a $\bk$-linear functor. If $M$ is a $\fD$-module then $f^*(M)=M \circ f$ is a $\fC$-module. This construction defines a functor $f^* \colon \Mod_{\fD} \to \Mod_{\fC}$. It is clear that $f^*$ preserves pointwise finiteness. Since limits and colimits in these module categories are computed pointwise, it follows that $f^*$ is both continuous and cocontinuous. It therefore has a left adjoint $f_!$ and a right adjoint $f_*$. These functors are sometimes called the left and right {\bf Kan extensions} of $f$.

\begin{proposition} \label{prop:kan-proj}
The functor $f_!$ takes projectives to projectives, while $f_*$ takes injectives to injectives. Moreover, for $x \in \fC$ we have natural isomorphisms
\begin{displaymath}
f_!(\bP_{\fC,x}) = \bP_{\fD,f(x)}, \qquad
f_*(\bI_{\fC,x}) = \bI_{\fD,f(x)}.
\end{displaymath}
\end{proposition}

\begin{proof}
Since $f_!$ is left adjoint to an exact functor, it takes projectives to projectives. Given a $\fD$-module $M$, we have
  \[
    \hom_\fD(f_! \bP_{\fC,x}, M) = \hom_\fC(\bP_{\fC,x}, f^*M) = M(f(x))
  \]
  so that $f_! \bP_{\fC,x}$ represents the same functor as $\bP_{\fD,f(x)}$. The arguments for $f_*$ are similar.
\end{proof}

\begin{proposition} \label{prop:!-finite}
The functor $f_!$ takes finitely generated $\fC$-modules to finitely generated $\fD$-modules.
\end{proposition}

\begin{proof}
Indeed, $f_!$ is right-exact and takes principal projectives to principal projectives.
\end{proof}

\begin{proposition} \label{prop:tensor-adjunction}
Let $M$ be a $\fC$-module and let $N$ be a $\fD^{\op}$-module. Then we have a natural identification
\begin{displaymath}
(f^{\op})^*N \odot_{\fC} M = N \odot_{\fD} f_!M.
\end{displaymath}
\end{proposition}

\begin{proof}
First, we have a natural map $(f^\op)^* N \odot_\fC f^* f_!M \to N \odot_\fD f_! M$ by the universal property of tensor products. We precompose this with the adjunction map $M  \to f^* f_! M$ tensored with the identity on $(f^\op)^* N$ to get $\rho_M \colon (f^{\op})^*N \odot_{\fC} M \to N \odot_{\fD} f_!M$. We will prove that this is an isomorphism for all $M$.
  
First suppose that $M=\bP_{\fC,x}$ is a principal projective. The composite defined above evaluated at $y \in \fC$ is
\[
 ((f^\op)^* N)(y) \otimes \bP_{\fC,x}(y) \to N(f(y)) \otimes f_! \bP_{\fC,x}(f(y))
\]
which can be rewritten as
\[
\Phi_y \colon  N(f(y)) \otimes \hom_\fC(x,y) \to N(f(y)) \otimes \hom_\fD(f(x),f(y)), \qquad n \otimes \alpha \mapsto n \otimes f(\alpha).
  \]
Next, we have
  \[
    N \odot_\fD f_! \bP_{\fC,x} = N \odot_\fD \bP_{\fD,f(x)} = N(f(x)) = (f^\op)^* N \odot_\fC \bP_{\fC,x}.
  \]
  where the first equality is Proposition~\ref{prop:kan-proj} and the last two equalities are Proposition~\ref{prop:tensor-eval}. The elements $N(f(x)) \otimes 1_x$ in the domain and target of $\Phi_x$ are identified with $N(f(x))$ under the previous isomorphism, and hence we see from the form of $\Phi_x$ that $\rho_{\bP_{\fC,x}}$ is an isomorphism. To see that $\rho_M$ is an isomorphism in general, we can pick a projective presentation for $M$ and use a diagram chase.
\end{proof}

\begin{proposition} \label{prop:pushfwd-formula}
Let $M$ be a $\fC$-module, and let $y \in \fD$. Then we have natural isomorphisms
\begin{displaymath}
(f_!M)(y) = (f^{\op})^*\bP_{\fD^{\op},y} \odot_{\fC} M, \qquad
(f_*M)(y) = \Hom_{\fC}(f^*\bP_{\fD,y}, M).
\end{displaymath}
\end{proposition}

\begin{proof}
  By Propositions~\ref{prop:tensor-eval} and~\ref{prop:tensor-adjunction},
  \[
    (f^\op)^* \bP_{\fD^\op, y} \odot_\fC M = \bP_{\fD^\op,y} \odot_\fD f_! M = (f_! M)(y),
  \]
  which proves the first identity. The second follows from adjunction and Proposition~\ref{prop:yoneda}.
\end{proof}

\begin{proposition} \label{prop:!-exact}
  The following are equivalent:
\begin{enumerate}
\item $f_*$ is exact.
\item $f^*(\bP_{\fD,y})$ is a projective $\fC$-module for all $y \in \fD$.
\item $f^*$ takes projective $\fD$-modules to projective $\fC$-modules.
\end{enumerate}
Exactness of $f_!$ is similarly related to $(f^{\rm op})^*$ preserving projectives, or $f^*$ preserving injectives.
\end{proposition}

\begin{proof}
If $f_*$ is exact then $f^*$, being its left adjoint, takes projectives to projectives. Thus (a) implies (c), which obviously implies (b). From Proposition~\ref{prop:pushfwd-formula}, we have $(f_*M)(y)=\Hom_{\fC}(f^*\bP_{\fD,y},M)$ for a $\fC$-module $M$. Thus, if $f^*\bP_{\fD,y}$ is projective then $M \mapsto (f_*M)(y)$ is an exact functor of $M$. We therefore see that (b) implies (a). The analogous results in the $f_!$ case are proved similarly.
\end{proof}

\begin{proposition} \label{prop:pushfwd-finite}
We have the following:
\begin{enumerate}
\item If $f^*$ preserves finite generation, then $f_*$ preserves pointwise finiteness and the functor $f_* \colon \Mod_{\fC}^{\pf} \to \Mod_{\fD}^{\pf}$ is the right adjoint of the functor $f^* \colon \Mod_{\fD}^{\pf} \to \Mod_{\fC}^{\pf}$.
\item If $(f^\op)^*$ preserves finite generation, then $f_!$ preserves pointwise finiteness and the functor $f_! \colon \Mod_{\fC}^{\pf} \to \Mod_{\fD}^{\pf}$ is the left adjoint of the functor $f^* \colon \Mod_{\fD}^{\pf} \to \Mod_{\fC}^{\pf}$.
\end{enumerate}
\end{proposition}

\begin{proof}
  Let $M$ be a pointwise finite $\fC$-module.

(a)  Suppose that $f^*$ preserves finite generation. Let $y$ be an object of $\fC$. Then the $\fC$-module $f^* \bP_{\fD,y}$ is finitely generated, and so we can find a surjection $\bigoplus_{i=1}^n \bP_{\fC,x_i} \to f^* \bP_{\fD,y}$ for some choice of objects $x_1, \ldots, x_n \in \fC$. We therefore have an injection
\begin{displaymath}
(f_*M)(y) = \Hom_{\fC}(f^*\bP_{\fD,y},M) \to \bigoplus_{i=1}^n \Hom_{\fC}(\bP_{\fC,x_i}, M) = \bigoplus_{i=1}^n M(x_i).
\end{displaymath}
Since each $M(x_i)$ is finite dimensional, the space on the right side is finite dimensional, and so $(f_*M)(y)$ is finite dimensional. Thus $f_*M$ is pointwise finite. The adjointness statement follows easily.

(b) Now suppose instead that $(f^\op)^*$ preserves finite generation. Let $y$ be an object of $\fC^\op$. Then we have a surjection $\bigoplus_{i=1}^m \bP_{\fC^\op,y_i} \to f^* \bP_{\fD^\op,y}$ and hence by taking pointwise duals, we have an injection $f^* \bI_{\fD,y} \to \bigoplus_{i=1}^m \bI_{\fC,y_i}$. This gives an injection
\[
  (f_!M)(y)^* = \Hom_\fC(M,f^*\bI_{\fD,y})\to \bigoplus_{i=1}^m \Hom_\fC(M, \bI_{\fC,y_i}) \to \bigoplus_{i=1}^m M(y_i)^*.
\]
The rest of the argument is the same as the previous case.
\end{proof}

\begin{proposition} \label{prop:dual-pushfwd}
Suppose that $\Hom_\fC(x,y)$ is finite-dimensional for all objects $x,y$, and let $M$ be a $\fC$-module.
\begin{enumerate}
\item The natural map $(f_!M)^{\vee} \to f^{\op}_*(M^{\vee})$ is an isomorphism.
\item If $M$ and $f_!^{\op}(M^{\vee})$ are pointwise finite, then there is a canonical isomorphism $(f_*M)^{\vee} \to f_!^{\op}(M^{\vee})$.
\end{enumerate}
\end{proposition}

\begin{proof}
(a) Let $y \in \fD$. We have natural maps and identifications
\begin{align*}
(f_!M)^{\vee}(y)
&= ((f_!M)(y))^*
= (f^* \bP_{\fD^{\op},y} \odot_{\fC} M)^*
\to \Hom_{\fC}(M, (f^* \bP_{\fD^{\op},y})^{\vee}) \\
&= \Hom_{\fC^{\op}}(f^* \bP_{\fD^{\op},y}, M^{\vee})
= (f_* M^{\vee})(y).
\end{align*}
By our finite-dimensionality assumption, the map above is an isomorphism (see Proposition~\ref{prop:tensor-hom-dual}).

(b) Follows by applying (a) to $M^{\vee}$ and identifying $(f^{\op}_!(M^{\vee}))^{\vee}{}^{\vee}$ with $f^{\op}_!(M^{\vee})$.
\end{proof}

\subsection{Indecomposable decompositions} \label{ss:decomp}

We now investigate how $\fC$-modules decompose under direct sum. Our main result is:
 
\begin{proposition} \label{prop:decomp}
Let $M$ be a pointwise finite $\fC$-module. Then $M$ decomposes into a direct sum of indecomposable $\fC$-modules. If additionally $\End_{\fC}(M)$ is finite dimensional then this decomposition is finite and unique up to permutation and isomorphism.
\end{proposition}

We note that if $M$ is pointwise finite and either $M$ or $M^{\vee}$ is finitely generated then $\End_{\fC}(M)$ is finite dimensional. We require some preliminaries before proving the proposition. For a $\fC$-module $M$, let $\Sigma(M)$ be the set of pairs $(A,B)$ where $A$ and $B$ are submodules of $M$ such that $M=A \oplus B$. We define a partial order $\le$ on $\Sigma(M)$ by $(A,B) \le (A',B')$ if $A \subset A'$ and $B' \subset B$.

\begin{lemma} \label{lem:decomp-1}
Let $M$ be a pointwise finite $\fC$-module and let $\{(A_i,B_i)\}_{i \in I}$ be a chain in $\Sigma(M)$. Let $A=\bigcup_{i \in I} A_i$ and $B=\bigcap_{i \in I} B_i$. Then $(A,B) \in \Sigma(M)$.
\end{lemma}

\begin{proof}
Fix $x \in \fC$. We must show that $M(x)=A(x) \oplus B(x)$. Since $M(x)$ is finite dimensional, the ascending chain $A_i(x)$ stabilizes, and the descending chain $B_i(x)$ stabilizes. Since limits and colimits of $\fC$-modules are computed pointwise, we have $A(x)=A_i(x)$ and $B(x)=B_i(x)$ for $i \gg 0$. Since $M(x)=A_i(x) \oplus B_i(x)$ for all $i$ by assumption, the result follows.
\end{proof}

\begin{lemma} \label{lem:decomp-2}
Let $M$ be a non-zero pointwise finite $\fC$-module. Then $M$ admits an indecomposable summand.
\end{lemma}

\begin{proof}
Let $x$ be such that $M(x)$ is non-zero. Let $\Sigma'$ be the subset of $\Sigma(M)$ consisting of pairs $(A,B)$ with $B(x) \ne 0$. The set $\Sigma'$ contains $(0, M)$ and is therefore non-empty. Suppose that $\{(A_i,B_i)\}$ is a chain in $\Sigma'$ and let $A=\bigcup A_i$ and $B=\bigcap B_i$. Then $(A,B) \in \Sigma(M)$ by Lemma~\ref{lem:decomp-1}. Since $M(x)$ is finite dimensional, the descending chain $B_i(x)$ stabilizes, and so $B(x)=B_i(x)$ for $i \gg 0$; in particular, $B(x)$ is non-zero. Thus $(A,B) \in \Sigma'$.

By Zorn's lemma, $\Sigma'$ has a maximal element, say $(A,B)$. Suppose $B$ is decomposable, say $B=B' \oplus B''$ with both $B'$ and $B''$ non-zero. Without loss of generality, suppose $B''(x) \ne 0$. Then $(A \oplus B', B'')$ belongs to $\Sigma'$ and is strictly larger than $(A,B)$, a contradiction. Thus $B$ is indecomposable, which completes the proof.
\end{proof}

\begin{proof}[Proof of Proposition~\ref{prop:decomp}]
  Let $\{E_j\}_{j \in J}$ be the set of all indecomposable summands of $M$. We say that a subset $U$ of $J$ is \emph{independent} if the map $\bigoplus_{j \in U} E_j \to M$ is injective; we then write $A_U$ for the image of this map.

Let $\Sigma'$ be the set of pairs $(U, B)$ where $U$ is an independent subset of $J$ and $B$ is a submodule of $M$ such that $M=A_U \oplus B$, i.e., $(A_U, B) \in \Sigma(M)$. The set $\Sigma'$ contains $(\emptyset, M)$, and is therefore non-empty. We partially order $\Sigma'$ by $(U,B) \le (U',B')$ if $U \subset U'$ and $B' \subset B$. Suppose that $\{(U_i,B_i)\}_{i \in I}$ is a chain in $\Sigma'$. Let $U=\bigcup_{i \in I} U_i$ and $B=\bigcap_{i \in I} B_i$. Then $U$ is independent and $A_U=\bigcup_{i \in I} A_{U_i}$. Thus $(A_U, B) \in \Sigma(M)$ by Lemma~\ref{lem:decomp-1}, and so $(U,B) \in \Sigma'$.

By Zorn's lemma, $\Sigma'$ contains a maximal element, say $(U, B)$. Suppose $B \ne 0$. Then $B$ contains an indecomposable summand $E$ by Lemma~\ref{lem:decomp-2}. Since $B$ is a summand of $M$, it follows that $E$ is a summand of $M$, and thus equal to $E_j$ for some $j \in J$. Note that since $E_j$ belongs to a complementary submodule to $A_U$, the set $U'=U \cup \{j\}$ is independent, and $A_{U'}=A_U \oplus E_j$. Writing $B=E_j \oplus B'$, we have $M=A_U \oplus B=A_{U'} \oplus B'$, and so $(U', B') \in \Sigma'$. Since $(U,B)<(U',B')$ we have a contradiction. Thus $B=0$, and so $M=A_U$ is a direct sum of indecomposable modules.

Finally, suppose that $\End_{\fC}(M)$ is finite dimensional. Then the indecomposable decomposition of $M$ is finite: indeed, if $M=N_1 \oplus \cdots \oplus N_r$ is any decomposition with each $N_i$ non-zero then $\dim \End_{\fC}(M) \ge r$. If $E$ is any indecomposable summand of $M$ then $\End_{\fC}(E)$ is finite dimensional and has no non-trivial idempotents, and is thus local in the sense of \cite[\S 3]{krause} (see \cite[Corollary 19.19]{lam}). The uniqueness therefore follows from \cite[Theorem~4.2]{krause}
\end{proof}

\subsection{Upwards and downwards categories} \label{ss:upwards}

Let $|\fC|$ be the set of isomorphism classes in $\fC$. Suppose that $|\fC|$ is given a partial ordering $\le$. We say that $\fC$ is {\bf upwards} if, whenever there is a non-zero morphism $x \to y$, we have $x \le y$. We say that $\fC$ is {\bf downwards} if, whenever there is a non-zero morphism $x \to y$, we have $x \ge y$.

\begin{proposition} \label{prop:low-proj}
  Let $\fC$ be an upwards category and let $P$ be a projective $\fC$-module.
  Suppose that $\{x \in |\fC| \mid P(x) \ne 0\}$ has a minimal element $x_0$. Let $M$ be an $\End_\fC(x_0)$-summand of $P(x_0)$. Then the $\fC$-submodule of $P$ generated by $M$ is a summand of $P$; in particular, it is projective.
\end{proposition}

\begin{proof}
Let $\fX$ be the full subcategory of $\fC$ spanned by $x_0$ and let $\iota \colon \fX \to \fC$ be the inclusion; we identify $\fX$-modules with $\End_{\fC}(x_0)$-modules. Then $P(x_0) = \iota^* P$ is a projective $\fX$-module and hence so is $M$. Thus $Q = \iota_! M$ is a projective $\fC$-module with a map $Q \to P$.

  Next, $M$ is a quotient of $P$ and $Q$ by our assumption of minimality of $x_0$ and the fact that $\fC$ is upwards. Hence we lift $P \to M$ along $Q \to M$ to get a map $P \to Q$ such that the composition $Q \to P \to Q$ is the identity in degree $x_0$. This shows that $Q \to P$ is a split inclusion and hence its cokernel is projective.
\end{proof}

The following are analogous and we omit the proofs.

\begin{proposition}
  Let $\fC$ be an upwards category and let $I$ be an injective $\fC$-module. Suppose that $\{x \in |\fC| \mid I(x) \ne 0\}$ has a maximal element $x_0$. If $I' \subset I$ is the largest submodule such that $I'(x_0) = 0$, then $I'$ and $I/I'$ are injective.
\end{proposition}

We have analogous results for downwards categories which we state without proof (or note that the opposite of a downwards category is an upwards category).

\begin{proposition} 
  Let $\fC$ be a downwards category and let $I$ be an injective $\fC$-module.
  Suppose that $\{x \in |\fC| \mid I(x) \ne 0\}$ has a maximal element $x_0$. If $I' \subset I$ is the largest submodule such that $I'(x_0)=0$, then $I'$ and $I/I'$ are injective.
\end{proposition}

\begin{proposition}
  Let $\fC$ be a downwards category and let $P$ be a projective $\fC$-module. Suppose that $\{x \in |\fC| \mid P(x) \ne 0\}$ has a minimal element $x_0$. If $P' \subset P$ is the submodule generated by $P(x_0)$, then $P'$ and $P/P'$ are projective.
\end{proposition}

\begin{proposition} \label{prop:min-res}
  Let $\fC$ be an upwards category such that $\End_\fC(x)$ is semi-simple for all objects $x$ and such that $|\fC|$ contains no infinite decreasing sequences. If $M$ is a $\fC$-module, then it admits a minimal projective resolution $P_\bullet \to M \to 0$, i.e., for every simple $\fC$-module $S$, the differentials of $\Hom(P_\bullet,S)$ are identically $0$.

  In particular, if $\Ext^1_\fC(M,S)=0$ for all simple $\fC$-modules $S$, then $M$ is projective.
\end{proposition}

\begin{proof}
  It suffices to show that there is a map from a projective module $P \to M$ such that $\Hom(M,S)\to\Hom(P,S)$ is 0 for every simple $\fC$-module $S$. Let $\fC'$ be the subcategory of $\fC$ with all non-endomorphisms removed and let $i \colon \fC' \to \fC$ be the inclusion. For each object $x$, let $N(x)$ be the quotient of $M(x)$ by the subspace generated by the images of $M(y) \to M(x)$ for morphisms $y \to x$ with $y<x$. Then $\bigoplus_x N(x)$ is a projective $\fC'$-module by semi-simplicity, and we take $P = \bigoplus_x i_! N(x)$. We get a map $P \to M$ which is surjective: if $m \in M(x)$, then the set of $y$ such that there is a morphism $y \to x$ so that $m$ is in the span of the images of $M(y) \to M(x)$ has minimal elements, and so $m$ is in the image of the sum of the corresponding projective modules.

  For the last statement, if $P_\bullet \to M \to 0$ is a minimal projective resolution, then $\Ext^1_\fC(M,S)=0$ for all simple $\fC$-modules $S$ implies that $P_1=0$, and hence $P_0 \to M$ is an isomorphism.
\end{proof}

\subsection{Base change} \label{ss:base-change}

Consider a diagram of $\bk$-linear categories
\begin{displaymath}
\xymatrix{
\fA \ar[r]^{g'} \ar[d]_{f'} & \fB \ar[d]^f \\
\fC \ar[r]^g & \fD }
\end{displaymath}
that is commutative up to isomorphism.

\begin{proposition} \label{prop:bc}
We have the following:
\begin{enumerate}
\item For any $\fB$-module $M$, there is a natural map (the \textbf{base change map}) of $\fC$-modules
\begin{displaymath}
\phi_M \colon f'_! (g')^* M \to g^* f_! M.
\end{displaymath}
\item For any $b \in \fB$ and $c \in \fC$ there is a natural map
\begin{displaymath}
\psi_{b,c} \colon (f'^\op)^* \bP_{\fC^{\op},c} \odot_\fA (g')^* \bP_{\fB,b}  \to \Hom_{\fD}(f(b), g(c))
\end{displaymath}
\item $\phi_M$ is surjective for all $M$ if and only if $\psi_{b,c}$ is surjective for all $b$ and $c$.
\item $\phi_M$ is an isomorphism for all $M$ if and only if $\psi_{b,c}$ is an isomorphism for all $b$ and $c$.
\end{enumerate}
\end{proposition}

\begin{proof}
  Consider the following sequence of maps (the equalities follow by adjunctions and using $gf' = fg'$)
  \begin{align*}
    \hom_\fD(f_!M,f_!M) &= \hom_\fB(M,f^*f_!M)\\
                        &\to \hom_\fA(g'^*M, g'^* f^* f_! M)\\
                        &=\hom_\fA(g'^*M, f'^*g^*f_!M)\\
                        &= \hom_\fC(f'_! g'^*M, g^*f_!M)
  \end{align*}
  Taking the image of the identity on $f_!M$ gives $\phi_M$.

  We define $\psi_{b,c}$ as the following composition
  \begin{align*}
    (f'^\op)^*\bP_{\fC^{\op},c} \odot_\fA (g')^* \bP_{\fB,b} &= \bP_{\fC^\op,c} \odot_\fC f'_! (g')^* \bP_{\fB,b}\\
                                                       &\to \bP_{\fC^\op,c} \odot_\fC g^* f_! \bP_{\fB,b}\\
                                                       &= \bP_{\fC^\op,c} \odot_\fC g^* \bP_{\fD,f(b)}\\
    &= g^* \bP_{\fD,f(b)}(c)\\
    &= \hom_\fD(f(b), g(c))
  \end{align*}
  where the first equality is Proposition~\ref{prop:tensor-adjunction}, the next map uses functoriality of tensor products and the map $\phi_{\bP_{\fB,b}}$ from (a), the second equality is Proposition~\ref{prop:kan-proj}, the third equality is Proposition~\ref{prop:tensor-eval}, and the last equality follows from definitions.

From the constructions, we see that $\psi_{b,c}$ is surjective, respectively an isomorphism, if and only if the same is true for $1_{\bP_{\fC^\op,c}} \otimes \phi_{\bP_{\fB,b}}$. For any map $f \colon M \to N$ of $\fC$-modules, tensoring with $\bP_{\fC^\op,c}$ gives $f(c) \colon M(c) \to N(c)$ by Proposition~\ref{prop:tensor-eval}. Hence $f$ is surjective, respectively an isomorphism, if and only if the same is true after tensoring with $\bP_{\fC^\op,c}$ for all objects $c$. In particular, $\psi_{b,c}$ is surjective, respectively an isomorphism, for all $b \in \fB$ and $c \in \fC$ if and only if $\phi_{\bP_{\fB,b}}$ is surjective, respectively an isomorphism, for all $b \in \fB$. Now consider an arbitrary $\fB$-module $M$ and pick a presentation $P_1 \to P_0 \to M \to 0$ where each $P_i$ is a sum of principal projectives. By exactness of pullback and right-exactness of left Kan extensions, we get
  \[
    \xymatrix{f'_! (g')^* P_1 \ar[r] \ar[d] & f'_! (g')^* P_0 \ar[r] \ar[d] & f'_! (g')^* M \ar[r] \ar[d] & 0 \\
      g^* f_! P_1 \ar[r] &       g^* f_! P_0 \ar[r] &       g^* f_! M \ar[r] &       0 } 
  \]
  Each $P_i$ is a direct sum of principal projectives, so by a diagram chase, we see that $\phi_P$ is surjective, respectively an isomorphism, for all principal projective modules $P$ if and only if the same is true for $\phi_M$ for all $\fB$-modules $M$.
\end{proof}

We also have a version using $f_*$. The proof is similar to the one above, so we omit it.

\begin{proposition} \label{prop:bc-alt}
We have the following:
\begin{enumerate}
\item For any $\fB$-module $M$, there is a natural map (the \textbf{base change map}) of $\fC$-modules
\begin{displaymath}
\phi'_M \colon g^* f_* M \to f'_* (g')^* M .
\end{displaymath}
\item For any $b \in \fB$ and $c \in \fC$ there is a natural map
\begin{displaymath}
\psi'_{b,c} \colon \hom_\fD(g(c), f(b))^* \to \hom_\fA((f')^*\bP_{\fC,c}, (g')^* \bI_{\fB,b})
\end{displaymath}
\item $\phi'_M$ is injective for all $M$ if and only if $\psi'_{b,c}$ is injective for all $b$ and $c$.
\item $\phi'_M$ is an isomorphism for all $M$ if and only if $\psi'_{b,c}$ is an isomorphism for all $b$ and $c$.
\end{enumerate}
\end{proposition}

\subsection{Tensor products} \label{ss:tensor}

Suppose that $\fC$ has a monoidal operation $\amalg$; we use this notation since in the examples of interest to us the monoidal structure is given by disjoint union. We let $\emptyset$ denote the unit object for $\amalg$. Let $M$ and $N$ be $\fC$-modules. We define $M \otimes_{\fC} N$ to be the $\fC$-module $\amalg_!(M \boxtimes N)$. Here $M \boxtimes N$ denotes the $(\fC \times \fC)$-module given by $(x,y) \mapsto M(x) \otimes N(y)$, and $\amalg_!$ is the pushforward along the monoidal operation $\amalg \colon \fC \times \fC \to \fC$. We also define $M \otimes_{\fC,*} N$ by $\amalg_*(M \boxtimes N)$, though this will be less used. When there is no danger of ambiguity, we write $\otimes$ in place of $\otimes_{\fC}$. These tensor products are sometimes called {\bf Day convolution}.

\begin{proposition} \label{prop:tensor-prod} 
We have the following:
\begin{enumerate}
\item The tensor product $\otimes$ is cocontinuous (and thus right exact) in each variable.
\item The tensor product $\otimes$ naturally gives $\Mod_{\fC}$ the structure of a monoidal category, with unit object $\bP_{\emptyset}$.
\item If $\amalg$ is a symmetric monoidal operation then so is $\otimes$.
\item For objects $x,y \in \fC$, we have a natural isomorphism $\bP_x \otimes \bP_y \cong \bP_{x \amalg y}$.
\item If $M$ and $N$ are finitely generated then so is $M \otimes N$.
\end{enumerate}
Analogous dual statements hold for $\otimes_*$.
\end{proposition}

\begin{proof}
  (a) This follows since $\boxtimes$ is bi-cocontinuous and $\amalg_!$ is cocontinuous, being a left adjoint.

  (b)   Since $\amalg$ is monoidal, we have a natural isomorphism $\amalg \circ (\amalg \times 1) \to \amalg \circ (1 \times \amalg)$ as functors $\fC\times \fC \times \fC \to \fC$.  Let $M,N,P$ be $\fC$-modules. Then we have an associator defined by
  \[
    (M \otimes N) \otimes P = \amalg_!(\amalg \times 1)_! (M \boxtimes N \boxtimes P) \to \amalg_!(1 \times \amalg)_!(M \boxtimes N \boxtimes P) = M \otimes (N \otimes P).
  \]
  The pentagon axiom transfers from its validity for $\amalg$.
  
  Let $M,N$ be $\fC$-modules. Then we have
  \[
    \hom_\fC(\bP_\emptyset \otimes M, N) = \hom_{\fC \times \fC}(\bP_\emptyset \boxtimes M, \amalg^* N).
  \]
  Hence a choice of map $\bP_\emptyset \otimes M \to N$ is a choice of maps $\bP_\emptyset(x) \otimes M(y) \to N(x \amalg y)$ for all $x,y \in \fC$ compatible with the action of morphisms. In particular, $x=\emptyset$ determines the case of arbitrary $x$, and so it is determined by maps $M(y) \to N(y)$ compatible with morphisms, i.e., a choice of map $M \to N$. Hence $\bP_\emptyset \otimes M \cong M$ since they represent the same functor. We leave verification of the remaining axioms to the reader.

  (c) If $\amalg$ has a symmetry, i.e., a natural isomorphism $\amalg \to \amalg \circ \tau$ where $\tau$ is the switching map, then we use it to define a symmetry for $\otimes$:
  \[
    M\otimes N = \amalg_!(M \boxtimes N) \to \amalg_!\tau_!(M \boxtimes N) = \amalg_!(N \boxtimes M) = N \otimes M.
  \]

  (d) The argument is similar to the one used in (b): $\bP_x \otimes \bP_y$ represents the same functor as $\bP_{x \amalg y}$.

  (e) Given surjections $\bigoplus_i P_{x_i} \to M$ and $\bigoplus_j P_{y_j} \to N$, we have a surjection $\bigoplus_{i,j} P_{x_i \amalg y_j} \to M\otimes N$ using (a) and (d) (all sums finite).
\end{proof}

We define $\Tor^{\fC}_{\bullet}(-, -)$ to be the derived functor of $\otimes_{\fC}$. The usual argument shows that this is balanced, i.e., one can compute $\Tor$ by using resolutions in either variable.

Since $\bP_{\emptyset}$ is the unit object for $\otimes$, we have a natural isomorphism $\bP_{\emptyset} \otimes \bP_{\emptyset} \to \bP_{\emptyset}$. We can thus regard $\bP_{\emptyset}$ as an algebra object of $\Mod_{\fC}$. We define an {\bf ideal} of $\bP_{\emptyset}$ (or of $\fC$) to be a $\fC$-submodule of $\bP_{\emptyset}$. Let $\fa$ be an ideal and let $M$ be a $\fC$-module. We define $\fa M$ to be the image of the composite map $\fa \otimes M \to \bP_{\emptyset} \otimes M \to M$, where the first map is induced by the inclusion $\fa \to \bP_{\emptyset}$ and the second is the canonical isomorphism. In particular, if $\fa$ and $\fb$ are ideals then we have a product ideal $\fa\fb$.

We now give an exactness criterion for tensor products. For simplicity, we assume that $\fC$ is the linearization of a category $\cC$ with finite $\Hom$ sets that is closed under $\amalg$. Consider the following condition on an object $x \in \fC$:
\begin{itemize}
\item[$(S_x)$] There exist maps $\{ \phi_i \colon x \to y_i \amalg z_i\}_{i \in I}$ in $\cC$ with the following property: given any map $\psi \colon x \to y \amalg z$ in $\cC$ there exists a factorization $\psi = (\alpha \amalg \beta) \circ \phi_i$ for some $i$ and morphisms $\alpha \colon y_i \to y$, and $\beta \colon z_i \to z$ in $\cC$; moreover, if $\psi=(\alpha' \amalg \beta') \circ \phi_j$ is a second such factorization then $i=j$ and $(\alpha,\beta)=(\alpha',\beta') \circ \sigma$ for some $\sigma \in \Aut_{\cC}(y_i) \times \Aut_{\cC}(z_i)$ fixing $\phi_i$.
\end{itemize}

\begin{proposition} \label{prop:amalg-proj}
Suppose $(S_x)$ holds, and let $G_i \subset \Aut_{\cC \times \cC}(y_i,z_i)$ be the stabilizer of $\phi_i$. Then $\amalg^*(\bP_x) \cong \bigoplus_{i \in I} (\bP_{(y_i,z_i)})_{G_i}$, where the subscript denotes coinvariants. In particular, if $\bk$ has characteristic~$0$ then $\amalg^*(\bP_x)$ is projective.
\end{proposition}

\begin{proof}
The morphism $\phi_i$ induces a morphism $\bP_{(y_i,z_i)} \to \amalg^*(\bP_x)$ that factors through $(\bP_{(y_i,z_i)})_{G_i}$ since $\phi_i$ is $G_i$-invariant. We thus have a map $f \colon \bigoplus_{i \in I} (\bP_{(y_i,z_i)})_{G_i} \to \amalg^*(\bP_x)$. At an object $(y,z) \in \fC \times \fC$, the map $f$ is the linearization of the map
\begin{displaymath}
\coprod_{i \in I} \Hom_{\cC \times \cC}((y_i,z_i),(y,z))/G_i \to \Hom_{\cC}(x, y \amalg z).
\end{displaymath}
The condition $(S_x)$ exactly ensures that the above map is a bijection, which proves that $f$ is an isomorphism. If $\bk$ has characteristic~0 then $(\bP_{(y_i,z_i)})_{G_i}$ is a summand of $\bP_{(y_i,z_i)}$, and therefore projective, and so $\amalg^*(\bP_x)$ is projective. (Note: we have assumed $\cC$ has finite $\Hom$ sets, so the group $G_i$ is finite.)
\end{proof}

\begin{proposition} \label{prop:tensor-exact}
Suppose $(S_x)$ holds for all $x$ and $\bk$ has characteristic~$0$. Then $\otimes_*$ is exact on $\Mod_{\fC}$, and $\otimes$ is exact on $\Mod_{\fC^{\op}}$
\end{proposition}

\begin{proof}
Since $\amalg^*(\bP_x)$ is projective for all $x$ (Proposition~\ref{prop:amalg-proj}), we see that $\amalg_*$ and $\amalg^{\op}_!$ are exact (Proposition~\ref{prop:!-exact}), from which the result follows.
\end{proof}

\section{Triangular categories} \label{sec:tri-cat}

\subsection{The definition}

We assume in what follows that the field $\bk$ is algebraically closed. Let $\fG$ be a $\bk$-linear category satisfying the following condition:
\begin{itemize}
\item[(T0)] The category $\fG$ is essentially small, and all $\Hom$ spaces are finite dimensional.
\end{itemize}
We denote the set of isomorphism classes in $\fG$ by $\vert \fG \vert$. Recall that a subcategory is {\bf wide} if it contains all objects. The following are the central definitions of this paper:

\begin{definition} \label{defn:tri-struct}
A {\bf triangular structure} on $\fG$ is a pair $(\fU, \fD)$ of wide subcategories of $\fG$ such that the following axioms hold: 
\begin{itemize}
\item[(T1)] We have $\End_{\fU}(x)=\End_{\fD}(x)$ for all objects $x$, and this ring is semi-simple.
\item[(T2)] There exists a partial order $\le$ on the set $\vert \fG \vert$ such that:
\begin{enumerate}
\item For all $x \in \vert \fG \vert$ there are only finitely many $y \in \vert \fG \vert$ with $y \le x$.
\item The category $\fU$ is upwards with respect to $\le$ (see \S \ref{ss:upwards}).
\item The category $\fD$ is downwards with respect to $\le$.
\end{enumerate}
\item[(T3)] For all $x,z \in \fG$, the natural map
\begin{displaymath}
\bigoplus_{y \in \vert \fG \vert} \Hom_{\fU}(y, z) \otimes_{\End_{\fU}(y)} \Hom_{\fD}(x, y) \to \Hom_{\fG}(x,z)
\end{displaymath}
is an isomorphism. \qedhere
\end{itemize}
\end{definition}

\begin{definition} \label{defn:tri-cat}
A {\bf triangular category} is a $\bk$-linear category satisfying (T0) equipped with a triangular structure.
\end{definition}

We refer to an order as in (T2) as an {\bf admissible order}. Admissible orders are not unique in general. However, there is a unique weakest admissible order $\preceq$, which can be defined by $x \preceq y$ if $x \le y$ for all admissible orders $\le$; we refer to $\preceq$ as the {\bf canonical order}.

Suppose that $\fG$ and $\fG'$ are triangular categories. A functor $\fG \to \fG'$ is called {\bf triangular} if it carries $\fU$ into $\fU'$ and $\fD$ into $\fD'$. (Warning: a triangular functor need not induce an order-preserving map $\vert \fG \vert \to \vert \fG' \vert$.) A {\bf triangular equivalence} $\fG \to \fG'$ is an equivalence that is triangular and whose quasi-inverse is also triangular; a triangular equivalence induces equivalences $\fU \to \fU'$ and $\fD \to \fD'$, and a bijection $\vert \fG \vert \to \vert \fG' \vert$ that is compatible with the canonical orders.

Suppose we have a triangular structure on $\fG$. Let $\fM=\fU \cap \fD$. Thus $\fM$ contains all objects of $\fG$ and a morphism belongs to $\fM$ if and only if it belongs to both $\fU$ and $\fD$. By axiom (T2), we see that if $x \to y$ is a non-zero morphism in $\fM$ then $x$ and $y$ are isomorphic. Thus $\fM$ is essentially determined by its endomorphism rings, and these are given by $\End_{\fM}(x)=\End_{\fU}(x)=\End_{\fD}(x)$. In axiom (T3) above, we could instead tensor over $\End_{\fM}(y)$, and this gives the axiom a more symmetrical appearance. We think of $\fU$ as an upper triangular parabolic subalgebra, $\fD$ as a lower triangular parabolic subalgebra, and $\fM$ as the common Levi factor. Axiom (T3) can be seen as a Poincar\'e--Birkhoff--Witt decomposition.

\begin{remark}
The assumption that $\bk$ is algebraically closed can be relaxed: it is enough that every simple $\fM$-module is absolutely simple.
\end{remark}

\subsection{Pushforwards and pullbacks}

Fix a triangular category $\fG$ with subcategories $\fU$, $\fD$, and $\fM$, and an admissible order $\le$. We name the various inclusion functors as follows:
\begin{displaymath}
\xymatrix@C=4em@R=3em{
\fM \ar[r]^{j'} \ar[d]_{i'} \ar[rd]^k & \fU \ar[d]^i \\
\fD \ar[r]^j & \fG }
\end{displaymath}

\begin{proposition} \label{prop:tri-pushfwd}
We have the following:
\begin{enumerate}
\item The base change map $j'_!(i')^*M \to i^*j_!M$ is an isomorphism for any $M \in \Mod_{\fD}$.
\item The base change map $j^*i_*M \to i'_*(j')^*M$ is an isomorphism for any $M \in \Mod_{\fU}$.
\item If $M \in \Mod_{\fG}$ is a summand of an object in the essential image of $j_!$ then $i^*(M)$ is a projective $\fU$-module. In particular, $i^*$ takes projective $\fG$-modules to projective $\fU$-modules.
\item If $M \in \Mod_{\fG}$ is a summand of an object in the essential image of $i_*$ then $j^*(M)$ is an injective $\fD$-module. In particular, $j^*$ takes injective $\fG$-modules to injective $\fD$-modules.
\item The functors $i_*$ and $j_!$ are exact.
\end{enumerate}
\end{proposition}

\begin{proof}
  (a) follows from Proposition~\ref{prop:bc}: we have to check that the maps
  \[
    (i'^\op)^* \bP_{\fD^\op,d} \odot_\fM j'^* \bP_{\fU,u} \to \hom_\fG(i(u), j(d))
  \]
  are isomorphisms for all $d \in \fD$ and $u \in \fU$. Since all morphisms in any skeletal category of $\fM$ are endomorphisms, we can rewrite the map as
  \[
    \bigoplus_{x \in |\fM|} \hom_\fD(x,d) \otimes_{\End_\fM(x)} \hom_\fU(u, x) \to \hom_\fG(u,d)
  \]
  where the components are given by composition. But this is an isomorphism by the axioms of a triangular category. (b) is similar.

  For (c), suppose $M=j_!(N)$. Then by (a), we have $i^*(M) \cong j'_!((i')^*(N))$, and anything in the essential image of $j'_!$ is projective (as $j'_!$ takes projectives to projectives and $\Mod_{\fM}$ is semi-simple). (d) is similar. (e) now follows from Proposition~\ref{prop:!-exact}.
\end{proof}

The following proposition summarizes some basic finiteness properties of pushforwards and pullbacks.

\begin{proposition} \label{prop:ij-finite}
We have the following:
\begin{enumerate}
\item The functors $i^*$ and $(i')^*$ preserve finite generation.
\item The functors $i_*$, $i'_*$, $j_!$, and $j'_!$ preserve pointwise finiteness.
\end{enumerate}
\end{proposition}

\begin{proof}
(a) Since $(i')^*$ takes simple $\fD$-modules to simple $\fM$-modules, it follows that $(i')^*$ takes finite length $\fD$-modules to finite length $\fM$-modules. Since the category $\fD$ is downwards, any finitely generated $\fD$-module has finite length. Thus $(i')^*$ preserves finite generation.

We have
\begin{displaymath}
i^*(\bP_{\fG,x})=i^*(j_!(\bP_{\fD,x}))=j'_!((i')^*\bP_{\fD,x}).
\end{displaymath}
Since $(i')^*\bP_{\fD,x}$ is a finitely generated $\fM$-module (see above) and $j'_!$ preserves finite generation (Proposition~\ref{prop:!-finite}), we see that $i^*(\bP_{\fG,x})$ is finitely generated, and so $i^*$ preserves finite generation.

(b) The statements for $i_*$ and $i'_*$ follow from (a) and Proposition~\ref{prop:pushfwd-finite}.

Suppose $M$ is a pointwise finite $\fM$-module. We can then realize $M$ as a quotient of $\bigoplus_{x \in \fG} \bP_{\fM,x}^{\oplus n(x)}$ for $n(x) \in \bN$. We thus see that $j'_!(M)$ is a quotient of $\bigoplus_{x \in \fG} \bP_{\fU,x}^{\oplus n(x)}$, and so $j'_!(M)(y)$ is a quotient of $\bigoplus_{x \in \fG} \bP_{\fU,x}^{\oplus n(x)}(y)$. This sum is finite since for $\bP_{\fU,x}(y)$ is non-zero only if $x \le y$, by the upwards property of $\fU$. Thus $j'_!(M)$ is pointwise finite.

Finally, suppose that $M$ is a pointwise finite $\fD$-module. We have $i^*(j_!(M)) = j'_!((i')^*(M))$ by Proposition~\ref{prop:tri-pushfwd}. Of course, $(i')^*(M)$ is pointwise finite, and so $j'_!((i')^*(M))$ is pointwise finite by the previous paragraph. Thus $i^*(j_!(M))$ is pointwise finite, and so $j_!(M)$ is pointwise finite too (since $i$ is bijective on objects).
\end{proof}

From the above proposition, we obtain the following useful noetherianity criterion:

\begin{proposition} \label{prop:Unoeth}
If $\Mod_{\fU}$ is locally noetherian then so is $\Mod_{\fG}$.
\end{proposition}

\begin{proof}
Suppose $M$ is a finitely generated $\fG$-module and $N_1 \subset N_2 \subset \cdots$ is an ascending chain of submodules. Then $i^*(M)$ is finitely generated by Proposition~\ref{prop:ij-finite}(a), and so the chain $i^*(N_1) \subset i^*(N_2) \subset \cdots$ stabilizes since $\Mod_{\fU}$ is locally noetherian. Since $i$ is bijective on objects, the original chain stabilizes too. Thus $M$ is noetherian, and so $\Mod_{\fG}$ is locally noetherian.
\end{proof}

Finally, we examine how pushforwards interact with $\Ext$.

\begin{proposition} \label{prop:ext-compare}
Let $M$ be a $\fG$-module, let $N$ be a $\fD$-module, and let $N'$ be a $\fU$-module. Then we have natural isomorphisms
\begin{displaymath}
\Ext^r_{\fG}(j_!N, M) = \Ext^r_{\fD}(N, j^*M), \qquad
\Ext^r_{\fG}(M, i_*N') = \Ext^r_{\fU}(i^*M, N')
\end{displaymath}
for all $r \ge 0$. 
\end{proposition}

\begin{proof}
Let $P_{\bullet} \to N$ be a projective resolution of $N$ as a $\fD$-module. Since $j_!$ is exact and takes projectives to projectives, it follows that $j_!P_{\bullet} \to j_!N$ is a projective resolution of $\fG$-modules. Thus the complex $\Hom_{\fD}(P_{\bullet}, j^*M)$ computes $\Ext^{\bullet}_{\fD}(N, j^*M)$, while the complex $\Hom_{\fG}(j_!P_{\bullet}, M)$ computes $\Ext^{\bullet}_{\fG}(j_!N, M)$. These complexes are isomorphic by adjunction, and so the result for $j$ follows. The result for $i$ is similar.
\end{proof}

\subsection{Weights} \label{ss:weights}

Define $\Lambda$ to be the set of isomorphism classes of simple $\fM$-modules. We refer to elements of $\Lambda$ as {\bf weights}. Given $\lambda \in \Lambda$, we choose a representative simple $\fM$-module $S_{\lambda}$ of the class of $\lambda$. From the structure of $\fM$, we see that $S_{\lambda}$ is non-zero on a unique isomorphism class of $\fG$. We refer to this class as the \textbf{support} of $\lambda$, and denote it by $\supp(\lambda)$. Since $\End_{\fM}(x)$ is finite dimensional for all $x$, there are only finitely many weights with a given support. We partially order $\Lambda$ through the ordering on supports: that is, we define $\lambda<\mu$ if $\supp(\lambda)<\supp(\mu)$. For any $\mu$, there are only finitely many $\lambda$ with $\lambda \le \mu$.

Given an $\fM$-module $M$ and a weight $\lambda$, we define the {\bf multiplicity} of $\lambda$ in $M$, denoted $m_{\lambda}(M)$ to be the multiplicity of $S_{\lambda}$ in $M$, which we regard as an element of $\bN \cup \{\infty\}$. We say that $\lambda$ {\bf occurs} in $M$ if its multiplicity is non-zero. We note that $M$ is pointwise finite if and only if $m_{\lambda}(M)$ is finite for all $\lambda$. We extend this notation and terminology to $\fD$-, $\fU$- and $\fG$-modules by simply restricting to $\fM$, e.g., we say that $\lambda$ occurs in a $\fG$-module if it occurs in its restriction to $\fM$.

We regard $S_{\lambda}$ as a $\fU$- or $\fD$- module by letting all maps between non-isomorphic objects act by zero. When clarity is required, we write $S_{\fM,\lambda}$, $S_{\fU,\lambda}$, or $S_{\fD,\lambda}$ to indicate the relevant module category.

\subsection{Duality} \label{ss:tridual}

Let $\widehat{\fG}=\fG^{\op}$ be the opposite category of $\fG$, and put $\widehat{\fU}=\fD^{\op}$ and $\widehat{\fD}=\fU^{\op}$. It is clear that $\widehat{\fG}$ satisfies (T0).

\begin{proposition}
The pair $(\widehat{\fU}, \widehat{\fD})$ is a triangular structure on $\widehat{\fG}$.
\end{proposition}

\begin{proof}
Axiom (T1) formally holds. Let $\le$ be an admissible order on $|\fG|$. We claim that $\le$ is also an admissible order on $|\widehat{\fG}|$. Condition (a) is automatic. If $x \to y$ is a morphism in $\widehat{\fD} = \fU^\op$, then $x \le y$ and so $\widehat{\fD}$ is downwards. Similarly, $\widehat{\fU}$ is upwards, so the claim follows. Thus (T2) holds. Finally, we have a commutative diagram
\begin{displaymath}
\xymatrix{
\bigoplus_{y \in |\fG^\op|} \Hom_{\widehat{\fU}}(y,z) \otimes_{\End_{\widehat{\fM}}(y)} \Hom_{\widehat{\fD}}(z,x) \ar[r] \ar@{=}[d] & \Hom_{\fG^\op}(x,z) \ar@{=}[d] \\
\bigoplus_{y \in |\fG|} \Hom_{\fD}(z,y) \otimes_{\End_{\fM}(y)^\op} \Hom_{\fU}(x,z) \ar[r] & \Hom_{\fG}(z,x) }
\end{displaymath}
(Note that left and right modules are interchanged upon going from a ring to its opposite.) The bottom map is an isomorphism by (T3) for $(\fU, \fD)$, and so the top map is an isomorphism as well. Thus $(\widehat{\fU}, \widehat{\fD})$ satisfies (T3).
\end{proof}

\begin{definition}
We call $\widehat{\fG}$, equipped with $(\widehat{\fU}, \widehat{\fD})$, the {\bf dual triangular category} to $\fG$.
\end{definition}

We use hats to denote the constructions associated to $\widehat{\fG}$, such as $\widehat{\imath}$, $\widehat{\jmath}$, and $\widehat{\Lambda}$. We note that there is a natural bijection $\Lambda \to \widehat{\Lambda}$, which we denote by $\lambda \mapsto \lambda^{\vee}$.

\begin{definition} \label{defn:transpose}
A \textbf{transpose} on $\fG$ is a triangular equivalence $\tau \colon \fG \to \widehat{\fG}$ such that
\begin{enumerate}
\item $\widehat{\tau} \circ \tau$ is isomorphic to $\id_{\fG}$; and 
\item the induced bijection $\tau \colon \Lambda \to \widehat{\Lambda}$ coincides with $(-)^{\vee}$. \qedhere
\end{enumerate}
\end{definition}

Given a transpose functor, we get a (covariant) equivalence $\Mod_{\fG} \cong \Mod_{\widehat{\fG}}$; thus, for a $\fG$-module $M$, we can (and usually do) regard $M^{\vee}$ as a $\fG$-module.

\subsection{Standard modules}

We now come to an important definition. For $\lambda \in \Lambda$, define the \textbf{standard module} associated to $\lambda$ by
\[
  \Delta_{\lambda} = j_!(S_{\fD,\lambda}),
\]
To explain the significance of these modules, we introduce another concept. We say that a $\fG$-module $M$ is a \textbf{lowest weight module} with lowest weight $\lambda \in \Lambda$ if the following conditions hold:
\begin{enumerate}
\item The weight $\lambda$ occurs in $M$ with multiplicity one.
\item $M$ is generated as a $\fG$-module by the unique copy of $S_{\lambda}$ that it contains.
\item If a weight $\mu$ occurs in $M$ then $\lambda \le \mu$.
\end{enumerate}

\begin{proposition}
We have the following:
\begin{enumerate}[\rm \indent (a)]
\item $\Delta_{\lambda}$ is a lowest weight module of lowest weight $\lambda$.
\item Every other lowest weight module with lowest weight $\lambda$ is a quotient of $\Delta_{\lambda}$.
\item There is a unique simple quotient $L_{\lambda}$ of $\Delta_{\lambda}$.
\item Every simple $\fG$-module is isomorphic to $L_{\lambda}$ for a unique $\lambda$.
\item The simple constituents of $\Delta_{\lambda}$, other than $L_{\lambda}$, have the form $L_{\mu}$ with $\mu>\lambda$.
\end{enumerate}
\end{proposition}

\begin{proof}
(a) This is clear from the base change formula.

(b) If $M$ is a lowest weight module of weight $\lambda$, then we have a $\fD$-module map $S_\lambda \to j^*M$. By adjunction, this gives a $\fG$-module map $\Delta_\lambda \to M$. Since it is an isomorphism in degree $\supp(\lambda)$ and $M$ is generated in this degree, this map is surjective.

(c) Any submodule of $\Delta_\lambda$ that is nonzero in degree $\supp(\lambda)$ is all of $\Delta_\lambda$. Hence, the sum of all submodules which are zero in degree $\supp(\lambda)$ is the unique largest submodule of $\Delta_\lambda$, and $L_\lambda$ is the quotient by this submodule.

(d) Let $M$ be a simple $\fG$-module and let $x \in \vert \fG \vert$ be minimal such that $M_x \ne 0$ (this exists by axiom (b) of triangular categories). If $N$ is a $\End_{\fU}(x)$-submodule of $M_x$ then the $\fG$-submodule of $M$ generated by $N$ is just $N$ in degree $x$; thus, if $N$ is non-zero then $N=M_x$ by simplicity. We thus see that $M_x$ is a simple $\End_{\fU}(x)$-module, and thus corresponds to some weight $\lambda \in \Lambda$. Hence $M$ is a quotient of $\Delta_\lambda$, and must be $L_{\lambda}$ since that is the unique simple quotient.

(e) The kernel of the quotient $\Delta_\lambda \to L_\lambda$ is concentrated in degrees strictly larger than $\supp(\lambda)$ since it is an isomorphism in degree $\supp(\lambda)$.
\end{proof}

\begin{proposition}
We have $L_{\lambda}^{\vee} \cong L_{\lambda^{\vee}}$.
\end{proposition}

\begin{proof}
The $\widehat{\fG}$-module $L_{\lambda}^{\vee}$ is simple and has lowest weight $\lambda^{\vee}$, and so it must be isomorphic to $L_{\lambda^{\vee}}$.
\end{proof}

There is a dual notion to standard modules: we define the {\bf costandard module} by
\[
  \nabla_{\lambda} = i_*(S_{\fU,\lambda}).
\]
As expected, duality interchanges standard and costandard modules:

\begin{proposition}
We have $\Delta_{\lambda}^{\vee} \cong \nabla_{\lambda^{\vee}}$.
\end{proposition}

\begin{proof}
We have
\begin{displaymath}
\Delta_{\lambda}^{\vee}=(j_! S_{\lambda})^{\vee}=\widehat{\imath}_*(S_{\lambda^{\vee}}) = \nabla_{\lambda^{\vee}},
\end{displaymath}
where in the second step we used Proposition~\ref{prop:dual-pushfwd}. (Recall that $\widehat{\imath}=j^{\op}$.)
\end{proof}

\begin{proposition}
The costandard module $\nabla_{\lambda}$ has a unique simple submodule. It is isomorphic to $L_{\lambda}$, and every simple constituent of $\nabla_{\lambda}/L_{\lambda}$ has the form $L_{\mu}$ with $\mu > \lambda$.
\end{proposition}

\begin{proof}
This follows from duality.
\end{proof}

\begin{remark} \label{rmk:transp-duals}
Suppose $\fG$ has a transpose $\tau$. Then, under the identification $\Mod_{\fG} \cong \Mod_{\widehat{\fG}}$ provided by $\tau$, we have $L_{\lambda} \leftrightarrow L_{\lambda^{\vee}}$ and $\Delta_{\lambda} \leftrightarrow \Delta_{\lambda^{\vee}}$ and $\nabla_{\lambda} \leftrightarrow \nabla_{\lambda^{\vee}}$. Thus if we regard $(-)^{\vee}$ as a contravariant endofunctor of $\Mod_{\fG}$ then $L_{\lambda}^{\vee}=L_{\lambda}$ and $\Delta_{\lambda}^{\vee}=\nabla_{\lambda}$ and $\nabla_{\lambda}^{\vee}=\Delta_{\lambda}$.
\end{remark}

\begin{proposition} \label{prop:hom-std-costd}
We have 
\begin{displaymath}
\dim \Hom_{\fG}(\Delta_{\lambda}, \nabla_{\mu}) = \delta_{\lambda,\mu}.
\end{displaymath}
Moreover, the image of any non-zero map $\Delta_{\lambda} \to \nabla_{\lambda}$ is $L_{\lambda}$.
\end{proposition}

\begin{proof}
We have
\begin{displaymath}
\Hom_{\fG}(\Delta_{\lambda}, \nabla_{\mu})
=\Hom_{\fU}(i^*(\Delta_{\lambda}), S_{\mu})
=\Hom_{\fU}(j'_!(S_{\lambda}), S_{\mu})
=\Hom_{\fM}(S_{\lambda}, S_{\mu}).
\end{displaymath}
As $\Hom_{\fM}(S_{\lambda}, S_{\mu})$ has dimension $\delta_{\lambda,\mu}$, the first statement is established. Since $\Hom_{\fG}(\Delta_{\lambda}, \nabla_{\lambda})$ is one-dimensional, it must be spanned by the composite map $\Delta_{\lambda} \to L_{\lambda} \to \nabla_{\lambda}$, from which the second statement follows.
\end{proof}

\subsection{Multiplicities} \label{ss:mult}

For a $\fG$-module $M$, we write $[M:L_{\lambda}]$ for the multiplicity of $L_{\lambda}$ in $M$, which we regard as an element of $\bN \cup \{\infty\}$. Precisely, let $S$ be the set of natural numbers $n$ for which there exists a chain of $\fG$-submodules
\begin{displaymath}
0 \subseteq F_1 \subsetneq G_1 \subseteq F_2 \subsetneq G_2 \subseteq \cdots \subseteq F_n \subsetneq G_n \subseteq M
\end{displaymath}
with $G_i/F_i \cong L_{\lambda}$ for all $i$. Then $[M:L_{\lambda}]$ is defined to be the supremum of $S$. This notion of multiplicity has the expected properties, namely:
\begin{itemize}
\item $[-:L_{\lambda}]$ is additive in short exact sequences.
\item If $N$ is a submodule of $M$ then $[N:L_{\lambda}] \le [M:L_{\lambda}]$.
\item If $M=\bigcup_{i \in I} M_i$ (directed union) then $[M:L_{\lambda}]=\sup_{i \in I} [M_i \colon L_{\lambda}]$.
\end{itemize}
For proofs, see, e.g., \cite[\S A.1]{increp}. We note that since $\lambda$ has multiplicity one in $L_{\lambda}$, we have $[M:L_{\lambda}] \le m_{\lambda}(M)$. The following technical proposition is helpful when studying multiplicities in infinite length modules.

\begin{proposition} \label{prop:nice-filt}
Let $M$ be a pointwise finite $\fG$-module and let $\Xi$ be a finite set of weights. We can then find a filtration $0=F_0 \subset \cdots \subset F_r=M$ such that for each $i$ the module $F_i/F_{i-1}$ is either simple or does not contain any weight in $\Xi$.
\end{proposition}

\begin{proof}
  We are free to replace $\Xi$ with a larger finite set, so we may as well assume it is downwards closed (i.e., $\lambda \in \Xi$ and $\mu<\lambda$ implies $\mu \in \Xi$). Define the \emph{total $\Xi$-multiplicity} of $M$ to be the sum of all multiplicities of weights in $\Xi$. We proceed by induction on this quantity. If the total $\Xi$-multiplicity is~0 there is nothing to prove, so suppose this is not the case. Let $\lambda$ be a minimal element of $\Xi$ occurring in $M$. Since $\Xi$ is downwards closed, it follows that $\lambda$ is a minimal weight of $M$. We can thus find a non-zero map $\phi \colon \Delta_{\lambda} \to M$.

  Let $K$ be the maximal proper submodule of $\Delta_{\lambda}$, so that $\Delta_{\lambda}/K=L_{\lambda}$. Since $\phi$ is non-zero, its kernel is a proper submodule of $\Delta_{\lambda}$, and thus contained in $K$. Consider the filtration $0 \subset \phi(K) \subset \phi(\Delta_{\lambda}) \subset M$. Now, $\phi(K)$ is a submodule of $M$ with strictly smaller total $\Xi$-multiplicity (since the multiplicity of $\lambda$ in $\phi(K)$ is~0); thus, by induction, $\phi(K)$ has a filtration of the desired kind. The quotient $\phi(\Delta_{\lambda})/\phi(K)$ is isomorphic to $L_{\lambda}$, and so it too trivially admits a filtration of the desired kind. Finally, $M/\phi(\Delta_{\lambda})$ also has smaller total $\Xi$-multiplicity than $M$, since $\phi(\Delta_{\lambda})$ has $\lambda$ in it, and so by induction it too admits a filtration of the desired kind. Splicing these filtrations together gives the desired filtration on $M$.
\end{proof}

\begin{proposition} \label{prop:mult-formula}
Let $M$ be an arbitrary $\fG$-module. Then
\begin{displaymath}
m_{\lambda}(M) = \sum_{\mu \le \lambda} m_{\lambda}(L_{\mu}) [M:L_{\mu}].
\end{displaymath}
\end{proposition}

\begin{proof}
  Let $\alpha(M)=m_{\lambda}(M)$ and $\beta(M)=\sum_{\mu \le \lambda} m_{\lambda}(L_{\mu}) [M:L_{\mu}]$. We must show $\alpha(M)=\beta(M)$ for all $M$. The equality clearly holds if $M$ is simple (since if $\lambda$ appears in $L_{\mu}$ then $\mu \le \lambda$), or if $m_{\lambda}(M)=0$ (since we must then have $m_{\lambda}(L_{\mu})=0$ if $L_{\mu}$ is a constituent of $M$).
  
  Suppose now that $M$ is pointwise finite. Let $0=F_0 \subset \cdots \subset F_r=M$ be a filtration as in Proposition~\ref{prop:nice-filt} with $\Xi=\{\lambda\}$. Since $\alpha(F_i/F_{i-1})=\beta(F_i/F_{i-1})$ for each $i$, and both $\alpha$ and $\beta$ are additive in short exact sequences, it follows that $\alpha(M)=\beta(M)$.

  Finally, let $M$ be arbitrary. Write $M=\bigcup_{i \in I} M_i$ where $\{M_i\}_{i \in I}$ is a directed family of pointwise finite submodules of $M$. Then $\alpha(M_i)=\beta(M_i)$ for each $i$. Since $\alpha(M)=\sup \alpha(M_i)$ and $\beta(M)=\sup \beta(M_i)$, it follows that $\alpha(M)=\beta(M)$. This completes the proof. 
\end{proof}

\begin{corollary}
Let $M$ be a $\fG$-module. Then $M$ is pointwise finite if and only if $[M:L_{\lambda}]$ is finite for all $\lambda$.
\end{corollary}

\begin{proof}
If $M$ is pointwise finite then $m_{\lambda}(M)$ is finite, and so $[M:L_{\lambda}]$ is too, since it is bounded above by $m_{\lambda}(M)$. Conversely, if $[M:L_{\lambda}]$ is finite for all $\lambda$ then the proposition shows that $m_{\lambda}(M)$ is finite for all $\lambda$, and so $M$ is pointwise finite.
\end{proof}

\begin{corollary}
Let $M$ be a $\fG$-module such that $[M:L_{\lambda}]=0$ for all $\lambda$. Then $M=0$.
\end{corollary}

\begin{proof}
The proposition shows that $m_{\lambda}(M)=0$ for all $\lambda$, which implies $M=0$.
\end{proof}

\subsection{Formal characters}

Let $\bQ \lbb \Lambda \rbb$ denote the set of all formal sums $\sum_{\lambda \in \Lambda} a(\lambda) [\lambda]$ where $a(\lambda)$ is a rational number. Given a pointwise finite $\fM$-module $M$, we define its {\bf formal character}, denoted $\Theta_M$, to be the element of $\bQ \lbb \Lambda \rbb$ given by $\sum_{\lambda \in \Lambda} m_{\lambda}(M) [\lambda]$. We define the formal character of a pointwise finite $\fU$-, $\fD$-, or $\fG$-module to be the formal character of its restriction to $\fM$.

\begin{proposition} \label{prop:std-char}
We have the following:
\begin{enumerate}
\item Let $\{M_{\lambda}\}_{\lambda \in \Lambda}$ be a family of pointwise finite modules such that $m_{\lambda}(M_{\lambda}) \ne 0$ and $m_{\mu}(M_{\lambda})$ is non-zero only for $\mu \ge \lambda$. Then the formal characters $\Theta_{M_{\lambda}}$ for $\lambda \in \Lambda$ are linearly independent.
\item The formal characters of the standard modules are linearly independent.
\item The formal characters of the simple modules are linearly independent.
\end{enumerate}
\end{proposition}

\begin{proof}
(a) Suppose $\sum_\lambda c_\lambda \Theta_{M_\lambda}=0$ is a non-trivial linear dependence. Let $\lambda$ be minimal with $c_{\lambda} \ne 0$. If $\mu<\lambda$ then $c_{\mu}=0$ and if $\mu \nlet \lambda$ then $m_{\lambda}(M_{\mu})=0$. Thus the coefficient of $[\lambda]$ in the sum is $c_{\lambda} m_{\lambda}(M_{\lambda}) \ne 0$, a contradiction. (b) and (c) follow immediately from (a).
\end{proof}

\begin{proposition}
Let $M$ and $N$ be pointwise finite $\fG$-modules. Then $\Theta_M=\Theta_N$ if and only if $[M:L_{\lambda}]=[N:L_{\lambda}]$ for all $\lambda$.
\end{proposition}

\begin{proof}
Clearly, if $[M:L_{\lambda}]=[N:L_{\lambda}]$ for all $\lambda$ then $m_{\lambda}(M)=m_{\lambda}(N)$ for all $\lambda$ (even without pointwise finiteness) by Proposition~\ref{prop:mult-formula}, and so $\Theta_M=\Theta_N$. Conversely, suppose that $\Theta_M=\Theta_N$, i.e., $m_{\lambda}(M)=m_{\lambda}(N)$ for all $\lambda$. We show that $[M:L_{\lambda}]=[N:L_{\lambda}]$ by induction on $\lambda$. Thus let $\lambda$ be given, and suppose that $[M:L_{\mu}]=[N:L_{\mu}]$ for all $\mu<\lambda$. Using  Proposition~\ref{prop:mult-formula}, we have
\begin{align*}
m_{\lambda}(L_{\lambda})[M:L_{\lambda}]
&= m_{\lambda}(M)-\sum_{\mu<\lambda} m_{\lambda}(L_{\mu})[M:L_{\mu}] \\
&= m_{\lambda}(N)-\sum_{\mu<\lambda} m_{\lambda}(L_{\mu})[N:L_{\mu}] \\
&=m_{\lambda}(L_{\lambda})[N:L_{\lambda}]
\end{align*}
Note that pointwise finiteness is crucial here to ensure that all quantities are finite so that we can subtract. Since $m_{\lambda}(L_{\lambda})=1$, we thus find $[M:L_{\lambda}]=[N:L_{\lambda}]$, as desired.
\end{proof}

\subsection{Standard filtrations}

Let $M$ be a $\fG$-module. A {\bf standard filtration} on $M$ is a filtration $0=F_0 \subset \cdots \subset F_r=M$ such that for each $1 \le i \le r$ we have $F_i/F_{i-1} \cong \Delta_{\mu(i)}$ for some $\mu(i) \in \Lambda$. We say that $M$ is {\bf semistandard} if it admits a standard filtration. Co-standard filtrations and co-semistandard objects are defined analogously. Since the two notions are dual, we concentrate on semistandard objects.

\begin{proposition} \label{prop:sstd}
Let $M$ be a finitely generated $\fG$-module. The following are equivalent:
\begin{enumerate}
\item $M$ is semistandard.
\item $i^*M$ is a projective $\fU$-module.
\item $\Ext^1_{\fG}(M, \nabla_{\lambda})=0$ for all $\lambda$.
\item $\Ext^r_{\fG}(M, N)=0$ for all co-semistandard $N$ and all $r \ge 1$.
\end{enumerate}
\end{proposition}

\begin{proof}
(a) $\Rightarrow$ (b).
Since the standard modules are projective as $\fU$-modules, any semistandard module is also projective as a $\fU$-module.

(b) $\Rightarrow$ (a).
Suppose that $i^*M$ is projective. Since $i^*M$ is finitely generated by Proposition~\ref{prop:ij-finite}, it is a direct sum of finitely many indecomposable projectives by Proposition~\ref{prop:decomp} and the number of summands depends only on $i^*M$. We prove by induction on this number that $M$ is semistandard. If the number is 0, there is nothing to show. Otherwise, let $\lambda$ be a minimal weight occurring in $M$, and let $\phi \colon \Delta_{\lambda} \to M$ be the corresponding map. Since $\Delta_{\lambda}$ and $M$ are projective as $\fU$-modules, it follows from Proposition~\ref{prop:low-proj} that $\phi$ is injective and $i^*\coker(\phi)$ is projective as a $\fU$-module. By uniqueness of decompositions, we have $i^*M \cong i^* \Delta_\lambda \oplus i^*\coker(\phi)$, and so $i^*\coker(\phi)$ has fewer indecomposable summands than $M$. By induction, $\coker(\phi)$ is semistandard, and thus $M$ is semistandard.

(b) $\Rightarrow$ (d).
Suppose that $i^*M$ is projective. Then $\Ext^r_{\fG}(M, \nabla_{\lambda})=\Ext^r_{\fU}(i^*M, S_{\lambda})$ for all $r$ by Proposition~\ref{prop:ext-compare}, and this vanishes for $r>0$ since $i^*M$ is projective by Proposition~\ref{prop:tri-pushfwd}(c). Thus, by d\'evissage, $\Ext^r_{\fG}(M, N)=0$ for any co-semistandard $N$ and $r>0$.

(d) $\Rightarrow$ (c). Obvious.

(c) $\Rightarrow$ (b). Suppose that $\Ext^1_{\fG}(M, \nabla_{\lambda})=0$ for all $\lambda$. Then, by Proposition~\ref{prop:ext-compare}, we have $\Ext^1_{\fU}(i^*M, S_{\lambda})=0$ for all $\lambda$. By Proposition~\ref{prop:min-res}, it follows that $i^*M$ is projective.
\end{proof}

\begin{corollary} \label{cor:semistd-summand}
Any summand of a semistandard $\fG$-module is semistandard.
\end{corollary}

Suppose that $M$ is semistandard, and let $0=F_0 \subset \cdots \subset F_r = M$ be a standard filtration. For $\mu \in \Lambda$, let $n(\mu)$ be the number of indices $i$ for which $F_i/F_{i-1}$ is isomorphic to $\Delta_{\mu}$. Then $\Theta_M = \sum_{\mu} n(\mu) \Theta_{\Delta_{\mu}}$. Since the formal characters of standard modules are linearly independent (Proposition~\ref{prop:std-char}), it follows that one can recover $n(\mu)$ from $\Theta_M$; the quantity $n(\mu)$ is therefore independent of the standard filtration. We define the {\bf multiplicity} of $\Delta_{\mu}$ in $M$, denoted $[M:\Delta_{\mu}]$, to be $n(\mu)$. We similarly define the multiplicity of $\nabla_{\mu}$ in a co-semistandard module. The following proposition gives a useful way of computing these multiplicities.

\begin{proposition} \label{prop:Delta-mult}
Let $M$ be a semistandard $\fG$-module. Then
\begin{displaymath}
[M:\Delta_{\lambda}] = \dim \Hom_{\fG}(M, \nabla_{\lambda})
\end{displaymath}
\end{proposition}

\begin{proof}
If $M=0$ there is nothing to prove, and if $M$ is standard the result follows from Proposition~\ref{prop:hom-std-costd}. Otherwise, we can find a short exact sequence
\begin{displaymath}
0 \to N \to M \to N' \to 0
\end{displaymath}
such that $N$ and $N'$ have standard filtrations of shorter length than $M$. We have an exact sequence
\begin{displaymath}
0 \to \Hom_{\fG}(N', \nabla_{\lambda}) \to \Hom_{\fG}(M, \nabla_{\lambda}) \to \Hom_{\fG}(N, \nabla_{\lambda}) \to \Ext^1_{\fG}(N', \nabla_{\lambda})
\end{displaymath}
Since this $\Ext^1$ group vanishes by Proposition~\ref{prop:sstd}, we conclude
\begin{displaymath}
\dim \Hom_{\fG}(M, \nabla_{\lambda}) = \dim \Hom_{\fG}(N, \nabla_{\lambda}) + \dim \Hom_{\fG}(N', \nabla_{\lambda}).
\end{displaymath}
Of course, we also have
\begin{displaymath}
[M:\Delta_{\lambda}] = [N:\Delta_{\lambda}] + [N':\Delta_{\lambda}].
\end{displaymath}
Since $[N:\Delta_{\lambda}]=\dim \Hom_{\fG}(N, \nabla_{\lambda})$ by induction on the length of standard filtration, and similarly for $N'$, we see that the same holds for $M$.
\end{proof}

\begin{proposition} \label{prop:jsstd}
Let $M$ be a finite length $\fD$-module. Then $j_!M$ is semistandard and $[j_!M:\Delta_{\lambda}]=m_{\lambda}(M)$ for all $\lambda$.
\end{proposition}

\begin{proof}
Let $0=F_0 \subset \cdots \subset F_r=M$ be a filtration such that $F_i/F_{i-1} \cong S_{\mu(i)}$ for some weight $\mu(i)$. Since $j_!$ is exact, it follows that $0=j_!F_0 \subset \cdots \subset j_!F_r=j_!M$ is a filtration with $j_!F_i/j_!F_{i-1}=j_!S_{\mu(i)}=\Delta_{\mu(i)}$. The result follows.
\end{proof}

\subsection{Projectives modules}

Recall that $k \colon \fM \to \fG$ is the inclusion functor. Define
\[
  \wt{P}_{\lambda}=k_!(S_{\lambda}).
\]
The following proposition summarizes the basic properties of this module.

\begin{proposition} \label{prop:Ptilde}
We have the following:
\begin{enumerate}
\item $\wt{P}_{\lambda}$ is a finitely generated projective $\fG$-module.
\item We have $\dim \Hom_{\fG}(\wt{P}_{\lambda}, M)=m_{\lambda}(M)$ for any $\fG$-module $M$.
\item $\wt{P}_{\lambda}$ is semistandard; moreover, $[\wt{P}_{\lambda}:\Delta_{\lambda}]=1$ and $[\wt{P}_{\lambda}:\Delta_{\mu}]$ is non-zero only if $\mu \le \lambda$.
\item The formal characters of the $\wt{P}_{\lambda}$ are linearly independent.
\item Every (finitely generated) $\fG$-module is a quotient of a (finite) sum of $\wt{P}_{\lambda}$'s.
\end{enumerate}
\end{proposition}

\begin{proof}
(a) Since $S_{\lambda}$ is a finitely generated projective $\fM$-module and $k_!$ takes projectives to projectives and preserves finite generation, the claim follows.

(b) We have $\Hom_{\fG}(\wt{P}_{\lambda}, M)=\Hom_{\fM}(S_{\lambda}, k^*(M))$ by adjunction, the dimension of which is $m_{\lambda}(M)$ by definition.

(c) Since $k=j \circ i'$, we have $\wt{P}_{\lambda}=j_!(i'_!(S_{\fM,\lambda}))$. We claim that $m_{\lambda}(i'_!(S_{\fM,\lambda}))=1$, and $m_{\mu}(i'_!(S_{\fM,\lambda}))$ is non-zero only if $\mu \le \lambda$. By the downwards property, $S_{\fD,\lambda}$ is the top of $i'_!(S_{\fM,\lambda})$ and $i'_!(S_{\fM,\lambda})$ is~0 on objects that are not less than $\supp(\lambda)$. Hence by semi-simplicity, $\Hom_\fM(S_{\fM,\lambda}, (i')^* i'_!(S_{\fM,\lambda})) = \Hom_\fD(i'_!(S_{\fM,\lambda}), S_{\fD,\lambda}) = \Hom_\fM(S_{\fM,\lambda}, S_{\fM,\lambda})$, which is 1-dimensional. Thus the claim follows from Proposition~\ref{prop:jsstd}.

(d) This follows from (c) and the corresponding fact for standard modules.

(e) Let $M$ be a $\fG$-module. By adjunction, we have $\Hom_\fG(\wt{P}_\lambda, M) = \Hom_\fM(S_\lambda, k^*(M))$. Hence if $k^*(M) = \bigoplus_\lambda S_\lambda \otimes M_\lambda$, then we have a surjection $\bigoplus_\lambda \wt{P}_\lambda \otimes M_\lambda \to M$. If $M$ is finitely generated, then pick a finite generating set and a finite-dimensional $\fM$-submodule of $k^*M$ that contains it; write this submodule as $\bigoplus_\lambda S_\lambda \otimes M'_\lambda$. Then $\bigoplus_\lambda \wt{P}_\lambda \otimes M'_\lambda$ is a finite direct sum that surjects onto $M$.
\end{proof}

Write $\wt{P}_{\lambda}=P_1 \oplus \cdots \oplus P_r$, where each $P_i$ is an indecomposable projective $\fG$-module (Proposition~\ref{prop:decomp}). Since each $P_i$ is a summand of $\wt{P}_{\lambda}$, it is semistandard by Corollary~\ref{cor:semistd-summand}. We have
\begin{displaymath}
1 = [\wt{P}_{\lambda}:\Delta_{\lambda}] = [P_1:\Delta_{\lambda}] + \cdots + [P_r:\Delta_{\lambda}].
\end{displaymath}
It follows that $[P_i:\Delta_{\lambda}]$ is equal to~1 for exactly one value of $i$. We define $P_{\lambda}$ to be this summand $P_i$. Thus, to summarize, $P_{\lambda}$ is the unique (up to isomorphism) indecomposable summand of $\wt{P}_{\lambda}$ such that $[P_{\lambda}:\Delta_{\lambda}]=1$ (uniqueness is guaranteed by Proposition~\ref{prop:decomp}). The next proposition establishes some other important properties of these modules.

\begin{proposition} \label{prop:Plambda}
We have the following:
\begin{enumerate}
\item $P_{\lambda}$ is semistandard; moreover, $[P_{\lambda}:\Delta_{\lambda}]=1$ and $[P_{\lambda}:\Delta_{\mu}]$ is non-zero only for $\mu \le \lambda$.
\item $P_{\lambda}$ is the projective cover of $\Delta_{\lambda}$ and $L_{\lambda}$.
\item We have $\dim \Hom_{\fG}(P_{\lambda}, M) = [M:L_{\lambda}]$ for any $\fG$-module $M$.
\item The formal characters of the $P_{\lambda}$'s are linearly independent.
\item We have $\wt{P}_{\lambda} \cong \bigoplus_{\mu \le \lambda} P_{\mu}^{\oplus a(\mu)}$, where $a(\mu)=m_{\lambda}(L_{\mu})$ (note $a(\lambda)=1$).
\item Every (finitely generated) $\fG$-module is a quotient of a (finite) sum of $P_{\lambda}$'s.
\item Any finitely generated indecomposable projective $\fG$-module is isomorphic to some $P_{\lambda}$.
\end{enumerate}
\end{proposition}

\begin{proof}
(a) Since $P_{\lambda}$ is a summand of $\wt{P}_{\lambda}$, it is semistandard by Corollary~\ref{cor:semistd-summand}. We have $[P_{\lambda}:\Delta_{\lambda}]=1$ by definition. Since $[P_{\lambda}:\Delta_{\mu}] \le [\wt{P}_{\lambda}:\Delta_{\mu}]$, this can only be non-zero for $\mu \le \lambda$ by Proposition~\ref{prop:Ptilde}(c).

(b) Since $\Delta_\lambda$ is generated by $S_\lambda$ in degree $\supp(\lambda)$, we have a surjection $P_\lambda \to \Delta_\lambda$, and by composition with $\Delta_\lambda\to L_\lambda$, we also have a surjection $P_\lambda \to L_\lambda$. We note that $\End_{\fG}(P_{\lambda})$ is finite dimensional and has no non-trivial idempotents; it is thus local in the sense of \cite[\S 3]{krause} (see \cite[Corollary~19.19]{lam}). It therefore follows from \cite[Lemma 3.6]{krause} that $P_\lambda \to L_\lambda$ is a projective cover. This implies that $P_\lambda \to \Delta_\lambda$ is a projective cover as well.

(c) Put $\alpha(M)=\dim \Hom_{\fG}(P_{\lambda}, M)$ and $\beta(M)=[M:L_{\lambda}]$. Thus we wish to show $\alpha(M)=\beta(M)$ for all $M$. If $M=L_\lambda$, then both quantities are 1. If $M=L_\mu$ for $\mu \ne \lambda$, then $\beta(M)=0$. If $\alpha(M) \ne 0$, then there is a surjective map $P_\lambda \to L_\mu$. If $K$ is the kernel, then the image of $K$ in $L_\lambda$ is a proper submodule by (b) and hence 0, so this implies that we have a surjection $L_\mu = P_\lambda/K \to L_\lambda$, which is a contradiction, and hence $\alpha(L_\mu)=0$. Hence $\alpha(M)=\beta(M)$ if $M$ is simple.

If $\lambda$ does not occur in $M$, then $\alpha(M)=\beta(M)=0$.

Now, suppose $M$ is pointwise finite. Let $0=F_0 \subset \cdots \subset F_r=M$ be the filtration provided by Proposition~\ref{prop:nice-filt} with $\Xi=\{\lambda\}$. The subquotient $F_i/F_{i-1}$ is either simple or does not contain $\lambda$, and so we have $\alpha(F_i/F_{i-1})=\beta(F_i/F_{i-1})$. Since both $\alpha$ and $\beta$ are additive in short exact sequences, we conclude that $\alpha(M)=\beta(M)$. Finally, let $M$ be arbitrary and write $M=\varinjlim_i M_i$ where the $M_i$ range over the pointwise finite submodules of $M$. We have $\beta(M)=\sup \beta(M_i)$ and $\alpha(M)=\sup \alpha(M_i)$ since $\Hom(P_\lambda,M) = \varinjlim_i \Hom(P_\lambda,M_i)$ and the transition maps are injective. Since $\beta(M_i)=\alpha(M_i)$ for all $i$, it follows that $\alpha(M)=\beta(M)$, which completes the proof.

(d) This follows from (a) and the corresponding fact for standard modules.

(e) We first show that such a decomposition exists. We proceed by induction on $\lambda$. Write $\wt{P}_{\lambda}=P_{\lambda} \oplus Q$ for some projective $Q$. By Proposition~\ref{prop:Ptilde}(c), we have $[Q:\Delta_{\mu}]=0$ unless $\mu<\lambda$; note that $[\wt{P}_{\lambda}:\Delta_{\lambda}]=[P_{\lambda}:\Delta_{\lambda}]=1$, and so $[Q:\Delta_{\lambda}]=0$. It follows that $Q$ is a quotient (and therefore summand) of a sum of $\wt{P}_{\mu}$'s with $\mu<\lambda$. Thus, by the inductive hypothesis, $Q$ is a summand of a sum of such $P_{\mu}$'s, and thus (by Proposition~\ref{prop:decomp}), $Q$ is a sum of such $P_{\mu}$'s.

Now, write $\wt{P}_\lambda = \bigoplus_{\mu \le \lambda} P_\mu^{\oplus a(\mu)}$, which we can do by the previous paragraph. We have $\dim \Hom_\fG(P_\mu,L_\lambda) = \delta_{\mu,\lambda}$ by (c), and so
\begin{displaymath}
a(\mu) = \dim \Hom_\fG(\wt{P}_\lambda, L_\mu) = m_\lambda(L_\mu).
\end{displaymath}
where in the second step we used Proposition~\ref{prop:Ptilde}(b).

(f) This follows from (e) and the corresponding result for $\wt{P}_{\lambda}$.

(g) This follows from (f) and Proposition~\ref{prop:decomp}.
\end{proof}

\begin{proposition}[BGG reciprocity] \label{prop:bgg}
For $\lambda,\mu \in \Lambda$ we have
\begin{displaymath}
[P_{\lambda}:\Delta_{\mu}]=[\nabla_{\mu}:L_{\lambda}].
\end{displaymath}
If $\fG$ admits a transpose then these quantities coincide with $[\Delta_{\mu}:L_{\lambda}]$.
\end{proposition}

\begin{proof}
By Propositions~\ref{prop:Delta-mult} and \ref{prop:Plambda}(c), we have
\begin{displaymath}
[P_{\lambda}:\Delta_{\mu}] = \dim \Hom_{\fG}(P_{\lambda}, \nabla_{\mu}) = [\nabla_{\mu}:L_{\lambda}].
\end{displaymath}
If $\fG$ admits a transpose then $[\nabla_{\mu}:L_{\lambda}]=[\Delta_{\mu}:L_{\lambda}]$.
\end{proof}

\subsection{Injective modules} \label{ss:inj}

The discussion of the previous section applies equally well to injectives. We briefly summarize the main points. Put $\wt{I}_\lambda = k_*(S_\lambda)$. Write $\wt{I}_\lambda = I_1 \oplus \cdots \oplus I_r$ with each $I_i$ an indecomposable injective (Proposition~\ref{prop:decomp}). We have $[\wt{I}_{\lambda}:\nabla_{\lambda}]=1$, and so $[I_i:\nabla_\lambda]=1$ for exactly one value of $i$. We define $I_\lambda$ to be this summand. Note that $I_{\lambda^{\vee}} = P_{\lambda}^\vee$, since $P_{\lambda}^{\vee}$ is an indecomposable summand of $\wt{I}_{\lambda^{\vee}}=\wt{P}_{\lambda}^{\vee}$ containing $\nabla_{\lambda}$ with multiplicity one.

The following results are dual to those about $\wt{P}_\lambda$ and $P_\lambda$ above, so we omit the proofs.

\begin{proposition}
We have the following:
  \begin{enumerate}
  \item $I_\lambda$ is co-semistandard; $[I_\lambda:\nabla_\lambda]=1$ and $[I_\lambda:\nabla_\mu]$ is non-zero only for $\mu\le\lambda$.
  \item $I_\lambda$ is the injective hull of $\nabla_\lambda$ and $L_\lambda$.
  \item We have $\dim \Hom_\fG(M,I_\lambda) = [M:L_\lambda]$ for any $\fG$-module $M$.
  \item We have $\wt{I}_\lambda \cong \bigoplus_{\mu \le \lambda} I_\mu^{\oplus a(\mu)}$ where $a(\mu) = m_\lambda(L_\mu)$.
  \item Every $\fG$-module is a submodule of a direct product of $I_\lambda$'s.
  \end{enumerate}
\end{proposition}

\subsection{Tensor products} \label{ss:trimonoid}

Let $\fG$ and $\fG'$ be triangular categories. Then $(\fU \times \fU', \fD \times \fD')$ is a triangular structure on $\fG \times \fG'$, and we regard $\fG \times \fG'$ as a triangular category in this way. (Here $\times$ denotes the product in the sense of linear categories.) A {\bf monoidal triangular category} is a triangular category $\fG$ equipped with a monoidal operation $\amalg$ such that the functor $\amalg \colon \fG \times \fG \to \fG$ is triangular and all structure isomorphisms (such as the associator) belong to $\fM$.

Suppose that $\fG$ is a symmetric monoidal triangular category, with monoidal operation $\amalg$. Then $\amalg$ induces symmetric monoidal structures on $\fU$, $\fD$, and $\fM$, and the functors $i$, $i'$, $j$, $j'$, and $k$ are monoidal. The monoidal operation induces tensor products $\otimes_{\fG}$, $\otimes_{\fU}$, $\otimes_{\fD}$, and $\otimes_{\fM}$ on the module categories $\Mod_{\fG}$, $\Mod_{\fU}$, $\Mod_{\fD}$, and $\Mod_{\fM}$ as in \S\ref{ss:tensor}, and the various $(-)_!$ functors are monoidal with respect to these products.

\begin{proposition}
Let $M$ and $N$ be $\fD$-modules. Then we have a natural isomorphism $\Tor^{\fG}_p(j_!(M),j_!(N))=j_!(\Tor^{\fD}_p(M,N))$ for all $p$.
\end{proposition}

\begin{proof}
Let $P_{\bullet} \to M$ be a projective resolution of $\fD$-modules. Since $j_!$ is exact and takes projectives to projectives, we see that $j_!(P_{\bullet}) \to j_!(M)$ is a projective resolution of $\fG$-modules. Thus $\Tor^{\fG}_{\bullet}(j_!(M), j_!(N))$ is computed by $j_!(N) \otimes_{\fG} j_!(P_{\bullet})$. Since $j_!$ is monoidal, this is isomorphic to $j_!(N \otimes_{\fD} P_{\bullet})$. As $N \otimes_{\fD} P_{\bullet}$ is the complex computing $\Tor^{\fD}_{\bullet}(M, N)$ and $j_!$ is exact, the result follows.
\end{proof}

\begin{proposition} \label{prop:ss-tensor}
Suppose $\otimes$ is exact on $\Mod_{\fD}$, and let $M$ and $N$ be semistandard $\fG$-modules. Then $M \otimes_{\fG} N$ is semistandard and $\Tor^{\fG}_p(M, N)=0$ for $p>0$.
\end{proposition}

\begin{proof}
First suppose that $M=\Delta_{\lambda}$ and $N=\Delta_{\mu}$ are standard. Thus $M=j_!(S_{\lambda})$ and $N=j_!(S_{\mu})$, and so $M \otimes_{\fG} N=j_!(S_{\lambda} \otimes_{\fD} S_{\mu})$ since $j_!$ is monoidal. We therefore find that $M \otimes_{\fG} N$ is semistandard by Proposition~\ref{prop:jsstd}. By the previous lemma, we find that $\Tor^{\fG}_p(M,N)=\Tor^{\fD}_p(S_{\lambda},S_{\mu})=0$ for $p>0$, since $\otimes_{\fD}$ is exact.

We now treat the general case by induction on the sum of the lengths of the standard filtrations of $M$ and $N$. Suppose $M$ is non-zero and choose a short exact sequence
\begin{displaymath}
0 \to \Delta_{\lambda} \to M \to M' \to 0
\end{displaymath}
with $M'$ semistandard. Since $\Tor_1^{\fG}(M', N)=0$ by the inductive hypothesis, we find that the sequence
\begin{displaymath}
0 \to \Delta_{\lambda} \otimes_{\fG} N \to M \otimes_{\fG} N \to M' \otimes_{\fG} N \to 0
\end{displaymath}
is exact. Since the outer terms are semistandard by the inductive hypothesis, so is the middle term. We also have exact sequences
\begin{displaymath}
\Tor_p^{\fG}(\Delta_{\lambda}, N) \to \Tor_p^{\fG}(M, N) \to \Tor_p^{\fG}(M', N)
\end{displaymath}
for all $p$. The inductive hypothesis shows that the outer terms vanish, and so the middle one does as well.
\end{proof}

\subsection{A criterion for (T3)}

We now give a criterion to simplify the task of verifying the axiom (T3). Let $\fG$ be a $\bk$-linear category with wide subcategories $\fU$ and $\fD$ satisfying (T0), (T1), and (T2). Suppose that for all objects $x,y \in \fG$ we have a subset of $\Hom_{\fG}(x,y)$, whose elements we refer to as {\bf distinguished}, such that the following conditions hold:
\begin{enumerate}
\item The distinguished elements form a $\bk$-basis of $\Hom_{\fG}(x,y)$ for all $x$ and $y$.
\item The basis of distinguished element is adapted to $\Hom_{\fU}(x,y)$ and $\Hom_{\fD}(x,y)$, that is, these spaces are spanned by the distinguished elements they contain.
\item If $\beta \colon x \to y$ is a distinguished morphism in $\fD$ and $\alpha \colon y \to z$ is a distinguished morphism in $\fU$ then $\alpha \circ \beta$ is a distinguished morphism.
\item Any distinguished morphism $\phi \colon x \to z$ can be factored as $\alpha \circ \beta$, where $\beta \colon x \to y$ is a distinguished morphism in $\fD$, and $\alpha \colon y \to z$ is a distinguished morphism in $\fU$. Moreover, this factorization is unique up to $\fM$-isomorphism, in the sense that if $\phi = \alpha' \circ \beta'$, with $\beta' \colon x \to y'$ a distinguished morphism in $\fD$ and $\alpha' \colon y' \to z$ a distinguished morphism in $\fU$, then there exists an isomorphism $i \colon y \to y'$ in $\fM$ such that $\beta'=i \circ \beta$ and $\alpha'=\alpha \circ i^{-1}$.
\end{enumerate}

\begin{proposition} \label{prop:tri-crit}
In the above setting, axiom (T3) holds, and so $(\fU, \fD)$ is a triangular structure.
\end{proposition}

\begin{proof}
Without loss of generality, assume that $\fG$ is skeletal (isomorphic objects are equal). Pick $x,z \in \fG$ and consider the composition map
\begin{displaymath}
\psi \colon \bigoplus_{y \in \fG} \Hom_{\fU}(y, z) \otimes_{\End_{\fM}(y)} \Hom_{\fD}(x, y) \to \Hom_{\fG}(x,z)
\end{displaymath}
Let $\gamma_1, \ldots, \gamma_n$ be the distinguished elements of $\Hom_{\fG}(x,z)$, which form a basis by (a). For each $i$, choose a factorization $\gamma_i=\alpha_i \circ \beta_i$ where $\alpha_i \colon y_i \to z$ and $\beta_i \colon x \to y_i$ are distinguished morphisms in $\fU$ and $\fD$; this is possible by (d).

We claim that the elements $\alpha_i \otimes \beta_i$ span the domain of $\psi$. By (b), this space is spanned by elements of the form $\alpha \otimes \beta$ where $\alpha \colon y \to z$ and $\beta \colon x \to y$ are distinguished morphisms in $\fU$ and $\fD$, so it suffices to express these elements in the desired form. By (c), the composition $\alpha \circ \beta$ is distinguished, and therefore equal to $\gamma_i$ for some $i$. We thus have $\gamma_i=\alpha \circ \beta = \alpha_i \circ \beta_i$. By the uniqueness property in (d), we have $y=y_i$ and there is an automorphism $\sigma$ of $y$ in $\fM$ such that $\beta=\sigma \circ \beta_i$ and $\alpha=\alpha_i \circ \sigma^{-1}$. We thus have $\alpha \otimes \beta=\alpha_i\sigma^{-1} \otimes \sigma\beta_i=\alpha_i \otimes \beta_i$, since the tensor product is taken over $\End_{\fM}(y)$. This establishes the claim.

Since the elements $\alpha_i \otimes \beta_i$ span the domain of $\psi$ and map bijectively to a basis, it follows that $\psi$ is an isomorphism. Thus (T3) holds.
\end{proof}

\subsection{A generalization}

Given a triangular category $\fG$, one gets abelian categories $\Mod_{\fG}$, $\Mod_{\fU}$, $\Mod_{\fD}$, $\Mod_{\fM}$, and various functors between them coming from $i$, $j$, $i'$, and $j'$. In fact, one can abstract this situation, as follows. Consider a commutative (up to isomorphism) diagram of abelian categories
\begin{displaymath}
\xymatrix@C=4em{
\cA_m & \cA_u \ar[l]_-{(j')^*} \\
\cA_d \ar[u]^{(i')^*} & \cA \ar[l]_-{j^*} \ar[u]_{i^*} }
\end{displaymath}
We require a number of axioms; we list only the most important ones:
\begin{itemize}
\item The functors $i^*$, $j^*$, $(i')^*$, $(j')^*$ are faithful, continuous, cocontinuous, and admit both left and right adjoints.
\item The category $\cA_m$ is semi-simple, and $(i')^*$ and $(j')^*$ identify the semi-simple subcategories of $\cA_d$ and $\cA_u$ with $\cA_m$.
\item The simple objects of $\cA_m$ can be ordered so that all extensions in $\cA_u$ are upwards and all extensions in $\cA_d$ are downwards.
\item The base change isomorphisms of Proposition~\ref{prop:tri-pushfwd}(a,b) hold.
\end{itemize}
Nearly all constructions and results in this section can be carried out in this framework. This framework is in some ways more natural, since it provides exactly what is needed for the key constructions and results. However, we decided to work in the slightly more special setting of triangular categories since they are less technical and cover the examples of interest.

\section{The Brauer category} \label{sec:brauer}

\subsection{Motivation} \label{ss:brauer-motiv}

Before defining the Brauer category, we attempt to explain where it comes from. Let $V$ be a finite dimensional complex vector space equipped with a non-degenerate symmetric bilinear form. The idea of the Brauer category is to record all the obvious maps between tensor powers of $V$ that commute with the action of the orthogonal group $\bO(V)$. For clarity, we work with tensor powers where the power is a finite set: if $S$ is a finite set of cardinality $n$ then $V^{\otimes S}$ is isomorphic to $V^{\otimes n}$, but the tensor factors are indexed by $S$; when $S=[n]$, we identify $V^{\otimes S}$ with $V^{\otimes n}$. There are three fundamental examples of maps between tensor powers:
\begin{itemize}
\item Given a bijection $\phi \colon S \to T$ there is an induced isomorphism $\phi_* \colon V^{\otimes S} \to V^{\otimes T}$.
\item The map $a \colon V^{\otimes 2} \to \bC$ given by the form. (Note $\bC=V^{\otimes 0}$.)
\item The map $b \colon \bC \to V^{\otimes 2}$ that is dual to $a$.
\end{itemize}
There are three ways that we can build new maps from existing ones:
\begin{itemize}
\item Given maps $V^{\otimes S} \to V^{\otimes T}$ and $V^{\otimes T} \to V^{\otimes U}$, we can form their composition to obtain a map $V^{\otimes S} \to V^{\otimes U}$.
\item Given maps $V^{\otimes S_1} \to V^{\otimes T_1}$ and $V^{\otimes S_2} \to V^{\otimes T_2}$, we can form their tensor product to obtain a map $V^{\otimes (S_1 \amalg S_2)} \to V^{\otimes (T_1 \amalg T_2)}$.
\item Given two maps $V^{\otimes S} \to V^{\otimes T}$, we can form a linear combination.
\end{itemize}
Starting with the fundamental examples and applying these constructions, we can create an endless supply of maps.

There is a convenient way of recording the maps produced by the above constructions using certain diagrams. A {\bf Brauer diagram} from $S$ to $T$ is a perfect matching on $S \amalg T$; that is, it is a partition of the set $S \amalg T$ into disjoint subsets of size two. We picture such a diagram by regarding the elements of $S$ and $T$ as forming two rows of vertices, with the $S$ row below the $T$ row. We thus have three types of edges: \emph{vertical edges} contain one vertex in $S$ and one in $T$; \emph{horizontal edges in $S$} contain two vertices in $S$; and \emph{horizontal edges in $T$} contain two vertices in $T$. In essence, these three types of edges correspond to the three fundamental maps between tensor powers.

Given a Brauer diagram $\alpha$ from $S$ to $T$, we associate to it a linear map $\alpha_* \colon V^{\otimes S} \to V^{\otimes T}$ by applying the map $a$ along horizontal edges in the $S$ row, the map $b$ along horizontal edges in the $T$ row, and identifying the remaining tensor factors along vertical edges. To be more precise, let $S'$ (resp.\ $T'$) be the vertices of $S$ (resp.\ $T$) contained in a vertical edge. The vertical edges of $\alpha$ define a bijection $\alpha' \colon S' \to T'$. The map $\alpha_*$ is then the composite
\begin{displaymath}
\xymatrix{
V^{\otimes S} \ar[r]^{f} & V^{\otimes S'} \ar[r]^{g} & V^{\otimes T'} \ar[r]^{h} & V^{\otimes T} }
\end{displaymath}
where $f$ is the tensor product of the maps $a \colon V^{\otimes \{x,y\}} \to \bC$ over the horizontal edges $(x,y)$ in $S$, $g$ is simply $\alpha'_*$, and $h$ is the tensor product of the maps $b \colon \bC \to V^{\otimes \{x,y\}}$ over the horizontal edges $(x,y)$ in $T$.

\begin{example}
Suppose $\alpha$ is the following diagram from $[3]$ to $[5]$
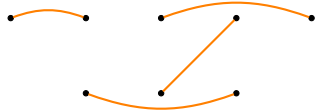
\begin{center}
\begin{tikzpicture}
\node [node] at (0,1) (a1) {};
\node [node] at (1,1) (a2) {};
\node [node] at (2,1) (a3) {};
\node [node] at (3,1) (a4) {};
\node [node] at (4,1) (a5) {};
\node [node] at (1,0) (b1) {};
\node [node] at (2,0) (b2) {};
\node [node] at (3,0) (b3) {};
\draw[thick, orange] (b1) to[out=-20,in=-160] (b3);
\draw[thick, orange] (b2)--(a4);
\draw[thick, orange] (a3) to[out=20,in=160] (a5);
\draw[thick, orange] (a1) to[out=20,in=160] (a2);
\end{tikzpicture}
\end{center}
Let $e_1, \ldots, e_n$ be an orthonormal basis of $V$. Then $\alpha_*$ is the map $V^{\otimes 3} \to V^{\otimes 5}$ given by
\begin{displaymath}
\alpha_*(e_i \otimes e_j \otimes e_k) = \delta_{i,k} \sum_{r,s=1}^n e_r \otimes e_r \otimes e_s \otimes e_j \otimes e_s. \qedhere
\end{displaymath}
\end{example}

It turns out that the composition of two maps associated to diagrams is always a scalar multiple of another such map. We now explain exactly how this works. Thus suppose that $\alpha$ is a Brauer diagram from $S$ to $T$, and $\beta$ is one from $T$ to $U$. Let $\beta \cup \alpha$ denote the graph on $S \amalg T \amalg U$ whose edge set is the disjoint union of the edge sets of $\alpha$ and $\beta$. Note that if two vertices of $T$ are joined in both $\alpha$ and $\beta$ then they will be joined by a double edge in $\beta \cup \alpha$. We picture $\beta \cup \alpha$ by simply stacking $\beta$ atop $\alpha$. We let $\beta \bullet \alpha$ be the Brauer diagram from $S$ to $U$ in which there is an edge from $x$ to $y$ if there is a path from $x$ to $y$ in $\beta \cup \alpha$. We also define $c(\beta,\alpha)$ to be the number of cycles in $\beta \cup \alpha$. Note that any cycle can only use vertices in $T$. We then have the fundamental formula:
\begin{equation} \label{eq:brauer-comp}
\beta_* \circ \alpha_* = (\dim{V})^{c(\beta,\alpha)} \cdot (\beta \bullet \alpha)_*.
\end{equation}
This can be proved be a straightforward computation. As a consequence, we see that the linear combinations of maps associated to diagrams account for all the maps obtained by starting with the three fundamental maps and applying the three basic methods of creating new maps.

\begin{example}
We highlight one simple but important example of composition. Let $\alpha$ be the unique Brauer diagram from $[0]$ to $[2]$ and let $\beta$ be the unique diagram from $[2]$ to $[0]$. Then $\beta \bullet \alpha$ is the empty diagram from $[0]$ to itself and $c(\beta,\alpha)=1$; also note that $\alpha_*=b$ and $\beta_*=a$. Thus \eqref{eq:brauer-comp} amounts to the fact that the composition
\begin{displaymath}
\xymatrix{
\bC \ar[r]^b & \bC^{\otimes 2} \ar[r]^a & \bC }
\end{displaymath}
is multiplication by $\dim(V)$. We can see this directly as follows. Let $e_1, \ldots, e_n$ be an orthonormal basis for $V$. Then $b$ is the map $1 \mapsto \sum_{i=1}^n e_i \otimes e_i$, while $a$ is the map $e_i \otimes e_j \mapsto \delta_{i,j}$. Thus $a$ maps each term of $b(1)$ to~1, and since there are $n$ terms we find $a(b(1))=n=\dim(V)$.
\end{example}

\begin{example}
We now give an example illustrating a more complicated composition. Let $\beta \colon [7] \to [5]$ be the diagram
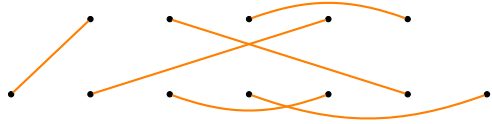
\begin{center}
\begin{tikzpicture}
\node [node] at (0,0) (a1) {};
\node [node, right of=a1] (a2) {};
\node [node, right of=a2] (a3) {};
\node [node, right of=a3] (a4) {};
\node [node, right of=a4] (a5) {};
\node [node, right of=a5] (a6) {};
\node [node, right of=a6] (a7) {};
\node [node, white] at (0,1) (b1) {};
\node [node, right of=b1] (b2) {};
\node [node, right of=b2] (b3) {};
\node [node, right of=b3] (b4) {};
\node [node, right of=b4] (b5) {};
\node [node, right of=b5] (b6) {};
\draw[thick, orange] (a1)--(b2);
\draw[thick, orange] (a6)--(b3);
\draw[thick, orange] (a2)--(b5);
\draw[thick, orange] (a4) to[out=-20,in=-160] (a7);
\draw[thick, orange] (a3) to[out=-20,in=-160] (a5);
\draw[thick, orange] (b4) to[out=20,in=160] (b6);
\end{tikzpicture}
\end{center}
and let $\alpha \colon [3] \to [7]$ be
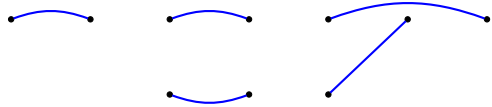
\begin{center}
\begin{tikzpicture}
\node [node] at (0,1) (a1) {};
\node [node, right of=a1] (a2) {};
\node [node, right of=a2] (a3) {};
\node [node, right of=a3] (a4) {};
\node [node, right of=a4] (a5) {};
\node [node, right of=a5] (a6) {};
\node [node, right of=a6] (a7) {};
\node [node, white] at (0,0) (c1) {};
\node [node, white, right of=c1] (c2) {};
\node [node, right of=c2] (c3) {};
\node [node, right of=c3] (c4) {};
\node [node, right of=c4] (c5) {};
\draw[thick, blue] (a1) to[out=20,in=160] (a2);
\draw[thick, blue] (a3) to[out=20,in=160] (a4);
\draw[thick, blue] (a5) to[out=20,in=160] (a7);
\draw[thick, blue] (c3) to[out=-20,in=-160] (c4);
\draw[thick, blue] (c5)--(a6);
\end{tikzpicture}
\end{center}
The graph $\beta \cup \alpha$ is then
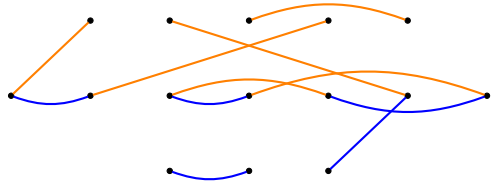
\begin{center}
\begin{tikzpicture}
\node [node, white] at (0,2) (b1) {};
\node [node, right of=b1] (b2) {};
\node [node, right of=b2] (b3) {};
\node [node, right of=b3] (b4) {};
\node [node, right of=b4] (b5) {};
\node [node, right of=b5] (b6) {};
\node [node] at (0,1) (a1) {};
\node [node, right of=a1] (a2) {};
\node [node, right of=a2] (a3) {};
\node [node, right of=a3] (a4) {};
\node [node, right of=a4] (a5) {};
\node [node, right of=a5] (a6) {};
\node [node, right of=a6] (a7) {};
\node [node, white] at (0,0) (c1) {};
\node [node, white, right of=c1] (c2) {};
\node [node, right of=c2] (c3) {};
\node [node, right of=c3] (c4) {};
\node [node, right of=c4] (c5) {};
\draw[thick, orange] (a1)--(b2);
\draw[thick, orange] (a6)--(b3);
\draw[thick, orange] (a2)--(b5);
\draw[thick, orange] (a4) to[out=20,in=160] (a7);
\draw[thick, orange] (a3) to[out=20,in=160] (a5);
\draw[thick, orange] (b4) to[out=20,in=160] (b6);
\draw[thick, blue] (a1) to[out=-20,in=-160] (a2);
\draw[thick, blue] (a3) to[out=-20,in=-160] (a4);
\draw[thick, blue] (a5) to[out=-20,in=-160] (a7);
\draw[thick, blue] (c3) to[out=-20,in=-160] (c4);
\draw[thick, blue] (c5)--(a6);
\end{tikzpicture}
\end{center}
This graph has a unique cycle, namely the 4-cycle on the vertices $\{3,4,5,7\}$ in the middle row, and so $c(\beta,\alpha)=1$. The graph $\gamma = \beta \bullet \alpha$ is
\begin{center}
\begin{tikzpicture}
\node [node] at (0,1) (b2) {};
\node [node, right of=b2] (b3) {};
\node [node, right of=b3] (b4) {};
\node [node, right of=b4] (b5) {};
\node [node, right of=b5] (b6) {};
\node [node,white] at (0,0) (c2) {};
\node [node, right of=c2] (c3) {};
\node [node, right of=c3] (c4) {};
\node [node, right of=c4] (c5) {};
\draw[thick, green!50!gray] (b4) to[out=20,in=160] (b6);
\draw[thick, green!50!gray] (b2) to[out=20,in=160] (b5);
\draw[thick, green!50!gray] (c3) to[out=-20,in=-160] (c4);
\draw[thick, green!50!gray] (c5)--(b3);
\end{tikzpicture}
\end{center}
Thus \eqref{eq:brauer-comp} becomes $\beta_* \circ \alpha_* = \dim(V) \cdot \gamma_*$.
\end{example}

\subsection{Definition}

Fix a field (or even a commutative ring) $\bk$ and let $\delta \in \bk$. We define the {\bf Brauer category} over $\bk$ with parameter $\delta$, denoted $\fG$, as follows. The objects of $\fG$ are finite sets. The set $\Hom_{\fG}(S,T)$ is the free $\bk$-module on the Brauer diagrams from $S$ to $T$. For Brauer diagrams $\alpha \in \Hom_{\fG}(S,T)$ and $\beta \in \Hom_{\fG}(T,U)$, the composite morphism $\beta \circ \alpha$ is defined to be $\delta^{c(\beta,\alpha)} \cdot (\beta \bullet \alpha)$. Composition for general morphisms is defined by linearity.

The endomorphism algebras in $\fG$ are exactly the classical Brauer algebras introduced by Brauer in \cite{brauer}, and which have been studied extensively since. The Brauer category itself appears in \cite{LZ-brauer}, \cite[\S 9]{deligne}, \cite[\S 3.6]{martin}, \cite[\S 2.1]{coulembier4}, and \cite{rui-song}.

A bijection $S \to T$ can be regarded as a Brauer diagram with only vertical edges, and thus defines an element of $\Hom_{\fG}(S, T)$. This construction is compatible with composition. In what follows, we tacitly identify a bijection with its corresponding Brauer diagram. It is not difficult to show that two finite sets $S$ and $T$ are isomorphic in $\fG$ if and only if they have the same cardinality.

We make one more simple observation about $\fG$ here. A set can admit a perfect matching only if it has even cardinality. Thus if there is a Brauer diagram from $S$ to $T$, i.e., if $\Hom_{\fG}(S,T)$ is non-zero, then $S$ and $T$ must have the same parity. Hence $\fG$ is the disjoint union (in the sense of linear categories) of $\fG_{\rm even}$ and $\fG_{\rm odd}$, where $\fG_{\rm even}$ (resp.\ $\fG_{\rm odd}$) is the full subcategory of $\fG$ on sets of even (resp.\ odd) cardinality. We thus see that any $\fG$-module decomposes into a direct sum of a $\fG_{\rm even}$-module and a $\fG_{\rm odd}$-module, and that $\fG_{\rm even}$-modules and $\fG_{\rm odd}$-modules do not interact with each other.

\emph{For the remainder of this section, we take $\bk$ to be the field $\bC$ of complex numbers.}

\subsection{Triangular structure}

We say that a Brauer diagram from $S$ to $T$ is {\bf upwards} if it contains no horizontal edges in $S$, and {\bf downwards} if it contains no horizontal edges in $T$. We define $\fU$ to be the wide subcategory of $\fG$ where $\Hom_{\fU}(S,T)$ is the subspace of $\Hom_{\fG}(S,T)$ spanned by upwards diagrams. We similarly define $\fD$ using downwards morphisms. We note that the morphisms in $\fM$ are exactly the linear combinations of bijections.

\begin{proposition}
With the above definitions, $\fG$ is a triangular category.
\end{proposition}

\begin{proof}
We verify the conditions from the definition:
\begin{enumerate}
\item[(T0)] Every object of $\fG$ is isomorphic to $[n]$ for some $n \in \bN$, and so $\fG$ is essentially small. The Hom sets are all finite-dimensional as the set of Brauer diagrams between two fixed sets is finite.
\item[(T1)] The endomorphism ring $\End_\fU([n]) = \End_\fD([n])$ is the group algebra of the symmetric group $\fS_n$ over $\bC$, and hence is semi-simple.
\item[(T2)] We have a natural bijection $\vert \fG \vert = \bN$, and the standard order on this set is clearly admissible.
\item[(T3)] We use Proposition~\ref{prop:tri-crit}. The distinguished elements of $\Hom_\fG(x,y)$ are the elements corresponding to the Brauer diagrams. Conditions (a) and (b) clearly hold. If $\alpha$ is a downwards diagram and $\beta$ is an upwards diagram, then $\beta \bullet \alpha$ has no closed loops, and so $\beta \circ \alpha = \beta \bullet \alpha$ is distinguished (i.e., is a Brauer diagram on the nose, and not a scalar multiple). Thus (c) holds.

Finally, we verify condition (d). Let $\alpha \colon [n] \to [m]$ be a given Brauer diagram. Let $S \subseteq [n]$ and $T \subseteq [m]$ be the elements contained in a horizontal edge. Let $p = n - |S| = m - |T|$. Then we define $\beta \colon [n] \to [p]$ to be any downwards Brauer diagram that agrees with the original one on $S$; note that $\beta$ necessarily induces a bijection $[n] \setminus S \to [p]$. We then define $\gamma \colon [p] \to [m]$ to be an upwards Brauer diagram that agrees with the original on $T$ and such that the bijection $[p] \to [m] \setminus T$ is chosen so that the composition $[n] \setminus S \stackrel{\beta}{\to} [p] \stackrel{\gamma}{\to} [m] \setminus T$ is the bijection induced by $\alpha$. We then have $\alpha=\gamma \circ \beta$. Any other such factorization simply differs by a permutation of $[p]$, and so (d) holds. \qedhere
\end{enumerate}
\end{proof}

\begin{remark}
If we take the coefficient field $\bk$ to have positive characteristic, then $\End_{\fU}([n])$ is no longer semi-simple, and so $(\fU, \fD)$ does not define a triangular structure on $\fG$. The representation theory of $\fG$ in positive characteristic is still interesting, but much harder than the characteristic~0 case, and we do not attempt to say anything about it.
\end{remark}

\subsection{Lowest weight theory}

As usual, we let $\fM=\fU \cap \fD$. As we saw above, $\End_{\fM}([n])$ is the group algebra $\bC[\fS_n]$ of the symmetric group. Its simple modules are the Specht modules $\mathrm{Sp}_{\lambda}$ for partitions $\lambda$ of $n$. We thus see that the set $\Lambda$ of weights for $\fG$ is the set of all partitions. For $\lambda \in \Lambda$, the simple $\fM$-module $S_{\lambda}$ can be described as follows: if $\lambda$ is a partition of $n$ then $S_{\lambda}([n])=\mathrm{Sp}_{\lambda}$ is the Specht module, and $S_{\lambda}([m])=0$ for $m \ne n$. We note that for a $\fG$- (or $\fU$-, $\fD$-, or $\fM$-) module $M$, the multiplicity $m_{\lambda}(M)$ is simply the multiplicity of $\mathrm{Sp}_{\lambda}$ in the $\fS_n$-representation $M([n])$.

The general theory provides us with a number of important $\fG$-modules:
\begin{itemize}
\item The simple object $L_{\lambda}$.
\item The standard module $\Delta_{\lambda}$ and co-standard module $\nabla_{\lambda}$.
\item The principal projective $\bP_n$ and injective $\bI_n$ corresponding to the object $[n]$.
\item The projective $\wt{P}_{\lambda}=k_!(S_{\lambda})$ and the indecomposable projective $P_{\lambda}$.
\item The injective $\wt{I}_{\lambda}=k_*(S_{\lambda})$ and the indecomposable injective $I_{\lambda}$.
\end{itemize}
These objects will feature prominently throughout this series of papers. We now make a few simple observations about them.

The objects $\Delta_{\lambda}$, $\bP_n$, and $\wt{P}_{\lambda}$ are, in a sense, independent of $\delta$. (The same holds for the dual objects $\nabla_{\lambda}$, $\bI_n$, and $\wt{I}_{\lambda}$.) The following proposition shows one way in which this is true. We let $c^\lambda_{\mu,\nu}$ denote the Littlewood--Richardson coefficients \cite[(2.13)]{expos}.

\begin{proposition}
We have the following: 
\begin{enumerate}
\item $m_{\mu}(\Delta_{\lambda})= \sum_\nu c^\mu_{\lambda,2\nu}$
\item $[\wt{P}_{\lambda}:\Delta_{\mu}]= \sum_\nu c^\lambda_{\mu,2\nu}$
\item $\bP_n \cong \bigoplus_{\vert \lambda \vert=n} \wt{P}_{\lambda}^{\oplus \dim \mathrm{Sp}_{\lambda}}$.
\end{enumerate}
Note that everything here is independent of $\delta$.
\end{proposition}

\begin{proof}
  (a) Let $|\lambda|=n$. We have
  \[
    i^*\Delta_\lambda = i^* j_! S_\lambda = \Hom_{\fS_n} ({\rm Sp}_\lambda, \bP_{\fU,n})
  \]
  Let $|\mu|=m$. For a fixed inclusion $[n] \to [m]$, $\fS_{m-n}$ acts transitively on the set of perfect matchings on $[m] \setminus [n]$, let $H_{m-n}$ be the stabilizer of a particular matching. Then we get
  \[
    i^*\Delta_\lambda([m]) = \Ind_{\fS_n \times H_{m-n}}^{\fS_m} {\rm Sp}_\lambda \cong \bigoplus_{\substack{|\mu|=m,\\
        |\nu|=(m-n)/2}} {\rm Sp}_\mu^{\oplus c^\mu_{\lambda,2\nu}}
  \]
    \cite[(2.13), Example 6.3.4]{expos}, so the formula follows.

    (b) By Proposition~\ref{prop:Delta-mult}, Proposition~\ref{prop:Ptilde}(b), and duality,
    \[
      [\wt{P}_\lambda: \Delta_\mu] = \dim \Hom_\fG(\wt{P}_\lambda, \nabla_\mu) = m_{\lambda}(\nabla_\mu) = m_{\lambda}(\Delta_\mu),
    \]
    and so the result follows from (a).

    (c) We have $\bC[\fS_n] = \bigoplus_{|\lambda|=n} {\rm Sp}_\lambda^{\oplus \dim {\rm Sp}_\lambda}$ as $\fS_n$-representations. It thus follows that $\bP_{\fM,n} = \bigoplus_{|\lambda|=n} S_{\lambda}^{\oplus \dim {\rm Sp}_\lambda}$ as $\fM$-modules. Apply $k_!$ to get the desired identity.
\end{proof}

The objects $P_{\lambda}$ and $L_{\lambda}$, by contrast, depend in subtle ways on $\delta$. The following example illustrates this for $P_{\lambda}$.

\begin{example} 
The principal projective $\bP_0$ is indecomposable, and coincides with both $\wt{P}_{\emptyset}$ and $P_{\emptyset}$. Let $\lambda=(2)$. The decomposition of $\wt{P}_{\lambda}$ can only involve $P_{\lambda}$ and $P_{\emptyset}$ for parity reasons. There is a unique map $f \colon \bP_2 \to \bP_0$ (up to scaling), and it factors through $\wt{P}_{\lambda}$. We thus see that $P_{\emptyset}$ is a summand of $\wt{P}_{\lambda}$ if and only if $f$ is surjective.

Let us examine $f$ more closely. By definition $f$ takes the identity element of $\bP_2([2])$ to the unique Brauer diagram $\alpha$ in $\bP_0([2])$. Since every morphism in $\fG$ can be factored as a downwards morphism followed by an upwards morphism, we see that $(\im{f})(0)$ is the space spanned by applying all downwards maps $[2] \to [0]$ to $\alpha$. There is a unique diagram $\beta \colon [2] \to [0]$, and $\beta \circ \alpha = \delta \cdot \id_{[0]}$. Thus if $\delta \ne 0$ then $\id_{[0]} \in \im(f)$, and so $f$ is surjective since $\id_{[0]}$ generates $\bP_0$; if $\delta=0$ then $\id_{[0]} \not\in \im(f)$, and so $f$ is not surjective.

We thus find that $\wt{P}_{\lambda}=P_{\lambda} \oplus P_{\emptyset}$ if $\delta \ne 0$, while $\wt{P}_{\lambda}=P_{\lambda}$ if $\delta=0$. The standard pieces of $\wt{P}_{\lambda}$ are $\Delta_{\lambda}$ and $\Delta_{\emptyset}=P_{\emptyset}$, each with multiplicity one. Thus $P_{\lambda}=\Delta_{\lambda}$ if $\delta \ne 0$, while $P_{\lambda}$ is a non-trivial extension of $\Delta_{\lambda}$ by $\Delta_{\emptyset}$ if $\delta=0$.
\end{example}

\subsection{The noetherian property}

The following is a fundamental result about the Brauer category and will be used constantly in what follows.

\begin{theorem}
The category $\Mod_{\fG}$ is locally noetherian, that is, any submodule of a finitely generated $\fG$-module is again finitely generated.
\end{theorem}

\begin{proof}
We have previously shown that $\Mod_{\fU}$ is locally noetherian (see \cite[Theorem 1.1, Remark 1.3]{sym2noeth}, and note that what is called $\mathbf{FIM}$ there is the same as our $\fU$ here), and so the result follows from Proposition~\ref{prop:Unoeth}.
\end{proof}

\begin{remark}
Some points related to the theorem:
\begin{enumerate}
\item The theorem implies that the category of finitely generated $\fG$-modules is abelian.
\item Due to (a), it makes sense to consider the Grothendieck group of finitely generated $\fG$-modules. Determining this group is a fundamental problem, which we solve in \cite{brauercat5}.
\item By \cite[Remark~1.3]{sym2noeth}, the category of $\fU$-modules is equivalent to the category of $\GL_{\infty}$-equivariant modules over the infinite variable polynomial ring $\Sym(\Sym^2(\bC^{\infty}))$ (with a polynomiality condition on the $\GL_{\infty}$-action). This point of view allows one to apply tools from algebraic geometry and commutative algebra to study $\fU$-modules, which is crucial to the proof of \cite[Theorem~1.1]{sym2noeth}. This perspective will be important in this series of papers too.
\item If $M$ is a $\fG$-module, then the lattice of submodules of $i^*M$ is generally much larger than that of $M$, so local noetherianity of $\Mod_\fG$ should be an easier property to prove than that of $\Mod_\fU$. However, the above proof is the only one we know. \qedhere
\end{enumerate}
\end{remark}

\subsection{Tautological modules} \label{ss:taut}

Suppose that $\delta=p$ is a non-negative integer. Equip $V=\bC^p$ with a non-degenerate symmetric bilinear form. Given a Brauer diagram $\alpha$ from $S$ to $T$, we defined a linear map $\alpha_* \colon V^{\otimes S} \to V^{\otimes T}$ in \S \ref{ss:brauer-motiv}. Moreover, from \eqref{eq:brauer-comp} and the definition of the Brauer category, it follows that formation of $\alpha_*$ is compatible with composition in $\fG$, that is, we have $(\beta \circ \alpha)_*=\beta_* \circ \alpha_*$. In other words, the rule $S \mapsto V^{\otimes S}$ defines a $\fG$-module, which we denote by $T_{p|0}$. The transition maps in $T_{p|0}$ are compatible with the action of $\bO_p$. We can thus regard $\bO_p$ as acting on $T_{p|0}$; alternatively, we can regard $T_{p|0}$ as a functor to $\Rep(\bO_p)$.

More generally, suppose that $\delta$ is an integer, and we have $\delta=p-q$ for non-negative integers $p$ and $q$ with $q$ even. Equip the super vector space $V=\bC^{p|q}$ with a non-degenerate symmetric bilinear form. Given a Brauer diagram $\alpha$ from $S$ to $T$, the same construction yields a map $\alpha_* \colon V^{\otimes S} \to V^{\otimes T}$ (keep in mind the sign conventions when dealing with super vector spaces), which is also functorial. We denote the resulting $\fG$-module by $T_{p|q}$. As above, the super group $\OSp_{p|q}$ acts on $T_{p|q}$, and we can view $T_{p|q}$ as a functor to $\Rep(\OSp_{p|q})$.

In essence, the Brauer category was defined so that the $T_{p|q}$ would be modules. For that reason, we refer to them as {\bf tautological modules}. These modules will play an important role in our study of the Brauer category, especially in \cite{brauercat4}. We emphasize that tautological modules exist only when the parameter $\delta$ is an integer. This is one manifestation of the fact that the Brauer category behaves very differently depending on whether $\delta$ is an integer or not.

\subsection{Duality} \label{ss:brauer-duals}

Mathematically, a Brauer diagram from $S$ to $T$ is exactly the same as a Brauer diagram from $T$ to $S$; the only difference between the two notions is which vertices we put on the bottom when we draw them. Moreover, this symmetry is compatible with composition. In other words, we have an equivalence of categories
\begin{displaymath}
\tau \colon \fG \to \widehat{\fG}.
\end{displaymath}
Explicitly, $\tau$ is the identity on objects, and the bijection
\begin{displaymath}
\tau \colon \Hom_{\fG}(S, T) \to \Hom_{\widehat{\fG}}(S, T)=\Hom_{\fG}(T, S)
\end{displaymath}
takes a Brauer diagram from $S$ to $T$ to the same graph, but regarded as a Brauer diagram from $T$ to $S$. It is clear that $\tau$ is a triangular equivalence and squares to the identity. Moreover, $\tau^*(S^{\vee}_{\lambda})$ is isomorphic to $S_{\lambda}$, since Specht modules are self-dual. Thus $\tau$ is a transpose on $\fG$ (in the sense of Definition~\ref{defn:transpose}).

Suppose that $M$ is a $\fG$-module. We can then regard the $\widehat{\fG}$-module $M^{\vee}$ as a $\fG$-module via $\tau$. In what follows we always do this. We thus have a duality functor
\begin{displaymath}
(-)^{\vee} \colon \Mod_{\fG}^{\rm op} \to \Mod_{\fG}.
\end{displaymath}
It induces an equivalence on the subcategories of pointwise finite modules. We also have duality functors between $\fU$ and $\fD$ modules, and a duality functor on $\fM$-modules.

On the main objects of interest, duality is given as follows:
\begin{align*}
S_{\lambda}^{\vee} &= S_{\lambda} &
L_{\lambda}^{\vee} &= L_{\lambda} &
\Delta_{\lambda}^{\vee} &= \nabla_{\lambda} \\
\bP_n^{\vee} &= \bI_n &
P_{\lambda}^{\vee} &= I_{\lambda} &
T_{p|q}^{\vee} &= T_{p|q}
\end{align*}
We have already explained that $S_{\lambda}$ is self-dual. The formulas for $L_{\lambda}^{\vee}$ and $\Delta_{\lambda}^{\vee}$ are discussed in Remark~\ref{rmk:transp-duals}. The formula for $\bP_n^{\vee}$ is discussed in \S \ref{ss:principal}, while that for $P_{\lambda}^{\vee}$ is discussed in \S \ref{ss:inj}. Finally, let $V=\bC^{p \vert q}$ equipped with a non-degenerate symmetric form. The form gives an isomorphism $V^{\otimes S} \cong (V^{\otimes S})^*$ for all finite sets $S$, and one verifies that this is compatible with maps in the Brauer category, after identifying it with its opposite; this shows that $T_{p \vert q}$ is self-dual.

\subsection{The tensor product}

Given Brauer diagrams $\alpha \colon S \to T$ and $\alpha' \colon S' \to T'$, the disjoint union of $\alpha$ and $\alpha'$ is a Brauer diagram from $S \amalg S'$ to $T \amalg T'$. It is clear that a disjoint union of upwards (resp.\ downwards) morphisms is upwards (resp.\ downwards). We thus see that $\fG$ has the structure of a symmetric monoidal triangular category. We thus get tensor products on the various modules categories as in \S \ref{ss:trimonoid}.

\begin{proposition} \label{prop:brauer-tensor-exact}
The tensor product $\otimes_{\fD}$ on $\Mod_{\fD}$ is exact.
\end{proposition}

\begin{proof}
Let $\cU$ be the wide subcategory of $\fG$ where the morphisms are upwards Brauer diagrams (not linear combinations of such diagrams). Then $\fU=\bk[\cU]$. We prove that property $(S_x)$ holds for all objects $x$ of $\fU$ and then apply Proposition~\ref{prop:tensor-exact} (since $\fD = \fU^\op$). Fix $x = [n]$. Consider the set of upwards diagrams $\phi \colon [n] \to [a] \amalg [b]$ where every horizontal edge has one endpoint in $[a]$ and one in $[b]$. We consider two of them to be equivalent if they differ by permutations of either $[a]$ or $[b]$ and let $I$ be a set of representatives for the equivalence classes. For $i \in I$, we denote the corresponding representative $\phi_i \colon [n] \to [a_i] \amalg [b_i]$.

  Let $\psi \colon [n] \to y \amalg z$ be any upwards Brauer diagram. Let $S \subseteq y$ be the subset obtained by removing all horizontal edges (and vertices) such that both of its vertices are in $y$ and similarly define $T \subseteq z$. Let $\phi \colon [n] \to S \amalg T$ be the restriction of $\psi$ to these subsets. There is a unique $i \in I$ such that $\phi = \phi_i$ for some choice of bijections $[a_i] \cong S$ and $[b_i] \cong T$. Furthermore, we then have a pair of upwards Brauer diagrams $\alpha \colon [a_i] \to y$ and $\beta \colon [b_i] \to z$ such that $\psi =(\alpha \amalg \beta) \circ \phi_i$ and this choice of $(\alpha,\beta)$ is unique up to the stabilizer of $\phi_i$ in $\Aut_{\cU}([a_i]) \times \Aut_{\cU}([b_i])$. This establishes $(S_x)$ and we finishes the proof.
\end{proof}

\begin{corollary}
If $M$ and $N$ are semistandard $\fG$-modules then $M \otimes_{\fG} N$ is semistandard and $\Tor_p^{\fG}(M,N)=0$ for $p>0$.
\end{corollary}

\begin{proof}
This follows from Proposition~\ref{prop:ss-tensor}.
\end{proof}

\begin{remark} \label{rmk:brauer-tensor}
A number of remarks related to tensor products:
\begin{enumerate}
\item The tensor product $\otimes_{\fG}$ is not exact. This will become clear in \cite{brauercat4}.
\item The tensor product $\otimes_{\fU}$ is also not exact. One can see this as follows. Let $A$ be the tca $\Sym(\Sym^2(\bC^{\infty}))$. Then $\Mod_{\fU}$ is equivalent to $\Mod_A$ (see \cite[(4.3.1)]{infrank}), and under this equivalence $\otimes_{\fU}$ corresponds to $\otimes_A$ (as one can see by considering projectives, for instance), which is not exact.
\item The exactness of $\otimes_{\fD}$ also follows (though more indirectly) from the results of \cite[(4.3.1)]{infrank}: there we show that $\Mod_{\fD}$ is equivalent to a certain category of representations of $\bO_{\infty}$ and that, under this equivalence, $\otimes_{\fD}$ corresponds to the usual tensor product.
\item An important problem is to classify ideals of $\fG$ (as defined in \S \ref{ss:tensor}). We solve this problem in \cite{brauercat3}. \qedhere
\end{enumerate}
\end{remark}

\section{The partition category} \label{sec:partition}

\subsection{Motivation} \label{ss:part-motiv}

Let $V=\bC^n$ be the permutation representation of the symmetric group $\fS_n$, with basis $e_1, \ldots, e_n$. The idea of the partition category is to record all the obvious $\fS_n$-equivariant maps between tensor powers of $V$.

Consider the map
\begin{displaymath}
a_{s,t} \colon V^{\otimes s} \to V^{\otimes t}, \qquad
e_{i_1,\ldots,i_s} \mapsto \sum_{j_1,\ldots,j_t} \delta_{i_1,\ldots,i_s,j_1,\ldots,j_t} e_{j_1,\ldots,j_t},
\end{displaymath}
where $e_{i_1,\ldots,i_s}=e_{i_1} \otimes \cdots \otimes e_{i_s}$, and $\delta$ is~1 if all indices are the same and~0 otherwise. For example:
\begin{itemize}
\item $a_{0,1} \colon \bC \to V$ maps 1 to the invariant vector $\sum_{i=1}^n e_i$.
\item $a_{1,0} \colon V \to \bC$ is the augmentation map, defined by $e_i \mapsto 1$ for all $i$.
\item $a_{1,1} \colon V \to V$ is the identity map.
\item $a_{2,1} \colon V^{\otimes 2} \to V$ is given by $e_{i,j} \mapsto \delta_{i,j} e_i$.
\item $a_{1,2} \colon V \to V^{\otimes 2}$ is given by $e_i \mapsto e_i \otimes e_i$.
\end{itemize}
We have an analogous map $a_{S,T} \colon V^{\otimes S} \to V^{\otimes T}$ for finite sets $S$ and $T$ (e.g., pick bijections $S \cong [s]$ and $T \cong [t]$ and transport $a_{s,t}$). One easily sees that these maps are $\fS_n$-equivariant.

We can now apply the same three constructions as in \S \ref{ss:brauer-motiv} (composition, tensor product, linear combination) to create more maps from these. The resulting maps are once again conveniently represented by certain diagrams, as follows. Let $S$ and $T$ be finite sets. A {\bf partition diagram} from $S$ to $T$ is a partition of the set $S \amalg T$ into non-empty disjoint subsets. Given a partition diagram $\alpha$ from $S$ to $T$, we let $\alpha_* \colon V^{\otimes S} \to V^{\otimes T}$ be the tensor product of the maps $a_{S \cap U, T \cap U}$ taken over the parts $U$ of $\alpha$. This is an $\fS_n$-equivariant map.

\begin{example}
Let $\alpha$ be the partition diagram from $[4]$ to $[5]$ given by
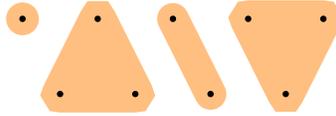
\begin{center}
\begin{tikzpicture}
\draw[orange!50!white,fill=orange!50!white] (1,1) circle (6pt);
\draw[line width=12pt,line cap=round,orange!50!white] (3.5,0) -- (3,1);
\draw[orange!50!white,fill=orange!50!white,line width=14pt,rounded corners=.5pt] (1.5,0) -- (2.5,0) -- (2,1) -- cycle;
\draw[orange!50!white,fill=orange!50!white,line width=14pt,rounded corners=.5pt] (4.5,0) -- (5,1) -- (4,1) -- cycle;
\node [node] at (1,1) (b2) {};
\node [node] at (2,1) (b3) {};
\node [node] at (3,1) (b4) {};
\node [node] at (4,1) (b5) {};
\node [node] at (5,1) (b6) {};
\node [node] at (1.5,0) (c1) {};
\node [node] at (2.5,0) (c2) {};
\node [node] at (3.5,0) (c3) {};
\node [node] at (4.5,0) (c4) {};
\end{tikzpicture}
\end{center}
Then $\alpha_* \colon V^{\otimes 4} \to V^{\otimes 5}$ is given by
\begin{displaymath}
\alpha_*(e_{i_1,i_2,i_3,i_4}) = \delta_{i_1,i_2} \sum_{j=1}^n e_{j,i_1,i_3,i_4,i_4}. \qedhere
\end{displaymath}
\end{example}

As in the previous case, the composition of two maps corresponding to diagrams is a scalar multiple of another such map. We now explain exactly how this works. Suppose that $\alpha$ is a partition diagram from $S$ to $T$ and $\beta$ is a partition diagram from $T$ to $U$. Let $\beta \cup \alpha$ denote the set $S \amalg T \amalg U$ equipped with all the parts from $\alpha$ and $\beta$ (possibly with multiplicities); this is not a partition since each vertex of $T$ appears in two parts. Let $\beta \Box \alpha$ denote the result of merging all parts in $\beta \cup \alpha$ that meet; this is a partition of $S \amalg T \amalg U$. We define $\beta \bullet \alpha$ to be the induced partition of $S \amalg U$ (i.e., intersect each part of $\beta \Box \alpha$ with $S \amalg U$, discarding empty sets), and we let $c(\beta, \alpha)$ denote the number of parts of $\beta \Box \alpha$ that contain only vertices of $T$. We then have the fundamental formula:
\begin{displaymath}
\beta_* \circ \alpha_* = (\dim{V})^{c(\beta,\alpha)} (\beta \bullet \alpha)_*.
\end{displaymath}
Again, this can be proved by a straightforward computation.

\begin{example}
We give an example illustrating composition in the partition category. Let $\beta \colon [7] \to [5]$ be given by the following diagram
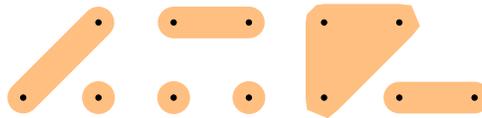
\begin{center}
\begin{tikzpicture}
\draw[line width=12pt,line cap=round,orange!50!white] (0,0) -- (1,1);
\draw[line width=12pt,line cap=round,orange!50!white] (2,1) -- (3,1);
\draw[orange!50!white,fill=orange!50!white] (1,0) circle (6pt);
\draw[orange!50!white,fill=orange!50!white] (3,0) circle (6pt);
\draw[orange!50!white,fill=orange!50!white] (2,0) circle (6pt);
\draw[orange!50!white,fill=orange!50!white,line width=14pt,rounded corners=.5pt] (4,0) -- (4,1) -- (5,1) -- cycle;
\draw[line width=12pt,line cap=round,orange!50!white] (5,0)--(6,0);
\node [node] at (0,0) (a1) {};
\node [node] at (1,0) (a2) {};
\node [node] at (2,0) (a3) {};
\node [node] at (3,0) (a4) {};
\node [node] at (4,0) (a5) {};
\node [node] at (5,0) (a6) {};
\node [node] at (6,0) (a7) {};
\node [node] at (1,1) (b2) {};
\node [node] at (2,1) (b3) {};
\node [node] at (3,1) (b4) {};
\node [node] at (4,1) (b5) {};
\node [node] at (5,1) (b6) {};
\end{tikzpicture}
\end{center}
and let $\alpha \colon [4] \to [7]$ be given by
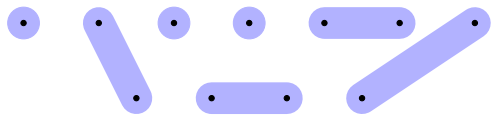
\begin{center}
\begin{tikzpicture}
\draw[blue!30!white,fill=blue!30!white] (0,1) circle (6pt);
\draw[blue!30!white,fill=blue!30!white] (2,1) circle (6pt);
\draw[blue!30!white,fill=blue!30!white] (3,1) circle (6pt);
\draw[line width=12pt,line cap=round,blue!30!white] (1.5,0) -- (1,1);
\draw[line width=12pt,line cap=round,blue!30!white] (2.5,0) -- (3.5,0);
\draw[line width=12pt,line cap=round,blue!30!white] (4,1) -- (5,1);
\draw[line width=12pt,line cap=round,blue!30!white] (4.5,0) -- (6,1);
\node [node] at (0,1) (a1) {};
\node [node] at (1,1) (a2) {};
\node [node] at (2,1) (a3) {};
\node [node] at (3,1) (a4) {};
\node [node] at (4,1) (a5) {};
\node [node] at (5,1) (a6) {};
\node [node] at (6,1) (a7) {};
\node [node] at (1.5,0) (c1) {};
\node [node] at (2.5,0) (c2) {};
\node [node] at (3.5,0) (c3) {};
\node [node] at (4.5,0) (c4) {};
\end{tikzpicture}
\end{center}
The partition $\beta \Box \alpha$ is then given by
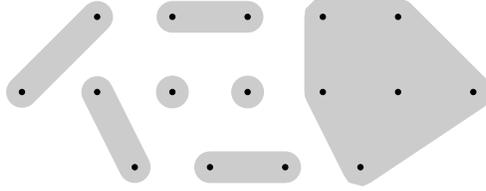
\begin{center}
\begin{tikzpicture}
\draw[line width=12pt,line cap=round,gray!40!white] (0,1) -- (1,2);
\draw[line width=12pt,line cap=round,gray!40!white] (1.5,0) -- (1,1);
\draw[line width=12pt,line cap=round,gray!40!white] (2,2) -- (3,2);
\draw[gray!40!white,fill=gray!40!white] (2,1) circle (6pt);
\draw[gray!40!white,fill=gray!40!white] (3,1) circle (6pt);
\draw[line width=12pt,line cap=round,gray!40!white] (2.5,0) -- (3.5,0);
\draw[gray!40!white,fill=gray!40!white,line width=14pt,rounded corners=.5pt] (4.5,0) -- (6,1) -- (5,2) -- (4,2) -- (4,1) -- cycle;
\node [node] at (0,1) (a1) {};
\node [node] at (1,1) (a2) {};
\node [node] at (2,1) (a3) {};
\node [node] at (3,1) (a4) {};
\node [node] at (4,1) (a5) {};
\node [node] at (5,1) (a6) {};
\node [node] at (6,1) (a7) {};
\node [node] at (1,2) (b2) {};
\node [node] at (2,2) (b3) {};
\node [node] at (3,2) (b4) {};
\node [node] at (4,2) (b5) {};
\node [node] at (5,2) (b6) {};
\node [node] at (1.5,0) (c1) {};
\node [node] at (2.5,0) (c2) {};
\node [node] at (3.5,0) (c3) {};
\node [node] at (4.5,0) (c4) {};
\end{tikzpicture}
\end{center}
There are two parts concentrated in the middle, and so $c(\beta,\alpha)=2$. The partition $\gamma=\beta \bullet \alpha$ is given by
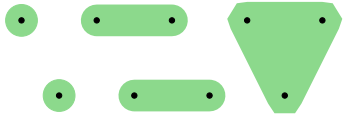
\begin{center}
\begin{tikzpicture}
\draw[green!50!gray!60!white,fill=green!50!gray!60!white] (1.5,0) circle (6pt);
\draw[green!50!gray!60!white,fill=green!50!gray!60!white] (1,1) circle (6pt);
\draw[line width=12pt,line cap=round,green!50!gray!60!white] (2,1) -- (3,1);
\draw[line width=12pt,line cap=round,green!50!gray!60!white] (2.5,0) -- (3.5,0);
\draw[green!50!gray!60!white,fill=green!50!gray!60!white,line width=14pt,rounded corners=.5pt] (4.5,0) -- (5,1) -- (4,1) -- cycle;
\node [node] at (1,1) (b2) {};
\node [node] at (2,1) (b3) {};
\node [node] at (3,1) (b4) {};
\node [node] at (4,1) (b5) {};
\node [node] at (5,1) (b6) {};
\node [node] at (1.5,0) (c1) {};
\node [node] at (2.5,0) (c2) {};
\node [node] at (3.5,0) (c3) {};
\node [node] at (4.5,0) (c4) {};
\end{tikzpicture}
\end{center}
We thus have $\beta_* \circ \alpha_* = (\dim V)^2 \cdot \gamma_*$.
\end{example}

\subsection{Definition}

Fix a field (or even a commutative ring) $\bk$ and let $\delta \in \bk$. We define the {\bf partition category} over $\bk$ with parameter $\delta$, denoted $\fG$, as follows. The objects of $\fG$ are finite sets. The set $\Hom_{\fG}(S,T)$ is the free $\bk$-module on the partition diagrams from $S$ to $T$. For partition diagrams $\alpha \in \Hom_{\fG}(S,T)$ and $\beta \in \Hom_{\fG}(T,U)$, the composite morphism $\beta \circ \alpha$ is defined to be $\delta^{c(\beta,\alpha)} \cdot (\beta \bullet \alpha)$. Composition for general morphisms is defined by linearity. As with the Brauer category, we tacitly regard bijections as morphisms in the partition category.

The endomorphism rings in the partition category are the classical partition algebras introduced in \cite{martin} and \cite{jones}, and studied in \cite{halverson-ram}. The partition category itself appears in \cite[\S 5]{martin2} and \cite[\S 2.2]{jellyfish}.

\emph{For the remainder of this section, we take $\bk$ to be the field $\bC$ of complex numbers.}

\subsection{Triangular structure}

We say that a partition diagram from $S$ to $T$ is {\bf upwards} if each part meets $T$, and meets $S$ at~0 or~1 elements; we define {\bf downwards} analogously. We define $\fU$ (resp.\ $\fD$) to be the wide subcategory whose $\Hom$ spaces are spanned by upwards (resp.\ downwards) morphisms.

\begin{proposition}
Assume $\bk=\bC$. Then $\fG$, with the above $(\fU,\fD)$, is a triangular category.
\end{proposition}

\begin{proof}
We verify the conditions from the definition:
\begin{enumerate}
\item[(T0)] Every object of $\fG$ is isomorphic to $[n]$ for some $n \in \bN$, and so $\fG$ is essentially small. The Hom sets are all finite-dimensional as the set of partition diagrams between two fixed sets is finite.
\item[(T1)] The endomorphism ring $\End_\fU([n]) = \End_\fD([n])$ is the group algebra of the symmetric group $\fS_n$ over $\bC$, and hence is semi-simple.
\item[(T2)] We have a natural bijection $\vert \fG \vert = \bN$, and the standard order on this set is admissible.
\item[(T3)] We use Proposition~\ref{prop:tri-crit}. The distinguished elements of $\Hom_\fG(x,y)$ are the elements corresponding to the partition diagrams. Conditions (a) and (b) clearly hold. If $\alpha$ is a downwards diagram and $\beta$ is an upwards diagram, then $\beta \Box \alpha$ has no parts concentrated in the middle, and so $\beta \circ \alpha = \beta \bullet \alpha$ is distinguished (i.e., is a partition diagram on the nose, and not a scalar multiple). Thus (c) holds.

  Finally, we verify condition (d). Let $\phi \colon S \to T$ be a partition diagram. Let $U$ be the set of parts of $\phi$ that have nonempty intersection with both $S$ and $T$. We define a partition diagram $\alpha \colon S \to U$ as follows: each $u \in U$ is joined with the elements of $S$ that it contains, and all other elements of $S$ are joined as they were in $\phi$. We similarly define a partition diagram $\beta\colon U \to T$. Then $\alpha$ is downwards, $\beta$ is upwards, and $\phi = \beta \circ \alpha$. Any other such factorization has to go through a set of the same size as $U$, so we see that the factorization is unique up to a permutation of $U$.
  \qedhere
\end{enumerate}
\end{proof}

\begin{remark}
We do not know if the upwards category $\fU$ is noetherian; in fact, we suspect it is not. However, we will show in \cite{brauercat3} that all finitely generated $\fG$-modules have finite length, and are therefore noetherian.
\end{remark}

\subsection{Tautological modules}

Suppose that $\delta=n$ is a non-negative integer. Let $V=\bC^n$ be the permutation representation of $\fS_n$. Given a partition diagram $\alpha$ from $S$ to $T$, we defined a map $\alpha_* \colon V^{\otimes S} \to V^{\otimes T}$ in \S \ref{ss:part-motiv}. The definition of the partition category ensures that formation of $\alpha_*$ is compatible with composition in $\fG$. We thus have a $\fG$-module given by $T(S)=V^{\otimes S}$ and $T(\alpha)=\alpha_*$, which we call the {\bf tautological module}. Since the maps $\alpha_*$ are $\fS_n$-equivariant, we can regard $\fS_n$ as acting on $T$. Alternatively, $T$ is a functor from $\fG$ to $\Rep(\fS_n)$.

\subsection{Duality}

Just like for Brauer diagrams, a partition diagram from $S$ to $T$ is exactly the same as a partition diagram from $T$ to $S$. This gives rise to a transpose functor $\tau \colon \fG \to \widehat{\fG}$ (in the sense of Definition~\ref{defn:transpose}). For the primary $\fG$-modules of interest, duals can be computed as in \S \ref{ss:brauer-duals}.

\subsection{The tensor product}

Just like for Brauer diagrams, given partition diagrams $\alpha \colon S \to T$ and $\alpha' \colon S' \to T'$, the disjoint union of $\alpha$ and $\alpha'$ is a partition diagram from $S \amalg S'$ to $T \amalg T'$. It is clear that a disjoint union of upwards (resp.\ downwards) morphisms is upwards (resp.\ downwards). We thus see that $\fG$ has the structure of a symmetric monoidal triangular category. We thus get tensor products on the various modules categories as in \S \ref{ss:trimonoid}.

\begin{proposition} \label{prop:partition-tensor-exact}
The tensor product $\otimes_{\fD}$ on $\Mod_{\fD}$ is exact.
\end{proposition}

\begin{proof}
We prove that property $(S_x)$ holds for all objects $x$ of $\fU$ and then apply Proposition~\ref{prop:tensor-exact} (since $\fD = \fU^\op$). Fix $x = [n]$. Consider the set of upwards diagrams $\phi \colon [n] \to [a] \amalg [b]$ where no part of $\phi$ is a subset of $[a]$ nor is a subset of $[b]$. We consider two of them to be equivalent if they differ by permutations of either $[a]$ or $[b]$ and let $I$ be a set of representatives for the equivalence classes. The rest of the proof is the same as the proof of Proposition~\ref{prop:brauer-tensor-exact}, so we omit the details.
\end{proof}

\begin{corollary}
If $M$ and $N$ are semistandard $\fG$-modules then $M \otimes_{\fG} N$ is semistandard and $\Tor_p^{\fG}(M,N)=0$ for $p>0$.
\end{corollary}

\begin{proof}
This follows from Proposition~\ref{prop:ss-tensor}.
\end{proof}

\begin{remark}
The exactness of $\otimes_{\fD}$ also follows (though more indirectly) from the results of \cite[(6.3.32)]{infrank}: there we show that $\Mod_{\fD}$ is equivalent to a certain category of representations of $\fS_{\infty}$ and that, under this equivalence, $\otimes_{\fD}$ corresponds to the usual tensor product.
\end{remark}

\section{Other combinatorial triangular categories} \label{sec:other}

\subsection{Brauer-like categories}

The Brauer category records the obvious maps between tensor representations of the orthogonal group. By considering other Lie (super)groups (or other representations), one obtains similar categories. We now explain a number of examples.

\subsubsection{The signed Brauer category}

This category records the obvious maps between tensor powers of the standard representation of the symplectic group.
\begin{itemize}
\item We define the signed Brauer category $\fG$ as follows. The objects are finite sets. A signed Brauer diagram from $S$ to $T$ is a Brauer diagram in which the horizontal edges are oriented. The space $\Hom_{\fG}(S,T)$ is spanned by signed Brauer diagrams from $S$ to $T$ modulo the following relation: if $\beta$ is obtained from $\alpha$ by inverting the orientation of a single edge then $\beta=-\alpha$. To compose two diagrams, proceed as in the Brauer category, but first adjust orientations so that all paths and cycles are oriented coherently.
\item The usual construction endows $\fG$ with a triangular structure. We have $\End_{\fM}([n])=\bC[\fS_n]$, and so the set of weights is again identified with the set of partitions.
\item Suppose $\delta=p$ is a non-negative even integer. Let $V$ be a symplectic space of dimension $p$. We then have a tautological module $T_{p|0}$ that takes $S$ to $V^{\otimes S}$. The action of $T_{p|0}$ on morphisms is defined similarly to the Brauer category case; the one modification is that the orientation is used to determine which tensor factor is placed first when applying $a$ or $b$. More generally, we have tautological modules $T_{p|q}$ whenever $\delta=p-q$ with $p$ even by considering super symplectic spaces.
\item The upwards category is noetherian by \cite[Theorem~1.1]{sym2noeth}, and so $\fG$ is noetherian.
\item The category $\fG$ has a transpose functor and monoidal structure, similar to the Brauer category. The functor $\otimes_{\fD}$ is exact.
\end{itemize}
In fact, the signed Brauer category is, in a sense, nothing new:

\begin{proposition} \label{prop:signed-brauer}
Let $\fG'$ be the Brauer category with parameter $-\delta$. Then we have an equivalence of triangular categories $\fG \cong \fG'$.
\end{proposition}

\begin{proof}
  We are free to replace $\fG$ with a skeletal subcategory, so we assume all objects have the form $[n]$ for some $n \in \bN$. Let $\alpha$ be a signed Brauer diagram from $[n]$ to $[m]$. Let $i \colon [n] \amalg [m] \to [n+m]$ be the bijection that is the identity on $[n]$ and takes $j \in [m]$ to $n+m+1-j$. Let $\alpha'$ be the oriented matching on $[n+m]$ obtained by transferring $\alpha$ via $i$, and orienting all former vertical edges from smaller values to larger values. Define $\epsilon(\alpha)$ to be the sign of any permutation that transforms $\alpha'$ to the standard oriented matching $\{(1,2),(3,4),\ldots\}$ with $2i-1$ oriented towards $2i$ for all $i$. Also, let $\ol{\alpha}$ be the Brauer diagram from $[n]$ to $[m]$ with the orientations forgotten. Then $\alpha \mapsto \epsilon(\alpha) \ol{\alpha}$ is a well-defined linear map $\Hom_\fG([n],[m]) \to \Hom_{\fG'}([n],[m])$.

Define $\Phi \colon \fG \to \fG'$ by $\Phi([n])=[n]$ and $\Phi(\alpha) = \epsilon(\alpha) \ol{\alpha}$. We claim this is a functor, i.e., for any signed Brauer diagram $\beta$ from $[m]$ to $[p$], we have $\Phi(\alpha \beta) = \Phi(\alpha) \Phi(\beta)$. Without loss of generality, we may assume that the horizontal edges of $\alpha$ and $\beta$ are oriented so that any loops in  $\alpha \cup \beta$ are coherently oriented. We thus have $\alpha \circ \beta = \delta^{c(\alpha,\beta)} \alpha \bullet \beta$ and $\ol{\alpha} \circ \ol{\beta} = \delta^{c(\ol{\alpha}, \ol{\beta})} \ol{\alpha} \bullet \ol{\beta}$. We also claim that $\epsilon(\alpha \bullet \beta) = \epsilon(\alpha) \epsilon(\beta)$, which will prove that $\Phi$ is compatible with composition. To see this, we will reduce to the case that $\alpha$ and $\beta$ have no horizontal edges, in which case this becomes multiplicativity of the sign of permutations.

First, suppose $\alpha \cup \beta$ has a loop consisting of $2r$ edges which looks like the following up to mirroring:
  
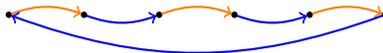
\begin{center}
\begin{tikzpicture}
\node [node] at (0,1) (a1) {};
\node [node] at (1,1) (a2) {};
\node [node] at (2,1) (a3) {};
\node [node] at (3,1) (a4) {};
\node [node] at (4,1) (a5) {};
\node [node] at (5,1) (a6) {};
\draw[thick, orange,->] (a5) to[out=20,in=160] (a6);
\draw[thick, orange,->] (a3) to[out=20,in=160] (a4);
\draw[thick, orange,->] (a1) to[out=20,in=160] (a2);
\draw[thick, blue,->] (a4) to[out=-20,in=-160] (a5);
\draw[thick, blue,->] (a2) to[out=-20,in=-160] (a3);
\draw[thick, blue,<-] (a1) to[out=-20,in=-160] (a6);
\end{tikzpicture}
\end{center}
By applying a bijection on $[m]$ (which does not affect $\epsilon(\alpha\bullet \beta)$ nor $\epsilon(\alpha) \epsilon(\beta)$), we can assume that this loop uses the numbers $1,\dots,2r$. To transform $\alpha'$ into the standard matching, we don't need to do anything to these edges. To transform $\beta'$ into the standard matching, we can use an even permutation on $\{1,\dots,2r\}$: first we move the leftmost vertex past the other $r-1$ edges and then we swap the orientations on those $r-1$ edges (here is where it is important that we used the flipped ordering on the target to define $i$) for a total sign of $(-1)^{2r-2}=1$. Hence, $\epsilon(\alpha\bullet \beta)^{-1} \epsilon(\alpha) \epsilon(\beta)$ is the same if we remove these edges from $\alpha$ and $\beta$.

Similarly, consider the following local modification to $\alpha$ and $\beta$:
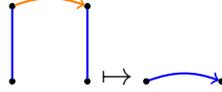
\begin{center}
  \begin{tikzpicture}
    \node [node] at (0,1) (a1) {};
    \node [node] at (1,1) (a2) {};
    \node [node] at (0,0) (b1) {};
    \node [node] at (1,0) (b2) {};
    \draw[thick, blue] (b1) to (a1);
    \draw[thick, blue] (b2) to (a2);
    \draw[thick, orange, ->] (a1) to[out=20,in=160] (a2);
  \end{tikzpicture} $\mapsto$
  \begin{tikzpicture}
    \node [node] at (0,0) (b1) {};
    \node [node] at (1,0) (b2) {};
    \draw[thick, blue, ->] (b1) to[out=20,in=160] (b2);
  \end{tikzpicture}
\end{center}
We can again apply a bijection to $[m]$ to assume that the top 2 vertices are $\{1,2\}$. Then we can apply an even permutation to $\beta'$ so that these two edges become $(1,2)$ and $(3,4)$ in $[n]$ and the remaining edges are shifted over by 2: we first send the second to rightmost vertex in $[m]$ to $3 \in [n]$ and shift everything else in $[m]$ and $\{4,\dots,n\}$ over by 1 and then we send the rightmost vertex in $[m]$ to $2 \in [n]$ and shift everything else in $[m]$ and $\{3,\dots,n\}$ over by 1 for a total sign of $(-1)^{m-2+n-3+m-1+n-2} = 1$. So again we see that $\epsilon(\alpha\bullet  \beta)^{-1} \epsilon(\alpha) \epsilon(\beta)$ remains the same if we make this modification. The same holds for the mirror versions of this modification with respect to the vertical and horizontal axes.

Finally, it is easy to see that removing any horizontal edges from $[n]$ or from $[p]$ does not affect $\epsilon(\alpha\bullet \beta)^{-1} \epsilon(\alpha) \epsilon(\beta)$. We conclude that $\Phi(\alpha) \Phi(\beta) = \Phi(\alpha \beta)$ in general.

It is clear that $\Phi$ fully faithful, and hence is an equivalence of categories. It is also clear that $\Phi$ and its quasi-inverse preserve the upwards and downwards categories, and are thus triangular.
\end{proof}

\begin{remark}
A few remarks related to the proposition:
\begin{itemize}
\item The proposition is closely related to the well-known fact that the Brauer algebra at parameter $\delta=-2n$ acts on tensor powers of the standard representation of $\Sp_{2n}$.
\item The equivalence $\fG \to \fG'$ is monoidal but \emph{not} symmetric monoidal; the same holds for the induced equivalence $\Mod_{\fG} \to \Mod_{\fG'}$. The situation is similar to the equivalence $\Rep(\Sp) \cong \Rep(\bO)$ discussed in \cite[(1.3.3)]{infrank}. In other words, the existence of the signed Brauer category can be viewed as the existence of a non-standard symmetric structure on the tensor product for the usual Brauer category.
\item For a partition $\lambda$, let $\lambda^{\dag}$ denote the transposed partition. Then, under the equivalence $\Mod_{\fG} \cong \Mod_{\fG'}$, the $\fG$-modules $P_{\lambda}$, $\Delta_{\lambda}$, and $L_{\lambda}$ correspond to the $\fG'$-modules $P_{\lambda^{\dag}}$, $\Delta_{\lambda^{\dag}}$, and $L_{\lambda^{\dag}}$ (and similarly for injectives, co-standards, etc.).
\item We will give a more conceptual proof of the proposition in \cite{brauercat2}. \qedhere
\end{itemize}
\end{remark}

\subsubsection{The walled Brauer category}

This category records the obvious maps between mixed tensor powers of the standard representation of the general linear group.
\begin{itemize}
\item We define the walled Brauer category $\fG$ as follows. The objects are 2-colored finite sets; we denote them as pairs $(S_1,S_2)$. A walled Brauer diagram between two 2-colored sets is a Brauer diagram where vertical edges join vertices of the same color, while horizontal edges join vertices of different colors. Composition works just as in the Brauer category.
\item This category appears in \cite[\S 10]{deligne}, \cite[\S 3]{comes}, and \cite[\S 2.2]{coulembier4} (as the ``oriented Brauer category'').
\item The endomorphism rings in $\fG$ are the walled Brauer algebras appearing in \cite{koike2}, \cite{turaev}, and \cite{walledbrauer}.
\item The usual construction endows $\fG$ with a triangular structure. In this case, we have $\End_{\fM}(([n], [m]))=\bC[\fS_n \times \fS_m]$, and so the set of weights is naturally identified with the set of pairs of partitions.
\item Suppose $\delta=p$ is a non-negative integer, and put $V=\bC^p$. We then have a tautological module $T_{p|0}$ defined as follows. The 2-colored set $(S_1,S_2)$ is taken to $V^{\otimes S_1} \otimes (V^*)^{\otimes S_2}$. Maps are defined similarly to the Brauer category case, with the maps $V \otimes V^* \to \bC$ and $\bC \to V \otimes V^*$ taking the place of the maps $a$ and $b$ from \S \ref{ss:brauer-motiv}. More generally, if $\delta$ is an integer and $\delta=p-q$ for non-negative integers $p$ and $q$ then we have a tautological module $T_{p|q}$ defined in the same manner with $V=\bC^{p|q}$.
\item The upwards category is noetherian by \cite[Theorem~1.2]{sym2noeth} (note that $\Sym(\bC\langle 1,1 \rangle)$ is equivalent to $\Mod_{\fU}$ by \cite[\S 3.3.1]{infrank}), and so $\fG$ is noetherian.
\item The category $\fG$ has a transpose functor and monoidal structure, similar to the Brauer category, and the functor $\otimes_{\fD}$ is exact.
\item The category $\fG$ has a triangular involution given by flipping colors, i.e., $(S_1,S_2) \mapsto (S_2,S_1)$. This induces an involution of $\Mod_{\fG}$ that acts on the named modules in the expected manner (e.g., $L_{\lambda,\mu} \mapsto L_{\mu,\lambda}$ and $\Delta_{\lambda,\mu} \mapsto \Delta_{\mu,\lambda}$).
\end{itemize}

\subsubsection{The periplectic Brauer category}

This category records the obvious maps between tensor powers of the standard representation of the periplectic super group. (Recall that a periplectic form on a super vector space $V$ is a symmetric linear form of odd degree, i.e., a linear map $\Sym^2(V) \to \bk[1]$.)
\begin{itemize}
\item We define the periplectic Brauer category $\fG$ as follows. The objects are finite sets. A periplectic Brauer diagram from $S$ to $T$ is a Brauer diagram in which the horizontal edges in $S$ are oriented, and the set of horizontal edges is totally ordered; reversing orientation of an edge in $S$ introduces a sign, as does transposing consecutive horizontal edges in the order. The composition of two morphisms is defined as for Brauer diagrams at parameter $\delta=0$, up to sign issues. See the following references for details.
\item This category was introduced in \cite{kujawa} (under the name ``marked Brauer category''), is discussed in \cite[\S 2.1]{coulembier} and \cite[\S 2.3]{coulembier4} (under the name ``periplectic Brauer category''), and also appears in \cite[Example~1.5(iii)]{brundan} (under the name ``odd Brauer supercategory'').
\item The endomorphism rings in $\fG$ are the periplectic Brauer algebras introduced by Moon \cite{moon} (see also \cite[Proposition 4.1]{JPW}), and studied in papers of Coulembier--Ehrig \cite{coulembier,coulembier2,coulembier3}.
\item The usual construction endows $\fG$ with a triangular structure. We have $\End_{\fM}([n])=\bC[\fS_n]$, and so the set of weights is identified with the set of partitions.
\item Let $V$ be a periplectic space of dimension $p \vert p$. Then we have a tautological $\fG$-module $T_p$ defined by $T_p(S)=V^{\otimes S}$. Morphisms are defined similarly to the Brauer case; see \cite[\S 5]{kujawa} for details.
\item The category of representations of the upwards periplectic Brauer diagram is locally noetherian by \cite[Theorem~1.1]{periplectic}: the upwards category is equivalent to the twisted skew-commutative algebra $\bigwedge(\Sym^2)$ rather than $\Sym(\Sym^2)$ as in the Brauer case because of the sign convention about swapping the order of horizontal edges.
\item There is an equivalence $\fG \cong \widehat{\fG}$ of triangular categories. Essentially, one flips the diagram as in previous cases, and then multiplies by a sign as in the proof of Proposition~\ref{prop:signed-brauer} to account for the discrepancies in which edges are oriented. This equivalence takes a weight $\lambda$ to its transpose $\lambda^{\dag}$ (in the sense of partitions), and is therefore not a transpose functor in the sense of Definition~\ref{defn:transpose}. The category $\fG$ has a monoidal structure, similar to the Brauer category, and the functor $\otimes_{\fD}$ is exact.
\end{itemize}

\subsubsection{The queer walled Brauer category}

This category records the obvious maps between mixed tensor powers of the standard representation of the queer super group. (Recall that the $n$th queer super group is the stabilizer of an odd-degree involution of $\bC^{n\vert n}$.)
\begin{itemize}
\item We define the queer walled Brauer category $\fG$ as follows. The objects are 2-colored sets; we denote them as pairs $(S_1,S_2)$. A queer walled Brauer diagram between two 2-colored sets is a walled Brauer diagram where some of the edges are allowed to have an additional marking. Composition is defined as in the walled Brauer category with $\delta=0$ up to sign issues (which uses the markings on the edges). See \cite[\S 4]{jungkang} for the details of determining this sign in the case of endomorphisms.
\item This category is discussed in \cite[\S 2.4]{coulembier4} under the name ``oriented Brauer--Clifford category.''
\item The endomorphism algebras in $\fG$ were introduced in \cite{jungkang}, under the name ``walled Brauer superalgebras.''
\item The usual construction endows $\fG$ with a triangular structure. The endomorphism algebra $\End_{\fM}(([n], [m]))$ is identified with $\cH_n \otimes \cH_m$, where $\cH_n$ is the $n$th Hecke--Clifford algebra, and so the set of weights is identified with the set of pairs of strict partitions, see \cite[\S 3.3]{chengwang}. (Recall that a partition is {\it strict} if it has no repeated parts.)
\item Let $V=\bC^{p|p}$ be a super vector space equipped with an odd-degree involution $\alpha$. We have a tautological module $T_p$ defined by $T_p(S_1,S_2) = V^{\otimes S_1} \otimes (V^*)^{\otimes S_2}$. The action of morphisms is similar to the walled Brauer case, with the marked edges using $\alpha$; see \cite[\S 3]{jungkang} for details in the case of endomorphisms (the general case being similar).
\item The noetherian property for the upwards category will be proven in \cite{queer}.
\item  The category $\fG$ has a transpose functor and monoidal structure, similar to the walled Brauer category, and the functor $\otimes_{\fD}$ is exact. 
\end{itemize}

\subsubsection{The spin-Brauer category}

This category records the obvious maps between representations of the form $V^{\otimes n} \otimes \Delta$, where $V$ is the standard representation of an orthogonal group and $\Delta$ is the spinor representation of its simply-connected cover. In addition to the pairings $V \otimes V \to \bC$ and $\bC \to V \otimes V$, there are also equivariant maps $V \otimes \Delta \to \Delta$ and $\Delta \to V \otimes \Delta$.
\begin{itemize}
\item We define the spin-Brauer category $\fG$ as follows. Its objects are finite sets. Given finite sets $S$ and $T$, a spin-Brauer diagram from $S$ to $T$ is a triple $(S_{\ast}, T_{\ast}, \alpha)$ where $S_{\ast} \subset S$ and $T_{\ast} \subset T$ are sets of marked vertices, there is a total ordering on $S_\ast \cup T_\ast$, and $\alpha$ is an ordinary Brauer diagram from $S \setminus S_{\ast}$ to $T \setminus T_{\ast}$. Then $\Hom_{\fG}(S,T)$ is the vector space spanned by spin-Brauer diagrams modulo certain linear relations. The composition of spin-Brauer diagrams is in general a complicated linear combination of spin-Brauer diagrams, so we omit the definition. The linear relations and composition rules are defined in detail for endomorphisms in \cite[\S 3]{laudone}; we will discuss the general case in detail in \cite{brauercat2}.
\item The endomorphism algebras in $\fG$ appear in \cite{laudone} and \cite{koike}.
\item We say that a spin-Brauer diagram from $S$ to $T$ is {\bf upwards} (resp.\ {\bf downwards}) if there are no marked vertices or horizontal edges in $S$ (resp.\ $T$). Let $\fU$ (resp.\ $\fD$) be the wide subcategory of $\fG$ where $\Hom_{\fU}(S,T)$ is spanned by upwards (resp.\ downwards) spin-Brauer diagrams. This defines a triangular structure on $\fG$. The categories $\fU$ and $\fD$ were introduced in \cite{spincat}. We have $\End_{\fM}([n])=\bC[\fS_n]$, and so the set of weights is identified with the set of partitions.
\item The category of representations of the upwards category is equivalent to the module category of the 2-step nilpotent twisted Lie superalgebra $\bV \oplus \Sym^2 \bV$ by \cite[Theorem 2.13]{spincat}. The latter is locally noetherian by \cite[\S 8.2]{symsp1}.
\item Suppose $\delta=p$ is a non-negative integer. Let $V$ be the standard representation of the orthogonal Lie algebra $\fso(p)$, and let $\Delta$ be the spinor representation (if $p$ is even, this is the direct sum of its two half spinor representations). There is then a tautological module $T_p$ defined by $T_p(S) = V^{\otimes S} \otimes \Delta$. See \cite[\S\S 4,5]{laudone} for the action of endomorphisms, and \cite[\S 2]{spincat} for the action of $\fU$ or $\fD$. More generally, if $\delta$ is an integer and we have $\delta=p-q$ for non-negative integers $p$ and $q$, with $q$ even, then a similar construction using $\fosp(p|q)$ yields a tautological module $T_{p|q}$. (Note that $T_{p\vert q}(S)$ will be infinite dimensional if $q>0$ though.)
\item The category $\fG$ does not have a symmetric monoidal structure. (Disjoint union provides a monoidal structure, but it is not symmetric due to the orderings on marked vertices.)
\item The category $\fG$ has a transpose functor.
\end{itemize}

\begin{remark}[The oscillator Brauer category]
We can combine the ideas from spin-Brauer category with the signed Brauer category. Here we add the additional data of orientations on the horizontal edges of the Brauer diagrams. All of the above applies to this category, which we call the {\bf oscillator Brauer category} due to the fact that the role of the spinor representation is assumed by the oscillator (or Weil) representation of the symplectic Lie algebra. The upwards category is studied in \cite[\S 3]{spincat}.
\end{remark}

\subsection{The Temperley--Lieb category}

Roughly speaking, the Temperley--Lieb category with parameter $\delta$ is the subcategory of the Brauer category where the diagrams are planar.
\begin{itemize}
\item  Suppose that $S = \{s_1 < \cdots <  s_n\}$ and $T = \{t_1 < \cdots < t_m\}$ are two totally ordered sets. We put a total ordering on $S \amalg T$ by $s_1 < \cdots < s_n < t_m < \cdots < t_1$. A pair of edges  $x < y$ and $z < w$ on the vertex set $S \amalg T$ is said to cross if $x<z<y<w$ or $z<x<w<y$.  We say that a Brauer diagram from $S$ to $T$ is {\bf planar} if no pair of edges cross. We define Temperley--Lieb category $\fG$ over a field $\bk$ as follows: the objects of $\fG$ are totally ordered finite sets, and the set $\Hom_{\fG}(S,T)$ is the free $\bk$-module on planar Brauer diagrams. Composition is carried out just as in the Brauer category (this preserves planarity and makes use of a chosen parameter $\delta \in \bk$).
\item The Temperley--Lieb category appears in \cite{abramsky}, \cite{chen}, and \cite[\S 3.6]{martin2}.
\item The endomorphism algebras are the classical Temperley--Lieb algebras defined in \cite{temperley}.
\item The triangular structure is defined just as for the Brauer category. However, in this case the endomorphism rings in $\fM$ are trivial (since Brauer diagrams representing non-trivial permutations are non-planar) and so we have a triangular structure for any coefficient field $\bk$ (of arbitrary characteristic). The set of weights is identified with $\bN$.
\item The upwards category is noetherian by the forthcoming paper \cite{ncpnoeth}.
\item Pick $q \in \bC \setminus \{0\}$ and set $\delta = -q-q^{-1}$. Let $V$ be the standard $2$-dimensional representation of the quantum group $\cU_q(\fsl_2)$. We have a functor $T$ on $\fG$ defined by $T([n]) = V^{\otimes n}$. Formulas for $\epsilon_1 \colon V^{\otimes 2} \to \bC$ and $\delta_1 \colon \bC \to V^{\otimes 2}$ are given in \cite[(1.14)]{FK}, and the vertical edges of Temperley--Lieb diagrams move tensor factors. These intertwine the $\cU_q(\fsl_2)$ structure defined in \cite[\S 1.3]{FK}, and hence $T$ defines a functor from $\fG$ to the category of $\cU_q(\fsl_2)$-modules.
\end{itemize}

\begin{remark} 
There are a number of variants that one can consider:
\begin{itemize}
\item One can consider Brauer diagrams that are planar when drawn on a cylinder; this is discussed in  \cite[Definition~6.1]{graham} as the ``Jones algebra,'' see also \cite{jones2}.
\item The classical Temperley--Lieb algebra is connected to the Hecke algebra of type~A. There are variants for other types. See \cite[\S 5]{graham2} or \cite{blob} for type B, which is known as the ``blob algebra.''
\item There are affine variants; see \cite{ernst} for type C.
\item One can also consider a variant of the partition category where the diagrams are planar. The endomorphism algebras in this category are the ordinary Temperley--Lieb category \cite{westbury}; we do not know if the category gives something different. \qedhere
\end{itemize}
\end{remark}

\subsection{The degenerate partition category}

This category is defined just like the partition category, but with one modification in how morphisms are composed: if $\alpha$ and $\beta$ are composable diagrams and some part of $\alpha$ meets some part of $\beta$ at $\ge 2$ vertices then $\beta \circ \alpha=0$. The theory developed in \cite{brauercat2} shows that this is a natural category to consider. However, we do not know if it has previously been considered, or if it relates to any natural centralizer algebras. It is triangular, using the same triangular structure as for the partition category, and $\otimes_{\fD}$ is exact.

\subsection{Finite sets}

There are several examples of triangular categories related to the category of sets that have played a prominent role in representation stability.

\subsubsection{The category $\FA$}

Let $\FA$, $\FI$, $\FS$, and $\FB$ denote the categories whose objects are finite sets and whose morphisms are all functions, injections, surjections, and bijections respectively. Let $\fG=\bC[\FA]$ be the linearization of $\FA$. This is a triangular category, with $\fU=\bC[\FI]$ and $\fD=\bC[\FS]$ (and $\fM=\bC[\FB]$). Thanks to the paper \cite{fimodule}, $\FI$-modules have received much attention. The noetherian property for $\FI$ was proved in characteristic~0 in \cite{delta-mod} and \cite{fimodule}, and over general noetherian coefficient rings in \cite{fi-noeth} and \cite{catgb}. Thus $\Mod_{\fG}$ is locally noetherian. In fact, it is also locally artinian: finitely generated $\fG$-modules have finite length. This result, and many others about $\FA$-modules, can be found in the paper \cite{wiltshire}. (It would be interesting to interpret the results of loc.\ cit.\ from the point of view of triangular categories: e.g., what are the standard objects, and what are their multiplicities in indecomposable projectives?)

\subsubsection{The category $\FA^{\op}$}

Unlike the other triangular categories discussed so far, $\fG=\bC[\FA]$ is not self-dual. The dual category $\widehat{\fG}=\bC[\FA^{\op}]$ is in fact far more complicated, and little is known about it at this time. The upwards category $\widehat{\fU}=\bC[\FS^{\op}]$ is locally noetherian by \cite[Theorem~8.1.2]{catgb}, and so $\Mod_{\widehat{\fG}}$ is also locally noetherian.

\subsubsection{The category $\FI\sharp$}

Let $\FI\sharp$ be the category whose objects are finite sets and where a morphism $S \to T$ is a triple $(S_0, T_0, i)$ where $S_0$ is a subset of $S$, $T_0$ is a subset of $T$, and $i \colon S_0 \to T_0$ is a bijection; such triples are often referred to as \emph{partial injections}, as $i$ can be regarded as an injection from a subset of $S$ to $T$. This category was introduced in \cite{fimodule}. Let $\cU$ be the wide subcategory of $\FI\sharp$ consisting of morphisms with $S=S_0$ (this is simply the category $\FI$), and let $\cD$ be the wide subcategory with $T=T_0$. Then $\fG=\bC[\FI\sharp]$ is a triangular category with $\fU=\bC[\cU]$ and $\fD=\bC[\cD]$ (and $\fM=\bC[\FB]$). We have already seen that $\Mod_{\fU}$ is locally noetherian, and so $\Mod_{\fG}$ is as well. In fact, $\Mod_{\fG}$ is semi-simple, as shown in \cite[Theorem 4.1.5]{fimodule}.

\subsection{Finite vector spaces}

Let $\bF$ be a finite field. Let $\VA$, $\VI$, $\VS$, and $\VB$ denote the categories whose objects are finite dimensional $\bF$-vector spaces, and whose morphisms are all linear maps, injective linear maps, surjective linear maps, and linear isomorphisms, respectively. Then $\fG=\bC[\VA]$ is triangular, with $\fU=\bC[\VI]$ and $\fD=\bC[\VS]$ (and $\fM=\bC[\VB]$). We have a transpose functor $\fG \to \widehat{\fG}$ induced by the duality functor $\VA \to \VA^{\op}$. It is known that $\Mod_{\fG}$ is semi-simple \cite[Corollary~1.3]{kuhn}.

\begin{remark}
The representation theory of $\VA$ is far more interesting when the coefficient field has the same characteristic as $\bF$. However, since $\VB$ is not semi-simple in this case, we do not have a triangular category.
\end{remark}

\def\YES{\color{blue!60!white}{Y}}
\def\NO{\color{red!60!white}{N}}
\begin{figure}[!h]
\begin{tabular}{lccc}
Name & Transpose & Monoidal & Tautological module(s)? \\
\hline
Brauer & \YES & \YES & If $\delta \in \bZ$ \\
Signed Brauer & \YES & \YES & If $\delta \in \bZ$ \\
Walled Brauer & \YES & \YES & If $\delta \in \bZ$ \\
Spin Brauer & \YES & \NO & If $\delta \in \bZ$ \\
Periplectic Brauer & \NO & \YES & Yes \\
Queer walled Brauer & \YES & \YES & Yes \\
Partition & \YES & \YES & If $\delta \in \bN$ \\
Temperley--Lieb & \YES & \YES & Yes (for any $\delta$) \\
$\FA$ & \NO & \YES & No \\
$\FA^{\op}$ & \NO & \YES & No \\
$\FI\sharp$ & \YES & \YES & No \\
$\VA$ & \YES & \YES & No
\end{tabular}
\caption{Summary of triangular categories.} \label{fig:tricat}
\end{figure}

\section{Triangular categories from Lie theory} \label{sec:catO}

Let $\fg$ be a complex semisimple Lie algebra. Let $\fb_+$ be a Borel subalgebra and $\fb_-$ the opposite Borel, so that $\fh=\fb_+ \cap \fb_-$ is a Cartan subalgebra, and let $\fn_{\pm}$ be the nilpotent radical of $\fb_{\pm}$. We order the weights in the usual way, so that the roots in $\fn_+$ are positive. Let $\cO$ be the category of $\fg$-modules $M$ satisfying the following conditions:
\begin{enumerate}
\item $M$ is finitely generated;
\item $M$ decomposes into weight spaces under $\fh$; and
\item every element of $M$ is annihilated by some power of $\fn_-$.
\end{enumerate}
This is the famous category introduced by Bernstein--Gel'fand--Gel'fand in \cite{bgg} (see \cite{humphreys} for general background). We caution the reader that our conventions are the opposite of the usual ones in representation theory (e.g., in condition (c) one typically uses $\fn_+$ instead of $\fn_-$). We now explain how to view this category through the lens of triangular categories.

Let $\Lambda$ be a set of weights such that
\begin{enumerate}[(i)]
\item for any $\lambda \in \Lambda$ there are only finitely many $\mu \in \Lambda$ such that $\mu \le \lambda$; and
\item if $\lambda \in \Lambda$ and $\mu$ is a positive integral weight then $\lambda+\mu \in \Lambda$.
\end{enumerate}
Let $\cO(\Lambda)$ be the full subcategory of $\cO$ spanned by objects whose weights are contained in $\Lambda$. For a weight $\lambda$ we let $\fc_{\lambda}$ be the left ideal of $\cU(\fg)$ generated by the elements $X-\lambda(X)$ for $X\in \fh$, together with all elements of weight $\mu$ such that $\mu+\lambda \not\in \Lambda$.

Define a category $\fG=\fG(\Lambda)$ as follows. The objects are the elements of $\Lambda$. For $\lambda, \mu \in \Lambda$, put $\Hom_{\fG}(\lambda,\mu) = (\cU(\fg)/\fc_{\lambda})_{\mu-\lambda}$, where the subscript indicates the specified weight space. Composition in $\fG$ is induced by multiplication in $\cU(\fg)$; it is an easy exercise to verify that this is well-defined. Let $\fU$ be the wide subcategory of $\fG$ where $\Hom_{\fU}(\lambda,\mu)$ is the image of $\cU(\fb_+)_{\mu-\lambda}$ in $\Hom_{\fG}(\lambda,\mu)$; define $\fD$ analogously using $\fb_-$ instead.

\begin{proposition}
With the above definitions, $\fG$ is a triangular category.
\end{proposition}

\begin{proof}
We verify the axioms:
\begin{itemize}
\item[(T0)] The category $\fG$ is small by definition. Let $\lambda \in \Lambda$. If $Y \in \fn_-$ then $Y\lambda < \lambda$. It follows that some power of $\fn_-$ will carry $\Lambda$ outside of $\Lambda$, and so $\fc_{\lambda}$ contains $\fn_-^N$ for some $N$. By PBW (Poincar\'e--Birkhoff--Witt), it follows that $\cU(\fg)/\fc_{\lambda}$ has a basis consisting of elements of the form $X_1 \cdots X_r Y_1 \cdots Y_s$ where $X_i \in \fn_+$ and $Y_i \in \fn_-$ and $s<N$. It is thus clear that the weight spaces of $\cU(\fg)/\fc_{\lambda}$ are finite dimensional, and so the $\Hom$ spaces in $\fG$ are too.
\item[(T1)] By definition, $\End_{\fU}(\lambda)$ is the image of the zero weight space of $\cU(\fb_+)$ in $\End_{\fG}(\lambda)$. The zero weight space of $\cU(\fb_+)$ is $\cU(\fh)$, and every element of $\fh$ maps to a scalar in $\End_{\fG}(\lambda)$. We thus have $\End_{\fU}(\lambda)=\bC$. The same analysis applies to $\fD$. Thus $\End_{\fU}(\lambda)=\End_{\fD}(\lambda)$, and this ring is semi-simple.
\item[(T2)] It is clear that the order $\le$ on $\Lambda$ is admissible.
\item[(T3)] Let $\lambda,\mu \in \Lambda$. We must show that the natural map
\begin{displaymath}
\bigoplus_{\nu \in \Lambda} \Hom_{\fU}(\nu,\mu) \otimes_{\bC} \Hom_{\fD}(\lambda,\nu) \to \Hom_{\fG}(\lambda,\mu)
\end{displaymath}
is an isomorphism. As stated above, PBW tells us that $\Hom_\fG(\lambda, \mu)$ has a basis consisting of (certain) elements of the form $X_1\cdots X_r Y_1\cdots Y_s$ where the $X_i \in \fn_+$ and $Y_j \in \fn_-$ are chosen among some fixed basis where the total weight is $\lambda-\mu$. Let $\lambda-\nu$ be the weight of $Y_1 \cdots Y_s$. Then $\nu-\mu$ is the weight of $X_1 \cdots X_r$. Furthermore, $\nu-\mu$ is a positive integral weight and hence $\nu = (\nu-\mu) + \mu \in \Lambda$ by axiom (ii). So the natural map above is surjective. That it is injective also follows by appealing to PBW for each of $\Hom_\fU(\nu,\mu)$ and $\Hom_\fD(\lambda,\nu)$: note that $\fc_{\lambda}$ is generated by $\fc_{\lambda} \cap \cU(\fn_-)$, so the $Y_1 \cdots Y_s$ elements from our basis of weight $\lambda-\nu$ form a basis of $\Hom_{\fD}(\lambda, \nu)$. \qedhere
\end{itemize}
\end{proof}

One easily sees that there is an equivalence $\Phi \colon \Mod_{\fG(\Lambda)}^{\rm fg} \to \cO(\Lambda)$ given by $\Phi(M)=\bigoplus_{\lambda \in \Lambda} M_{\lambda}$. We can therefore apply the formalism of triangular categories to category $\cO$, as we now explain. By our proof of (T1) we have $\End_{\fM}(\lambda)=\bC$, and so the set of weights of $\fG(\Lambda)$ (in the sense of \S \ref{ss:weights}) is simply $\Lambda$. For $\lambda \in \Lambda$, the simple module $S_{\lambda}$ is given by $S_{\lambda}(\lambda)=\bC$ and $S_{\lambda}(\mu)=0$ for $\lambda \ne \mu$. In what follows, we use a $\Lambda$ superscript to denote the usual $\fG(\Lambda)$-modules, e.g., $L_{\lambda}^{\Lambda}$ is the simple $\fG(\Lambda)$-module corresponding to $\lambda$.

One easily sees that $\Phi(\Delta_{\lambda}^{\Lambda})$ is the Verma module associated to $\lambda$ in $\cO$; thus $\Delta_{\lambda}^{\Lambda}$ is essentially independent of $\Lambda$. From this, we see that the same holds for the simple modules.

The projective objects $\wt{P}_{\lambda}^{\Lambda}$ and $P_{\lambda}^{\Lambda}$ do not enjoy the same independence in general. However, there is one important case where it does hold. Suppose that $\Lambda$ contains the shifted Weyl orbit of $\lambda$. Then $\cO(\Lambda)$ contains the entire block $\cA$ of $\cO$ associated to $\lambda$ \cite[\S 1.13]{humphreys}, and so we have $\cO(\Lambda)=\cA \oplus \cA'$ for some complementary category $\cA'$. From the equivalence $\Phi$, we see that $\Mod_{\fG(\Lambda)}$ decomposes as $\cB \oplus \cB'$ where $\cB$ is the block containing $L_{\lambda}^{\Lambda}$ and $\cB'$ is a complementary category. Since $P^{\Lambda}_{\lambda}$ is the projective cover of $L_{\lambda}^{\Lambda}$ in $\Mod_{\fG(\Lambda)}$, it must belong to $\cB$. Thus $\Phi(P_{\lambda}^{\Lambda})$ is the projective cover of $\Phi(L_{\lambda}^{\Lambda})$ in $\cA$, and therefore in all of $\cO$ since $\cA$ is a block. In particular, $P_{\lambda}^{\Lambda}$ is independent of $\Lambda$ when $\Lambda$ contains the shifted orbit of $\lambda$. (The projective $\wt{P}^{\Lambda}_{\lambda}$ is basically never independent of $\Lambda$: it bleeds into more blocks as $\Lambda$ is enlarged.)

From the above, we conclude that $\cO$ has enough projectives and that the projectives have standard filtrations (apply $\Phi$ to the standard filtration of $P_{\lambda}^{\Lambda}$ with $\Lambda$ sufficiently large). Furthermore, our version of BGG reciprocity (Proposition~\ref{prop:bgg}) recovers the classical one (we note that $\fG(\Lambda)$ does admit a transpose).

We thus see that the triangular formalism recovers the most basic properties of $\cO$, taking only the block decomposition as input. The proofs given by this method are nearly identical to the original ones from \cite{bgg}.

\begin{remark} \label{rmk:algrep}
  One can also apply the triangular formalism to algebraic representations in positive characteristic. The basic idea is similar to the above, except we use a finite set of weights, and the universal enveloping algebra is replaced by the hyperalgebra (or algebra of distributions, see \cite[Chapter~I.7]{jantzen}). Verma's conjecture, proved by Sullivan \cite{sullivan}, allows one to connect modules over the hyperalgebra (and thus this category) to algebraic representations.
\end{remark}

\end{document}